\documentclass{amsart}
\usepackage[foot]{amsaddr}

\usepackage{graphicx}

\usepackage{multicol}
\usepackage[bottom]{footmisc}
\usepackage{bm}
\usepackage{cite}
\usepackage{url}
\usepackage{cancel}
\usepackage[margin=3cm]{geometry}
\usepackage{lipsum}

\usepackage{float}
\usepackage{subfig}

\newcommand{\pa}{\partial}
\newcommand{\lt}{\left}
\newcommand{\rt}{\right}
\newcommand{\R}{\mathbb{R}}

\usepackage{mathtools}
\usepackage{graphicx}
\usepackage{multirow}
\usepackage[mathscr]{euscript}
\usepackage{enumitem}
\usepackage{changepage}
\usepackage{bbm}
\usepackage{mathrsfs}
\newtheorem{theorem}{Theorem}[section]

\theoremstyle{definition}
\newtheorem{definition}[theorem]{Definition}

\theoremstyle{remark}
\newtheorem{remark}[theorem]{Remark}

\usepackage{color}

\newtheoremstyle{examplestyle}
   {}{}{}{}{\bfseries}{.}{.5em}{{\thmname{#1 }}{\thmnumber{#2}}{\thmnote{ (#3)}}}
\theoremstyle{examplestyle}
\newtheorem{examplecase}{Example}[section]

\makeindex             

\begin{document}

\title[  ]{High-order well-balanced finite volume schemes for hydrodynamic equations with nonlocal free energy}

\author{Jos\'e A. Carrillo}
\address[Jos\'e A. Carrillo]{Mathematical Institute, University of Oxford, Oxford OX2 6GG, UK}
\curraddr{}
\email{carrillo@maths.ox.ac.uk}
\thanks{}

\author{Manuel J. Castro}
\address[Manuel J. Castro]{Dpto. An\'alisis Matem\'atico, Estad\'istica e Investigaci\'on Operativa y Matem\'atica Aplicada. Universidad de M\'alaga. \\
Bulevar Louis Pasteur, 31, 29010 M\'alaga, Spain}
\curraddr{}
\email{mjcastro@uma.es}
\thanks{}

\author{Serafim Kalliadasis}
\address[Serafim Kalliadasis]{Department of Chemical Engineering, Imperial College London, SW7 2AZ, UK}
\curraddr{}
\email{s.kalliadasis@imperial.ac.uk}
\thanks{}

\author{Sergio P. Perez}
\address[Sergio P. Perez]{Departments of Chemical Engineering and Mathematics, Imperial College London, SW7 2AZ, UK}
\curraddr{}
\email{sergio.perez15@imperial.ac.uk}
\thanks{}

\begin{abstract}
We propose high-order well-balanced finite-volume schemes for a broad class
of hydrodynamic systems with attractive-repulsive interaction forces and
linear and nonlinear damping. Our schemes are suitable for free energies
containing convolutions of an interaction potential with the density, which
are essential for applications such as the Keller-Segel model, more general
Euler-Poisson systems, or dynamic-density functional theory. Our schemes
are also equipped with a nonnegative-density reconstruction which allows
for vacuum regions during the simulation. We provide several prototypical
examples from relevant applications highlighting the benefit of our
algorithms elucidate also some of our analytical results.
\end{abstract}

\maketitle


%
%
%
%

\section{Introduction}\label{sec:intro}

Well-balanced schemes have emerged as a paramount tool to simulate systems
governed by balance/conservation laws. This is due to their ability to
numerically preserve steady states and resolve small perturbations of those
states even with coarse meshes. Well-balanced schemes were introduced nearly
three decades ago, with the initial works by Berm\'udez and V\'azquez
\cite{bermudez1994upwind}, Greenberg and Leroux \cite{greenberg1996well} and
Gosse \cite{gosse2000well}. One of the most popular applications for
well-balanced schemes from the beginning has been the shallow-water
equations. Further contributions of note are the hydrostatic reconstruction
in \cite{audusse2004fast,castro2007well} together with application scenarios
where well-balanced schemes have proven quite successful: tsunami propagation
\cite{castro2019third}, coastal hydrodynamics \cite{marche2007evaluation} and
irregular topographies \cite{gallardo2007well}, to name but a few. Inspired
by the strong results for the shallow-water equations, plenty of authors have
successfully employed well-balanced schemes in a plethora of balance-law
problems from wave propagation in elastic media \cite{xing2006high} and
chemosensitive movement of cells \cite{filbet2005approximation} to flow
through a nozzle \cite{gascon2001construction} and the Euler equations with
gravity \cite{klingenberg2019arbitrary,thomann2019second}. Recently, a general procedure to construct high-order well-balanced schemes for 1D balance laws is described in \cite{castro2020}.

In our previous work \cite{carrillo2018wellbalanced} we extended the
applicability of well-balanced schemes to the broad class of hydrodynamic
models with attractive-repulsive interaction forces. In particular we
considered interactions associated with nonlocal convolutions or functions of
convolutions, which is commonplace in applications such as the Keller-Segel
model \cite{calvez2017equilibria}, more general Euler-Poisson systems
\cite{hadvzic2019class} or in dynamic-density functional theory (DDFT)
\cite{goddard2012general,goddard2012unification}. This class of balance laws
may contain linear or nonlinear damping effects, such as the Cucker-Smale
alignment term in collective behaviour \cite{cucker2007emergent}. The
corresponding hydrodynamic systems have the general form
\begin{equation}\label{eq:generalsys}
\begin{cases}
\partial_{t}\rho+\nabla\cdot\left(\rho \bm{u}\right)=0,\quad \bm{x}\in\R^{d},\quad t>0,\\[2mm]
{\displaystyle \pa_{t}(\rho \bm{u})\!+\!\nabla\!\cdot\!(\rho \bm{u}\otimes \bm{u})\!=-\rho \nabla \frac{\delta \mathcal{F}[\rho]}{\delta \rho}-\gamma\rho \bm{u}-\!\rho\!\!\int_{\R^{d}}\!\!\psi(\bm{x}-\bm{y})(\bm{u}(\bm{x})-\bm{u}(\bm{y}))\rho(\bm{y})\,d\bm{y}
,}
\end{cases}
\end{equation}
where the free energy functional $\mathcal{F}[\rho]$ contains the pressure
$P(\rho)$ and generic potential terms $H(\bm{x},\rho)$ and can be decomposed
as
\begin{equation*}
-\rho \nabla \frac{\delta \mathcal{F}}{\delta \rho}=-\nabla P(\rho)-\rho \nabla H(\bm{x},\rho).
\end{equation*}
The potential terms $H(\bm{x},\rho)$ involve an external field $V(x)$ and an interaction potential $W(x)$ convoluted with the density $\rho$, so that
\begin{equation*}\label{eq:pot}
H(\bm{x},\rho)=V(\bm{x})+W(\bm{x})\star \rho.
\end{equation*}
Finally, the free-energy functional has the form
\begin{equation}\label{eq:freeenergy}
  \mathcal{F}[\rho]=\int_{\R^{d}}\Pi(\rho)d\bm{x}+\int_{\R^{d}}V(\bm{x})\rho(\bm{x})d\bm{x}+\frac{1}{2}\int_{\R^{d}}\int_{\R^{d}}W(\bm{x}-\bm{y})\rho(\bm{x})\rho(\bm{y}) d\bm{x} d\bm{y},
\end{equation}
where $\rho \Pi''(\rho)=P'(\rho)$.
The steady states of the system \eqref{eq:generalsys}, whose preservation at the discrete level is the main aim of the design of well-balanced schemes, are characterized by
\begin{equation}\label{eq:steadyvarener}
\frac{\delta \mathcal{F}}{\delta \rho}=\Pi'(\rho)+H(\bm{x},\rho)= \text{constant on each connected component of}\ \mathrm{supp}(\rho)\ \text{and}\ u=0,
\end{equation}
where the constant can vary on different connected components of
$\mathrm{supp}(\rho)$. These steady states are stationary ($u=0$) due to the dissipation of the linear damping $-\gamma\rho u$ or nonlinear damping in the system \eqref{eq:generalsys}, which ensures that the momentum eventually vanishes. This is due to the fact that the total energy of the system, defined as the sum of kinetic energy and free energy,
\begin{equation}\label{eq:totalenergy}
  E(\rho,\bm{u})=\int_{\R^{d}}\frac{1}{2}\rho \left|\bm{u}\right|^2 d\bm{x}+\mathcal{F}(\rho),
\end{equation}
is formally dissipated (see \cite{giesselmann2017relative,carrillo2017weak,carrillo2018longtime}) as
\begin{equation}\label{eq:equalenergy}
  \frac{dE(\rho,\bm{u})}{dt}=-\gamma\int_{\R^{d}}\rho \left|\bm{u}\right|^2 d\bm{x}-\!\int_{\R^{d}}\int_{\R^{d}}\!\!\psi(\bm{x}-\bm{y})\left|\bm{u}(\bm{y})-\bm{u}(\bm{x})\right|^2 \rho(\bm{x})\,\rho(\bm{y})\,d\bm{x}\,d\bm{y}.
\end{equation}

Furthermore, the system \eqref{eq:generalsys} also satisfies an entropy identity
\begin{equation}\label{eq:entrineq}
 \pa_{t} \eta(\rho, \rho \textbf{u})+ \nabla \cdot \bm{G}(\rho, \rho \textbf{u})= -\rho \bm{u}\cdot \nabla H(\bm{x},\rho)-\gamma\rho \left|\bm{u}\right|^2-\rho \int_{\R^{d}}\!\!\psi(\bm{x}-\bm{y})\,\bm{u}(\bm{x}) \cdot (\bm{u}(\bm{x})-\bm{u}(\bm{y}))\rho(\bm{y})\,d\bm{y},
\end{equation}
where $\eta(\rho,\rho\bm{u})$ and $\bm{G}(\rho,\rho\bm{u})$ are the entropy and the entropy flux defined as
\begin{equation*}
\eta(\rho, \rho \bm{u})=\rho \frac{\left|\bm{u}\right|^2}{2}+\Pi (\rho), \quad \bm{G} (\rho, \rho \bm{u})=\rho \bm{u}\left(\frac{\left|\bm{u}\right|^2}{2}+\Pi' (\rho) \right).
\end{equation*}

In addition to the well-balanced property, many authors have sought to
construct numerical schemes that preserve the structural properties of the
system \eqref{eq:generalsys} during its temporal evolution. These endeavours
have aimed to first satisfy discretely the entropy identity
\eqref{eq:entrineq} and second the dissipation relations for the total energy
in \eqref{eq:equalenergy}, both for the original system \eqref{eq:generalsys}
and its overdamped versions. We refer the reader to Refs.
\cite{tadmor2016entropy,fjordholm2012arbitrarily,audusse2015very} for more information about
entropy stable schemes and to Refs. \cite{bailo1811fully,carrillo2015finite}
for insights details and useful insights energy dissipating schemes. The
well-balanced finite-volume scheme of our previous work
\cite{carrillo2018wellbalanced} was designed to be at the same time
well-balanced, entropy stable and energy dissipating, though only for first-
and second-order accuracy.

The main contribution of the present work is to extend our previous scheme
for the system \eqref{eq:generalsys} from first and second order to high
order. Several authors have already proposed high-order well-balanced schemes
for systems where the potential terms in the free energy
\eqref{eq:freeenergy} are local, such as the shallow-water equations
\cite{castro13,castro2017well,noelle2006well,chertock2018well,chalons2018high}, chemotaxis
\cite{filbet2005approximation} and other applications \cite{xing2006high}.
These schemes rely on finite differences, finite volumes or discontinuous
Galerkin approaches, and specially for the shallow-water case there have been
plenty of contributions devoted to particular configurations and scenarios:
presence of dry areas and bottom topography \cite{gallardo2007well}, tsunami
propagation in 2D meshes \cite{castro2019third}, traffic flow model \cite{chalons2018high}, moving steady states
\cite{noelle2007high,castro13,cheng2019new}, etc.

Here we consider the much broader class of free energies in
\eqref{eq:freeenergy}, which include interaction potentials leading to forces
given by convolution with the density $\rho$ and possible linear or nonlinear
damping effects from the field of collective behaviour. Applications of this
type include the Keller-Segel model, generalized Euler-Poisson systems
\cite{hadvzic2019class} and DDFT
\cite{goddard2012general,goddard2012unification}. We complement our
high-order finite volume schemes with the desired properties of
well-balancing and the nonnegativity of the density, which allows for vacuum
regions in the simulations. Our work lays the foundations for the
construction of well-balanced high-order schemes that may satisfy further
fundamental properties of the system \eqref{eq:generalsys}, such as the
discrete versions of the energy dissipation in \eqref{eq:equalenergy} and
entropy identity in \eqref{eq:entrineq}, or even the well-balanced property
for the challenging moving steady states. Developing schemes, amongst the
class of positivity-preserving high-order schemes introduced in the present
work, satisfying also the entropy stability and energy dissipating
properties, is a challenging open question.

The paper is structured as follows. First, in Section \ref{sec:numsch} we
explain the construction of our well-balanced high-order finite volume
scheme. In Subsection \ref{subsec:first_order} we begin by recalling our first-order numerical scheme from \cite{carrillo2018wellbalanced}, and then in Subsection \ref{subsec:high_order} we provide an first-attempt extension of such scheme to high order. The correct well-balanced formulation for that high-order scheme is provided in Subsection \ref{subsec:high_WB}. The last Subsection \ref{subsec:algorithm} contains the summarized algorithmic implementation of the scheme. Second, in Section \ref{sec:numsim} we depict
a battery of simulations for relevant applications of system
\eqref{eq:generalsys}. In Subsection \ref{subsec:validation} we numerically check the well-balanced property and high-order accuracy of our scheme,
and subsequently in Subsection \ref{subsec:numexperiments} we tackle
applications for varied choices of the free energy, leading to interesting
numerical experiments for which analytical results are limited in the
literature.

%
%
%
%
\section{High-order well-balanced finite volume scheme}\label{sec:numsch}
The different terms of the one-dimensional system \eqref{eq:generalsys} are
usually gathered in the form of
\begin{equation}\label{eq:compactsys}
\pa_t U + \pa_x F(U) = S_H(U,H)+S_D(x,U),
\end{equation}
with
\begin{equation*}
U=\begin{pmatrix}
\rho \\ \rho u
\end{pmatrix}, \quad
F(U)=\begin{pmatrix}
\rho u \\ \rho u^2+P(\rho)
\end{pmatrix}
\end{equation*}
and
\begin{equation*}
S_H(U,H)=\begin{pmatrix}
0 \\ -\rho \pa_x H
\end{pmatrix}, \quad
S_D(x,U)=\begin{pmatrix}
0 \\ -\gamma\rho u-\rho \displaystyle\int_{\R}\psi(x-y)(u(x)-u(y))\rho(y)\,dy
\end{pmatrix},
\end{equation*}
where $U$ are the unknown variables, $F(U)$ the fluxes, and $S_H(U,H)$ and $S_D(x,U)$ are the
sources related to forces with potential $H$ and damping terms respectively.  In what follows, we only consider the source term $S_H(U,H)$ due to the forces as we focus on the definition of a well-balanced high-order scheme for stationary solutions \eqref{eq:steadyvarener}. At the end of this Section we propose a high-order discretization of the source damping term $S_D(x,U)$ that vanishes at stationary states.

We consider a mesh composed by cells $[x_{i-1/2}, x_{i+1/2}]$, $1 \leq i \leq N$, whose length $\Delta x$ is supposed to be constant for simplicity.  Let us denote by $U_i(t)$ the approximation of the average of the exact solution at the $i$th cell, $[x_{i-1/2}, x_{i+1/2}]$ at time $t$,
\begin{equation}\label{eq:cellU}
U_i(t)=\sum_{j=1}^{n_s} \alpha_j U(x_i^j,t) \cong \frac{1}{\Delta x} \int_{x_{i-1/2}}^{x_{i+1/2}} U(x,t) \, dx,
\end{equation}
and we denote by $H_i(t)$ the approximation of the cell average of $H(x,\rho)=V(x)+W(x)\star \rho$ at time $t$,
\begin{equation*}
H_i(t) =\sum_{j=1}^{n_s} \alpha_j H(x_i^j, \rho(x_i^j, t)) \cong \frac{1}{\Delta x} \int_{x_{i-1/2}}^{x_{i+1/2}} H(x,\rho(x,t)) \, dx.
\end{equation*}
In the previous expressions we denote as $\alpha_j$ and $x_i^j$, for $j=1,\cdots, n_s$, the weights and quadrature points of a particular high-order quadrature formula for the cell $[x_{i-1/2}, x_{i+1/2}]$. In this work we employ the fifth-order standard Gaussian quadrature described in appendix \ref{app:gauss}.

As pointed out in the introduction, one of the main contributions of this work is to construct high-order well-balanced schemes for free energies that may depend on the convolution of the density and an interaction potential, $W(x)\star \rho$. These convolutions are included in the steady state relations in \eqref{eq:steadyvarener}, but for the discrete version of these relations one has to approximate the convolutions by a high-order quadrature formula. In the next definition we clarify the concept of well-balanced scheme for this kind of free energies.

\begin{definition}[Well-balanced scheme]\label{def:WB} We consider a semi-discrete method to approximate \eqref{eq:compactsys},
\begin{equation}\label{eq:wb}
\left\{
\begin{array}{ll}
    \displaystyle{\frac{dU_i}{dt}}=&\displaystyle{-\frac{1}{\Delta x} \mathcal{H}(\Delta x, U_j(t), H_j(t), j \in \mathcal{S}_i)}, \\[0.2cm]
    U(0)=&U_0,
\end{array}
\right.
\end{equation}
where $U(t)=\{U_i(t) \}_{i=1}^N$ represents the vector of the approximations of the averaged values of the exact solutions at time $t$, $U_0=\{U_i(0)\}$ is the vector of  the initial conditions, and $\mathcal{S}_i$ the stencil of the numerical scheme.

Now let us assume that $u(x)=0$, $\rho(x)$ is a smooth function and $H_{\Delta x}(x)$ is a discrete approximation of $H=V+W\star\rho$ with the form
\[
H_{\Delta x}(x)=V(x)+\Delta x\sum_{l=1}^M\sum_{m=1}^{n_s}\alpha_m W(x-x_l^m) \rho(x_l^m)
\]
and satisfying
\begin{equation}\label{eq:discretewb}
    \Pi'(\rho(x))+H_{\Delta x}(x)=C_\Gamma \mbox{ in each } \Lambda_\Gamma, \Gamma \in \mathbb{N},
\end{equation}
where $\Lambda_\Gamma$, $\Gamma\in\mathbb{N}$, denotes the possible infinite
sequence indexed by $\Gamma$ of subsets $\Lambda_\Gamma$ of subsequent
indices $i\in\left\{1,\ldots,N\right\}$ where $\rho(x)>0$ and $u=0$, and
$C_\Gamma$ the corresponding constant in that connected component of the
discrete support.

Then it follows that the semi-discrete numerical scheme \eqref{eq:wb} is said to be well-balanced for
\[
U=\left(
\begin{array}{c}
\rho(x) \\
0
\end{array}
\right) \quad \text{and} \quad H_{\Delta x}
\]
if the vector of their approximated averages is a critical point of \eqref{eq:wb}, i.e.
\[
\mathcal{H}(\Delta x, U_j, H_j, j\in \mathcal{S}_i)=0, \ 1 \leq i \leq N,
\]
where
\[
U_i=\sum_{j=1}^{n_s} \alpha_j U(x_i^j) \quad \text{and} \quad H_i=\sum_{j=1}^{n_s} \alpha_j H_{\Delta x} (x_i^j).
\]

\end{definition}

In what follows we begin by briefly recalling in Subsection \ref{subsec:first_order} the first-order well-balanced scheme for \eqref{eq:compactsys} introduced in \cite{carrillo2018wellbalanced}, which serves as an starting point to construct high-order schemes by employing a high-order reconstruction operator, as described in Subsection \ref{subsec:high_order}. Then in Subsection \ref{subsec:high_WB} we describe how to adapt these high-order schemes so that they are well-balanced in the sense defined in \ref{def:WB}.

%
%

\subsection{First-order numerical scheme}\label{subsec:first_order}
The first-order semi-discrete well-balanced finite-volume scheme  for system \eqref{eq:compactsys} introduced in \cite{carrillo2018wellbalanced} can be written as
\begin{equation}\label{eq:1orderwb}
    \frac{d U_i}{dt}= -\frac{1}{\Delta t} \left(\mathbb{F}_{i+1/2}^- - \mathbb{F}_{i-1/2}^+ \right),
\end{equation}
where $\mathbb{F}_{i+1/2}^\pm$ is defined using a standard consistent numerical flux for the homogeneous system applied to the so-called hydrostatic-reconstructed  states, with an extra term ensuring the consistency of the numerical scheme \eqref{eq:1orderwb} applied to system \eqref{eq:compactsys}, as well as its well-balanced character. In \cite{carrillo2018wellbalanced} the midpoint quadrature formula is used to approximate both the cell-averages of the exact solution and $H_{\Delta x}(x)$. In what follows we suppress the time-dependence in the cell averages for simplicity. We define $\mathbb{F}^\pm_{i+1/2}$ in terms of the cell averages $U_i$, $U_{i+1}$, $H_i$ and $H_{i+1}$ as
\begin{equation} \label{eq:numflux}
    \mathbb{F}_{i+1/2}^\pm( U_i,U_{i+1}, H_i, H_{i+1})=\mathcal{F}(U_{i+1/2}^{HR,-}, U_{i+1/2}^{HR,+}) \pm S_{i+1/2}^{HR,\pm},
\end{equation}
where $\mathcal{F}(U,V)$ is the standard local Lax-Friedrich  numerical flux for the homogeneous system,
\begin{equation}\label{eq:LF}
    \mathcal{F}(U,V)=\frac{1}{2}\lt(F(U)+F(V) - |\lambda(U,V)| (V-U)\rt),
\end{equation}
where $|\lambda(U,V)|$ is a bound for the maximum absolute value of the wave speeds for the Riemann problem with the constant states $U$ and $V$.

In \cite{carrillo2018wellbalanced} we employ the hydrostatic reconstruction firstly introduced in \cite{audusse2004fast} in the context of shallow-water equations. Here we denote the hydrostatic-reconstructed states as $U_{i+1/2}^{HR,\pm}$, and we compute them as follows:
\begin{enumerate}[label=\arabic*)]
    \item Firstly an intermediate state $H_{i+1/2}$ is computed as
    \[
    H_{i+1/2}=\max(H_{i+1}, H_i) .
    \]
    \item Next, we define the hydrostatic-reconstructed states as
    \[
    U_{i+1/2}^{HR,\pm}=\left(
    \begin{array}{c}
    \rho_{i+1/2}^{HR,\pm} \\[0.2cm]
    (\rho u)_{i+1/2}^{HR,\pm}
    \end{array}
    \right),
    \]
    where
\begin{equation}\label{eq:rhointerface}
\begin{gathered}
 \rho_{i+1/2}^{HR,-}=\xi \left(\Pi'\left(\rho_i\right)+H_i-H_{i+1/2}\right),\quad \lt(\rho u\rt)_{i+1/2}^{HR,-}=\rho_{i+1/2}^{HR,-}u_{i},\\[0.2cm]
  \rho_{i+1/2}^{HR,+}=\xi \left(\Pi'\left(\rho_{i+1}\right)+H_{i+1}-H_{i+1/2}\right),\quad \lt( \rho u\rt)_{i+1/2}^{HR,+}=\rho_{i+1/2}^{HR,+} u_{i+1},
\end{gathered}
\end{equation}
with $\xi(s)$ being the inverse function of $\Pi'(s)$ for $s>0$ and $u_i=(\rho u)_i/\rho_i$.
\end{enumerate}

The last ingredients for the flux in \eqref{eq:numflux} are the terms $S_{i+1/2}^{HR,\pm}$, which correspond to the correction introduced in the numerical scheme to guarantee consistency and well-balanced properties (see \cite{audusse2004fast,carrillo2018wellbalanced}),
\begin{equation}\label{eq:fluxsource}
S_{i+1/2}^{HR,+}=\begin{pmatrix}
0 \\ P\left(\rho_{i+1}\right)-P\left(\rho_{i+1/2}^{HR,+}\right)\end{pmatrix} \quad \text{and}\quad S_{i+1/2}^{HR,-}=\begin{pmatrix}
0 \\ P\left(\rho_{i+1/2}^{HR,-}\right)-P\left(\rho_i\right)\end{pmatrix}.
\end{equation}
It is straightforward to check that the semi-discrete numerical scheme \eqref{eq:1orderwb}-\eqref{eq:fluxsource} is well-balanced in the sense defined in definition \ref{def:WB} (see \cite{carrillo2018wellbalanced}).

%
%

\subsection{High-order extension}\label{subsec:high_order} The basic ingredients to design a high-order finite volume method for system \eqref{eq:compactsys}, assuming $S_D(x,U)=0$, are:

\begin{itemize}

\item a consistent first order numerical flux for system \eqref{eq:compactsys}, like the one proposed in \cite{carrillo2018wellbalanced} and described in the previous Subsection;

\item a high-order reconstruction operator, i.e. an operator that, given a family of cell values $\{U_i(t) \}$, provides at every cell $[x_{i-1/2}, x_{i+1/2}]$ a smooth function
that depends on the values at some neighbor cells whose indexes belong to the so-called stencil $\mathcal{S}_i$:
$$
R^U_i(x) = R^U_i (x; \{ U_j(t) \}_{j \in \mathcal{S}_i}),
$$
so that $R^U_i(x)$ is a high-order approximation of $U(x,t)$ in the $i$th cell at time $t$. Here we use  third-
and fifth-order CWENO reconstruction operators \cite{levy1999central,
levy2000compact, capdeville2008central}. The main advantage of CWENO compared to WENO (see \cite{Shu88,shuosh,Shu97}) reconstruction operators is that CWENO reconstructions achieve uniform high-order approximation in the entire cell, while WENO  reconstruction operators are proposed to achieve high-order approximation at the boundaries of the cell. Thus, standard WENO-5 reconstructions achieves $5$th-order at the boundaries of the cell, while it is only $3$rd-order at the interior points. Therefore, CWENO reconstruction operators are specially useful  in balance laws such as \eqref{eq:compactsys}, where the source term has to be evaluated at inner points of the cell. We complement
the CWENO reconstruction operators with the positive-density limiters from
\cite{zhang2010maximum} to ensure physical admissible reconstructed values for the density. For further details we refer the reader to appendix \ref{app:CWENO}.
\end{itemize}

Using these ingredients, one could consider a high-order finite-volume semi-discrete numerical method of the form:
\begin{equation}\label{eq:ho1}
\frac{d U_i}{dt} = -\frac{1}{\Delta x} \left( \mathbb{F}^-_{i+1/2} - \mathbb{F}^+_{i - 1/2}\right) +\frac{1}{\Delta x} \int_{x_{i-1/2}}^{x_{i+1/2}} S_H(R_i^U(x),R_i^H(x)) \, dx,
\end{equation}
where
\begin{itemize}
\item $R_U^t(x)$ and $R_i^H(x)$ are  the approximations of the solution $U(x,t)$ and the function $H_{\Delta x}(x)$, respectively, at the $i$th cell given by some high-order reconstruction operators from the sequence of cell values $\{U_i(t)\}$ and $\{H_i(t)\}$, respectively, i.e.
$$
R_i^U(x) = R_i (x; \{ U_j(t) \}_{j \in \mathcal{S}_i} ) \quad \text{and} \quad R_i^H(x)=R_i(x; \{H_j(t)\}_{j \in \mathcal{S}_i});
$$

\item $\mathbb{F}^\pm_{i+1/2}$ is the numerical flux defined in \eqref{eq:numflux} applied to the reconstructed states $U_{i+1/2}^{\mp}$ and $H_{i+1/2}^{\mp}$, i.e.
$$
\mathbb{F}^\pm_{i+1/2}=\mathbb{F}(U_{i+1/2}^{-}, U_{i+1/2}^{+}, H_{i+1/2}^{-}, H_{i+1/2}^{+})
$$
with
$$
U_{i+1/2}^{-} = R_i^U(x_{i+1/2}), \quad U_{i+1/2}^{+} = R_{i+1}^U(x_{i+1/2}),
$$
and
$$
H_{i+1/2}^{-}=R_i^H(x_{i+1/2}), \quad H_{i+1/2}^{+}=R_{i+1}^H(x_{i+1/2}).
$$
\end{itemize}
One can proof that the semi-discrete numerical scheme \eqref{eq:ho1} is a high-order numerical scheme of order $p>1$, if the following three conditions are satisfied (see \cite{castro2020} and the references there in):

\begin{enumerate}[label=(\roman*)]
    \item $H_{\Delta x}(x)$ is a high-order approximation of $H(x, \rho)$ of order $p>1$;
    \item $R^U_i(x)$ and $R_i^H(x)$ are high-order reconstruction operators of order at least $p>1$;
    \item the volume integral
\[
\frac{1}{\Delta x}\int_{x_{i-1/2}}^{x_{i+1/2}} S_H(R_i^U(x), R_i^H(x)) \, dx
\]
is computed exactly or approximated with a quadrature formula of order greater or equal to $p>1$.
\end{enumerate}

Unfortunately, when employing standard (CWENO, WENO, \ldots) reconstruction the resulting numerical scheme is, in general, not well-balanced. Indeed, if $\{U_i\}_{i=1}^N$ and $\{H_i\}_{i=1}^N$ are the cell averages of a discrete steady state satisfying \eqref{eq:discretewb}, then their reconstructions do not necessary satisfy the discrete relations
\begin{equation}\label{eq:recwb}
 \Pi'\left(R_i^{\rho}(x)\right)+R_i^{H}(x)=C, \ 1 \leq i \leq N.
\end{equation}
In the previous expression we suppose, for simplicity, that we have only one connected component.

In what follows we aim to propose a modified reconstruction procedure which respects \eqref{eq:recwb} for any discrete steady state satisfying \eqref{eq:discretewb}. As we show in the next Subsection, thanks to this modification we can prove that our scheme is both high-order accurate and well-balanced.

%
%

\subsection{High-order well-balanced numerical scheme}\label{subsec:high_WB}

Let us suppose that the sequences of cell averages $\{U_i\}_{i=1}^N$ and $\{K_i\}_{i=1}^N$ are known, with
\[
U_i=\left(
\begin{array}{c}
\displaystyle{\rho_i=\sum_{j=1}^{n_s} \alpha_j\rho(x_i^j)} \\[0.2cm]
\displaystyle{(\rho u)_i=\sum_{j=1}^{n_s}\alpha_j(\rho u)(x_i^j)}
\end{array}
\right)
\]
and
\begin{equation}\label{eq:defK}
    K_i=\sum_{j=1}^{n_s}
\alpha_j
\left[
\Pi'(\rho(x_i^j))+H_{\Delta x}(x_i^j)\right].
\end{equation}
For such cell averages we propose the following reconstruction procedure:
\begin{itemize}
    \item We consider a standard high-order reconstruction operator for the conserved variables $\rho$ and $\rho u$, and also applied to the sequence $\{K_i\}_{i=1}^N$,
    \begin{equation}\label{eq:reconsCWENO}
    \begin{gathered}
    R_i^{\rho}(x)=R_i\left( x,\{\rho_j\}_{j \in S_i}\right), \\[0.2cm]
    R_i^{\rho u}(x)=R_i\left( x, \{(\rho u)_j\}_{j \in S_i}\right),\\[0.2cm]
    R_i^{K}(x)=R_i\left( x,\{K\}_{j \in S_i}\right);
    \end{gathered}
    \end{equation}
    \item the reconstruction operator for $H_{\Delta x}(x)$ is defined as
    \begin{equation}\label{eq:reconsH}
    R_i^H(x)=R_i^K(x)-\Pi'\left( R_i^\rho(x)\right),
    \end{equation}
    with $R_i^H$ not being a polynomial since it depends on the function $\Pi'(\rho)$.
\end{itemize}

The previous reconstruction procedure satisfies the following property:
\begin{theorem}\label{th:recons}
Let $u=0$, $\rho(x)$ and $H_{\Delta x}(x)$ satisfying \eqref{eq:discretewb}, that is $\rho$, $u=0$ is a discrete stationary solution of system \eqref{eq:compactsys}, then the reconstructions $R_i^\rho$(x), $R_i^{\rho  u}(x)$ and $R_i^H(x)$ are discrete stationary solutions of system \eqref{eq:compactsys} at $[x_{i-1/2}, x_{i+1/2}]$.
\end{theorem}
\begin{proof}
Let us suppose for simplicity that the the stationary solution is only defined in one connected component.
Therefore, $K_i=C$, $1\leq i\leq N$.

As standard reconstruction operators like CWENO are exact for constant functions, we have that $R_i^{\rho u}(x)=0$ and $R_i^K(x)=C$. Therefore in \eqref{eq:recwb} we have that $C=\Pi'(R_i^\rho(x))+R_i^H(x)$, which proves the result setting $H_{\Delta x}(x)=R_i^H(x)$.
\end{proof}

It is important to remark that, even if the reconstruction procedure satisfies the discrete steady state of system in \eqref{eq:recwb}, the semi-discrete numerical scheme \eqref{eq:ho1} may not be in general well-balanced. This is because the integral
\[
\int_{x_{i-1/2}}^{x_{i+1/2}} S_H(R_i^U(x), R_i^H(x)) \, dx
\]
has to be numerically approximated, and if such integration is not exact then the well-balancing property may be destroyed (see \cite{castro2020}).  To overcome this difficulty, we follow the strategy proposed in \cite{castro2020}: a local discrete stationary solution is added to the numerical scheme for every cell. We denote this solution by $U_i^*(x)=(\rho_i^*(x), 0)^T$ and $H_i^*(x)$, and it satisfies
\begin{equation}\label{eq:local-steady}
    \frac{1}{\Delta x} \left(F(U_i^*(x_{i+1/2}))-F(U_i^*(x_{i-1/2})) \right)=\frac{1}{\Delta x} \int_{x_{i-1/2}}^{x_{i+1/2}} S_H(U_i^*(x),H_i^*(x))\, dx.
\end{equation}

The previous steady state relation in \eqref{eq:local-steady} is satisfied if we choose $U_i^*(x)=(\rho_i^*(x), 0)^T$ and $H_i^*(x)$ as
\begin{equation}\label{eq:localU}
U_i^*(x)=\left(
\begin{array}{c}
\rho_i^*(x)=R_i^\rho(x) \\
0
\end{array}
\right), \quad H_i^*(x)=K_i-\Pi'(R_i^\rho(x)).
\end{equation}
Observe that the convolution is indirectly approximated in the previous expression.

Now, we could rewrite the semi-discrete numerical scheme \eqref{eq:ho1} by just adding the steady state expression in \eqref{eq:local-steady}, yielding
\begin{equation}\label{eq:hoWB}
\begin{array}{ll}
\displaystyle{\frac{d U_i}{dt}} = &\displaystyle{ -\frac{1}{\Delta x} \left( \mathbb{F}^-_{i+1/2} - \mathbb{F}^+_{i - 1/2} - F(U_i^*(x_{i+1/2}))+F(U_i^*(x_{i-1/2}))\right)}\\[0.3cm]
& \displaystyle{+\frac{1}{\Delta x} \int_{x_{i-1/2}}^{x_{i+1/2}} S_H(R_i^U(x),R_i^H(x))-S_H(U_i^*(x),H_i^*(x)) \, dx}.
\end{array}
\end{equation}

The advantage of this new version of the scheme relies in the fact that the integral term in \eqref{eq:hoWB} could be approximated by any high-order quadrature formula, without perturbing the well-balanced character of the numerical scheme. This comes from the fact that, for any discrete stationary solution satisfying \eqref{eq:discretewb}, we have that $R_i^H(x)=H_i^*(x)$ and $R_i^\rho(x)=\rho_i^*(x)$, so that
\[
S_H(R_i^U(x),R_i^H(x))-S_H(U_i^*(x), H_i^*(x))=0.
\]
For such integral here we follow  \cite{noelle2006well}, where an n-th order Richardson extrapolation formula is proposed to evaluate source terms of the form $\rho \partial_x H$. We detail the fourth- and sixth-order formulas in the Appendix \ref{app:source}. We finally conclude with the following result.

 \begin{theorem} \label{th:WBfinal}
 The numerical scheme \eqref{eq:hoWB} with the reconstruction operators \eqref{eq:reconsCWENO} and \eqref{eq:reconsH} is well-balanced in the sense of Definition \ref{def:WB}.
 \end{theorem}
 \begin{proof}
 Let us suppose that $u=0$, $\rho(x)$, $H_{\Delta x}(x)$ satisfy \eqref{eq:discretewb}. We also assume for simplicity that the stationary solution is defined on one connected component.  Then, as proved in Theorem \eqref{th:recons}, the reconstructions $R_i^\rho(x)$, $R_i^{\rho u}(x)$ and $R_i^H(x)$ satisfy \eqref{eq:recwb}. Moreover,
 \[
 \frac{1}{\Delta x} \int_{x_{i-1/2}}^{x_{i+1/2}} S_H(R_i^U(x),R_i^H(x))-S_H(U_i^*(x),H_i^*(x)) \, dx=0,
 \]
 and the reconstructed states at the intercells verify
 \begin{equation}\label{eq:probWB1}
 \Pi'(\rho^+_{i+1/2})+H^+_{i+1/2}=\Pi'(\rho^-_{i+1/2})+H^-_{i+1/2}=C, \quad u_{i+1/2}^-=u_{i+1/2}^+=0.
 \end{equation}
 The relation \eqref{eq:probWB1} implies that the hydrostatic-reconstructed states satisfy
 \begin{equation}\label{eq:probWB2}
 \rho_{i+1/2}^{HR,-}=\rho_{i+1/2}^{HR,+}, \quad (\rho u)_{i+1/2}^{HR,-}=(\rho u)_{i+1/2}^{HR,+}=0.
 \end{equation}
 Using \eqref{eq:probWB1} and \eqref{eq:probWB2} and the definition of \eqref{eq:numflux} and \eqref{eq:fluxsource}, $\mathbb{F}_{i+1/2}^-$ reduces to
 \[
 \mathbb{F}_{i+1/2}^-=\left(
 \begin{array}{c}
 0 \\
 P(\rho_{i+1/2}^-)
 \end{array}
 \right).
 \]
 Analogously, we deduce that
 \[
 \mathbb{F}_{i-1/2}^+=\left(
 \begin{array}{c}
 0 \\
 P(\rho_{i-1/2}^+)
 \end{array}
 \right).
 \]
 Now, taking into account the definition of $U_i^*(x)$ given in \eqref{eq:localU},
 \[
 U_i^*(x_{i+1/2})=\left(
 \begin{array}{c}
      \rho_{i+1/2}^- \\
      0
 \end{array}
 \right), \quad U_i^*(x_{i-1/2})=\left(
 \begin{array}{c}
      \rho_{i-1/2}^+ \\
      0
 \end{array}
 \right),
 \]
 we finally conclude by noting that
\[
\frac{1}{\Delta x}\left( \mathbb{F}^-_{i+1/2} - \mathbb{F}^+_{i - 1/2} - F(U_i^*(x_{i+1/2}))+F(U_i^*(x_{i-1/2}))\right)=0.
\]
\end{proof}

\begin{remark}
The well-balanced reconstruction operators defined in \eqref{eq:reconsCWENO} and \eqref{eq:reconsH} employ, as expected, the approximated cell averages of the solution at each time step and an extra quantity corresponding to the cell average of the variation of the free-energy, denoted by $K_i$. This quantity plays an important role to achieve the well-balanced property of the reconstruction operators and the final numerical scheme.  Note that the semi-discrete numerical scheme \eqref{eq:hoWB} only allows to evolve in time the cell averages of conserved variables, and as a result we should provide an extra equation to evolve the variation of the free-energy. We propose the following: suppose that $\{\rho_i^n\}$ $\{(\rho u)_i^n\}$ and $\{K^n_i\}$ are known at time $t=n\Delta t$, and suppose in addition that we use the standard explicit first order Euler scheme to evolve the conserved variables up to time $t=(n+1)\Delta t$. Then we propose to update $K_i^{n+1}$ as
\begin{align}
K_i^{n+1}=\,&K_i^{n}+\sum_{j=1}^{n_s} \alpha_j \left[\Pi'\left(R_i^{\rho^{n+1}}(x_i^j)\right)-\Pi'\left(R_i^{\rho^n}(x_i^j)\right) \right]\nonumber\\
& +\sum_{j=1}^{n_s} \alpha_j \left[\sum_{l=1}^{M} \sum_{m=1}^{n_s} \Delta x\,\alpha_m W\left(x_i^j-x_l^m\right) \left[R_l^{\rho^{n+1}} (x_l^m) - R_l^{\rho^n}(x_l^m)\right]\right],
\label{eq:Kn+1}
\end{align}
where we use some high-order quadrature formula for the cell averages and convolution operator. Here we will use the fifth-order Gaussian quadrature described in Appendix \ref{app:gauss}. A similar procedure could be applied if a high-order RK-TVD scheme (see \cite{gottlieb1998total}) is used instead of the explicit Euler scheme to discretize the ODE system \eqref{eq:hoWB}. This is based on the classical observation that these schemes can be written as linear combinations of explicit Euler steps. Observe also that $K_i^{n+1}=K_i^n$ on discrete stationary solutions satisfying \eqref{eq:discretewb}.

\end{remark}

Finally, if the term $S_D(x,U)$ is now added to the system, our full high-order semi-discrete well-balanced finite volume scheme can be written as
\begin{align}
\frac{d U_i}{dt} = & -\frac{1}{\Delta x} \left( \mathbb{F}^-_{i+1/2} - \mathbb{F}^+_{i - 1/2} - F(U_i^*(x_{i+1/2}))+F(U_i^*(x_{i-1/2}))\right)\nonumber\\
& +\frac{1}{\Delta x} \int_{x_{i-1/2}}^{x_{i+1/2}} S_H(R_i^U(x),R_i^H(x))-S_H(U_i^*(x),H_i^*(x)) \, dx \label{eq:hofinal} \\
& - \gamma (\rho u)_i + \frac{1}{\Delta x} \sum_{j=1}^{n_s} \alpha_j R_i^{\rho,t}(x_i^j) \left[ \sum_{l=1}^{M} \sum_{m=1}^{n_s} \Delta x\,\alpha_m \psi (x_i^j-x_l^m)\left(R_i^{u,t}(x_i^j)-R_l^{u,t}(x_l^m)\right) R_l^{\rho,t}(x_l^m) \right] ,\nonumber
\end{align}
where $\displaystyle{R_i^{u,t}=\frac{R_i^{\rho u, t}}{R_i^{\rho,t}}}$. Note that the new terms do not affect to the well-balance property of the scheme as they vanishes when $u=0$.

\begin{remark}
In practical applications is quite important to guarantee that the numerical scheme preserves the non-negativity of the density $\rho_i(t)$. The high-order numerical scheme \eqref{eq:hofinal} preserves the non-negativity of the density as consequence of:
\begin{enumerate}
    \item the first order numerical flux preserves the non-negativity of the density (see \cite{carrillo2018wellbalanced});
    \item the application of the positive-density limiter introduced in \cite{zhang2010maximum} and described in appendix \ref{subsec:positiveCWENO};
    \item a suitable CFL restriction (see \eqref{eq:CFL} and \cite{zhang2010maximum}), described also in appendix \ref{subsec:positiveCWENO}.
\end{enumerate}
\end{remark}

%
%
\subsection{Algorithmic implementation of the scheme}\label{subsec:algorithm}

In this Subsection we summarize the steps to efficiently implement the high-order well-balanced finite-volume scheme of Subsection \ref{subsec:high_WB}.

The initial conditions for system \eqref{eq:compactsys} are the initial
density profile $\rho_0(x)$ and momentum profile $\lt(\rho u\rt)_0 (x)$.
These initial conditions are introduced in the numerical scheme by computing
their cell averages via high-order quadrature formula,
\begin{equation*}
\begin{gathered}
\rho_i^0= \sum_{j=1}^{n_{s}} \alpha_j \rho_0 (x_i^j),\\
\lt(\rho u\rt)_i^0= \sum_{j=1}^{n_{s}} \alpha_j (\rho u)_0 (x_i^j),
\end{gathered}
\end{equation*}
where the coefficients $\alpha_j$ denote the weights of the quadrature
formula that multiply the evaluation of $\rho_0(x)$ and $\lt(\rho u\rt)_0(x)$
at the quadrature points $x_i^j$, and $n_{s}$ denotes the number of quadrature points. Here we employ the fifth-order Gaussian quadrature formula described in the Appendix \ref{app:gauss}.

The initial cell averages of the derivative of the free energy \eqref{eq:discretewb} are also required, and are similarly computed via fifth-order Gaussian quadrature
\begin{equation*}
K_i^0=\sum_{j=1}^{n_{s}} \alpha_j \left[ \Pi'\left(\rho_0(x_i^j)\right)+V(x_i^j )+\sum_{l=1}^{n} \sum_{m=1}^{n_{s}} \Delta x \alpha_m W\left(x_i^j-x_l^m\right) \rho_0(x_l^m)\right].
\end{equation*}

The computation of the cell averages using quadrature formulas is only necessary at the initial time step of the algorithm. For further time steps, the algorithm presented here takes as inputs the cell averages $\rho_i^n$, $\lt(\rho u\rt)_i^n$ and $K_i^n$, evaluated at $t=n \Delta t$, and directly returns the cell averages $\rho_i^{n+1}$, $\lt(\rho u\rt)_i^{n+1}$ and $K_i^{n+1}$ at the subsequent time step  $t=(n+1) \Delta t$. The steps for such algorithm are:
\begin{enumerate}[label=\arabic*)]
\itemsep1em
\item Perform high-order reconstructions $R_i^{\rho}(x)$, $R_i^{\rho u}(x)$ and $R_i^{K}(x)$ from the sequences of cell values $\left\{\rho_i^n\right\}$, $\left\{\lt(\rho u\rt)_i^n\right\}$, $\left\{K_i^n\right\}$, following \eqref{eq:reconsCWENO}. In our case, such reconstructions are conducted via third-
and fifth-order CWENO reconstructions \cite{levy1999central,
levy2000compact, capdeville2008central}, details
provided in Appendix \ref{app:CWENO}.

For simplicity, the evaluations of the previous reconstructions at the intercells at time $t=n\Delta t$ are denoted as
\begin{equation}\label{eq:recons_b}
\begin{gathered}
\rho_{i-1/2}^+=R_{i}^{\rho}(x_{i-1/2}),\quad \rho_{i+1/2}^-=R_{i}^{\rho}(x_{i+1/2}),\\[6pt]
\lt(\rho u\rt)_{i-1/2}^+=R_{i}^{\rho u}(x_{i-1/2 }),\quad \lt(\rho u\rt)_{i+1/2}^-=R_{i}^{\rho u}(x_{i+1/2}),\\[6pt]
K_{i-1/2}^+=R_{i}^{K}(x_{i-1/2}),\quad K_{i+1/2}^-=R_i^{K}(x_{i+1/2}).
\end{gathered}
\end{equation}

Furthermore, when required, the reconstruction of the velocity field $u$ is computed as
\begin{equation*}
    R_i^{u}(x)=\frac{R_i^{\rho u}(x)}{R_i^{\rho}(x)},\quad u_{i-1/2}^+=R_{i}^{ u}(x_{i-1/2}),\quad u_{i+1/2}^-=R_{i}^{u}(x_{i+1/2}).
\end{equation*}

\item Obtain the reconstruction $R_i^H(x)$ for $H_{\Delta x}(x)$ from \eqref{eq:reconsH}, and their evaluations at the intercells as
\begin{equation*}
H_{i-1/2}^+=R_{i}^{H}(x_{i-1/2}),\quad H_{i+1/2}^-=R_i^{H}(x_{i+1/2}).
\end{equation*}
In general $H_{i+1/2}^+\neq H_{i+1/2}^-$. The average value between them is taken as
\begin{equation*}
H_{i+1/2}=\max\left(H_{i+1/2}^+,H_{i+1/2}^-\right).
\end{equation*}

\item Reconstruct the local discrete stationary solution in \eqref{eq:local-steady} for every cell, so that $U_i^*(x)$ and $H_i^*(x)$ are computed from \eqref{eq:localU}.

\item Perform the so-called hydrostatic reconstruction described in \eqref{eq:rhointerface}, but now with the high-order reconstructions at the intercells. By
    denoting as $\xi(s)$ the inverse function of $\Pi'(s)$ for $s>0$,
\begin{equation*}
\begin{gathered}
 \rho_{i+1/2}^{HR,-}=\xi \left(\Pi'\left(\rho_{i+1/2}^-\right)+H_{i+1/2}^--H_{i+1/2}\right),\quad \lt(\rho u\rt)_{i+1/2}^{HR,-}=\rho_{i+1/2}^{HR,-}u_{i+1/2}^-,\\
  \rho_{i+1/2}^{HR,+}=\xi \left(\Pi'\left(\rho_{i+1/2}^+\right)+H_{i+1/2}^+-H_{i+1/2}\right),\quad \lt( \rho u\rt)_{i+1/2}^{HR,+}=\rho_{i+1/2}^{HR,
 +} u_{i+1/2}^+.
\end{gathered}
\end{equation*}

\item The cell averages $\rho_i^{n+1}$ and $\lt(\rho u\rt)_i^{n+1}$ at the subsequent time step $t=(n+1)\Delta t$ are updated by means of \eqref{eq:hofinal},
where
\begin{equation} \label{eq:numflux_2}
    \mathbb{F}_{i+1/2}^\pm( U_{i+1/2}^{HR,-},U_{i+1/2}^{HR,+}, H_{i+1/2}^-, H_{i+1/2}^+)=\mathcal{F}(U_{i+1/2}^{HR,-}, U_{i+1/2}^{HR,+}) \pm S_{i+1/2}^{HR,\pm},
\end{equation}
with the Lax-Friedrich flux in \eqref{eq:LF},
\begin{equation}\label{eq:fluxsource_2}
S_{i+1/2}^{HR,+}=\begin{pmatrix}
0 \\ P\left(\rho_{i+1/2}^+\right)-P\left(\rho_{i+1/2}^{HR,+}\right)\end{pmatrix},\quad  S_{i+1/2}^{HR,-}=\begin{pmatrix}
0 \\ P\left(\rho_{i+1/2}^{HR,-}\right)-P\left(\rho_{i+1/2}^-\right)\end{pmatrix},
\end{equation}
\begin{equation}\label{eq:fluxhom_2}
F\left(U_{i}^*(x_{i-1/2})\right)=\begin{pmatrix}
0 \\ P\left(\rho_{i-1/2}^+\right) \end{pmatrix},\quad  F\left(U_{i}^*(x_{i+1/2})\right)=\begin{pmatrix}
0 \\ P\left(\rho_{i+1/2}^-\right) \end{pmatrix},
\end{equation}
and the integral for the high-order corrections in the source term is computed from the fourth- and sixth-order formulas in the Appendix \ref{app:source}.

\item Finally, update the value $K_i^{n+1}$ by means of fluctuations from \eqref{eq:Kn+1}.

\end{enumerate}

%
%

\section{Numerical simulations}\label{sec:numsim}

Here we employ the high-order finite volume scheme we developed in Section
\ref{sec:numsch} in a variety of relevant applications, taken from the fields
of gas dynamics, porous media, collective behavior and chemotaxis. First, in
Subsection \ref{subsec:validation} we conduct the validation of the
properties from the numerical scheme to ensure both that the high-order and
the well-balanced properties are numerically satisfied. Second, in Subsection
\ref{subsec:numexperiments} we proceed to apply the scheme to challenging
scenarios where analytical results are scarce.

The numerical flux for the simulations is chosen depending on the form of the
pressure term, which satisfies $P(\rho)=\rho^m$ with $m\geq 1$. A local
Lax-Friedrich numerical flux is employed for the examples with ideal-gas
pressure, where $m=1$ and the support of the density is not compact. On the
contrary, a kinetic scheme based on \cite{perthame2001kinetic} is employed
for pressures with $m>1$, due to the presence of vacuum regions and
compactly-supported densities. For the details of these two numerical fluxes
we refer the reader to our previous work \cite{carrillo2018wellbalanced}.

For the temporal integration we implement the versatile third order TVD Runge-Kutta method \cite{gottlieb1998total}, with the CFL number chosen as 0.7 in all the simulations. The CFL conditions for these two numerical fluxes are detailed in equation \eqref{eq:CFL} of appendix \ref{app:CWENO}. The boundary conditions are periodic unless otherwise specified. In all the simulation we set $\gamma=1$ in the linear damping while the nonlinear damping is in general deactivated, except for example \ref{ex:damping}. The number of cells employed to create the plots is $200$. For the figures we use the third-order time discretization scheme, unless otherwise stated.

In the following numerical simulations we focus on the temporal evolution of
the density, momentum and free-energy variation in \eqref{eq:defK}. For
illustrative purposes, we also plot the evolution of the discrete versions of
the total energy in \eqref{eq:totalenergy} and free energy in
\eqref{eq:freeenergy}, which are given by
\begin{equation}\label{eq:freeenergydiscrete}
E^\Delta= \sum_i \frac{\Delta x_i}{2}\rho_i u_i^2 +
\mathcal{F}^\Delta \quad \mbox{ and } \quad
\mathcal{F}^\Delta = \sum_i \Delta x_i \left[\Pi\left(\rho_i\right)+ V_i\rho_i \right]+\frac12 \sum_{i,j} \Delta x_i \Delta x_j W_{ij} \rho_i \rho_j .
\end{equation}
It is worth mentioning that previous works have constructed finite-volume
schemes that satisfy the discrete analog of the entropy identity in
\eqref{eq:entrineq} (see ``entropy stable schemes" in
\cite{tadmor2016entropy,fjordholm2012arbitrarily}) and free energy
dissipation property in \eqref{eq:equalenergy} (see ``energy dissipating
schemes" in \cite{bailo1811fully,carrillo2015finite,carrillo2018wellbalanced}). The extension of the present scheme
to satisfy the challenging discrete properties of entropy stability and energy
dissipation with high-order accuracy will be explored elsewhere.

%
%

\subsection{Validation of the numerical scheme}\label{subsec:validation}

The validation of the finite-volume scheme in Section \ref{sec:numsch}
encompasses a test for the well-balanced property and a test for the
high-order accuracy in the transient regimes. Both tests are conducted in two different scenarios, for
which different choices of the free energy in \eqref{eq:freeenergy} are
taken. The details of such scenarios are written in the examples
\ref{ex:idealpot} and \ref{ex:idealker} of below.

On the one hand, for the well-balanced property we show that the steady-state
solution is preserved in time up to machine precision. For this we select as
initial condition a density and momentum profile satisfying the steady states
obtained from~\eqref{eq:steadyvarener}. Numerically this means that the
discrete version of the variation of the free energy in \eqref{eq:discretewb}
holds, while the momentum vanishes throughout the domain. The results for
this test are depicted in table \ref{table:preserve}, for a simulation time
run from $t=0$ to $t=5$ and a number of cells of $50$.

\begin{table}[!htbp]
\centering
\caption{Preservation of the steady state for the examples \ref{ex:idealpot} and \ref{ex:idealker} with the third- and fifth-order schemes and double precision, at $t=5$}\label{table:preserve}

\begin{tabular}{cccc} \hline
                         & Order of the scheme & $L^1$ error \\
                         \hline
\multirow{2}{*}{Example \ref{ex:idealpot}} & 3\textsuperscript{rd}            & 1.7082e-16     \\
                         & 5\textsuperscript{th}         & 1.7094e-16    \\
                         \hline
\multirow{2}{*}{Example \ref{ex:idealker}} & 3\textsuperscript{rd}            & 5.5020E-17      \\
                         & 5\textsuperscript{th}             & 6.4514E-17     \\
                         \hline
\end{tabular}
\end{table}

On the other hand, the order of accuracy in the transient regimes is computed
by evaluating the $L^1$ error of the numerical solution for a particular
mesh-size grid $\Delta x$ with respect to a reference solution. To measure
the high-order accuracy of the scheme we select scenarios where shock waves
or sharp gradients in the density and momentum profile do not evolve, such as
the ones in the examples \ref{ex:idealpot} and \ref{ex:idealker}, and we let
the simulation run until $t=0.1$. We repeat this procedure by halving the
$\Delta x$ from the previous simulation and the order of the scheme is
computed as
\begin{equation*}
\text{Order of the scheme}=\ln_2\left(\frac{L^1\,\text{error} (\Delta x)}{L^1\,\text{error} (\Delta x/2)}\right).
\end{equation*}
For the system with a nonlocal free energy in \eqref{eq:generalsys} there are
generally no explicit solutions in the transient regime. This implies that
the reference solution has to be computed from the same numerical scheme with
an extremely refined $\Delta x$ with the aim of accepting that numerical
solution as the exact one. Here we take $25600$ cells to compute such
reference solution, while the other numerical simulations for the order of
accuracy employ 50, 100, 200 and 400 cells. The results showing the third-
and fifth-order of accuracy for the scenarios in examples \ref{ex:idealpot}
and \ref{ex:idealker} are displayed in table \ref{table:idealpot} and
\ref{table:idealker}.

%
%

\begin{examplecase}[Ideal-gas pressure under an attractive potential]\label{ex:idealpot}

For this example we select an ideal-gas free energy with pressure
$P(\rho)=\rho$ and with a quadratic external potential $V(x)=x^2/2$. The
steady state that we aim to preserve follows from
\begin{equation}\label{eq:constantidealpot}
\frac{\delta \mathcal{F}}{\delta \rho}=\Pi'(\rho)+H(x,\rho)=\ln(\rho)+\frac{x^2}{2}=\text{constant}\ \text{on}\ \mathrm{supp}(\rho)\ \text{and}\ u=0.
\end{equation}
Free energies of this type appear in the context of chemotaxis with a fixed
chemoattractant profile \cite{gamba2003percolation, serini2003modeling,
filbet2005approximation}, where cells will typically vary their direction
when reacting to the presence of a chemical substance, so that they are
attracted by chemically favorable environments and dodge unfavorable ones. In
chemotaxis there is a chemo-attractant function playing a role similar to the
external potential $V(x)$, and in more complex models this chemo-attractant
function may even have its own parabolic equation for its evolution. There
are numerous well balanced schemes for chemotaxis. Amongst them we highlight
the fully implicit finite-volume scheme in \cite{filbet2006finite}, the
scheme allowing for vacuum states in \cite{natalini2012well}, the Godunov
scheme in \cite{gosse2012asymptotic} and the high-order finite-volume and
finite-differences schemes in \cite{filbet2005approximation,xing2006high,
natalini2012asymptotic}.

The density profile for the steady state in \eqref{eq:constantidealpot} for
an initial mass $M_0$ satisfies a Gaussian distribution of the form
\begin{equation}\label{eq:steadyidealpot}
\rho_{\infty}=M_0 \frac{e^{-x^2/2}}{\int_\R e^{-x^2/2} dx}.
\end{equation}

\begin{table}[!htbp]
\centering
\caption{Accuracy test for Example \ref{ex:idealpot} with the third- and fifth-order schemes, at $t=0.1$}
\label{table:idealpot}
\begin{tabular}{c c c c c }
\hline
\multirow{2}{*}{\begin{tabular}[c]{@{}c@{}}Number of \\ cells\end{tabular}} & \multicolumn{2}{c}{ Third-order} & \multicolumn{2}{c}{Fifth-order}   \\ \cline{2-5}
                                                                                & $L^1$ error    & order    & $L^1$ error           & order          \\ \hline
50                                                                              &               1.4718E-04 &     -     &      1.9260E-05                &     -             \\ \hline
100                                                                             &               2.3726E-05 &     2.63     &        5.1254E-07               &           5.23        \\ \hline
200                                                                             &               2.4182E-06 &    3.29      &        2.1997E-08               &         4.54         \\ \hline
400                                                                             &               2.6708E-07 &    3.18      &        9.2613E-10               &         4.57        \\ \hline
\end{tabular}
\end{table}

For the order of accuracy test we take as initial condition a perturbation of the steady state in \eqref{eq:steadyidealpot},
\begin{equation}\label{eq:icidealpot}
\rho(x,t=0)=M_0\frac{e^{-x^2/2}+0.1*e^{-5(x+3)^2}}{\int_\R\left(e^{-x^2/2}+0.1*e^{-5(x+3)^2}\right)dx},\quad \rho u(x,t=0)=0,\quad x\in [-5,5],
\end{equation}
with $M_0$ equal to $1$ so that the total mass is also 1. The order of
accuracy test from this example is shown in table \ref{table:idealpot}, while
the temporal evolution of the density, momentum, free-energy with respect to
the density, total energy and free energy are displayed in figure
\ref{fig:idealpot}. The third- and fifth-order of accuracy of the numerical
scheme in Section \ref{sec:numsch} are evident from table
\ref{table:idealpot}. From figure \ref{fig:idealpot} (A) we notice how the
Gaussian distribution corresponding to the steady state is reached at $t=12$,
while in figure \ref{fig:idealpot} (C) the variation of the free energy is
constant throughout the domain, given that the density is not compactly
supported and \eqref{eq:discretewb} is satisfied. Finally, it is evident from
figure \ref{fig:idealpot} (D) that the discrete analogs of the total energy
and free energy \eqref{eq:freeenergydiscrete} decay in time.

\begin{figure}[ht!]
\begin{center}
\subfloat[Evolution of the density]{\protect\protect\includegraphics[scale=0.4]{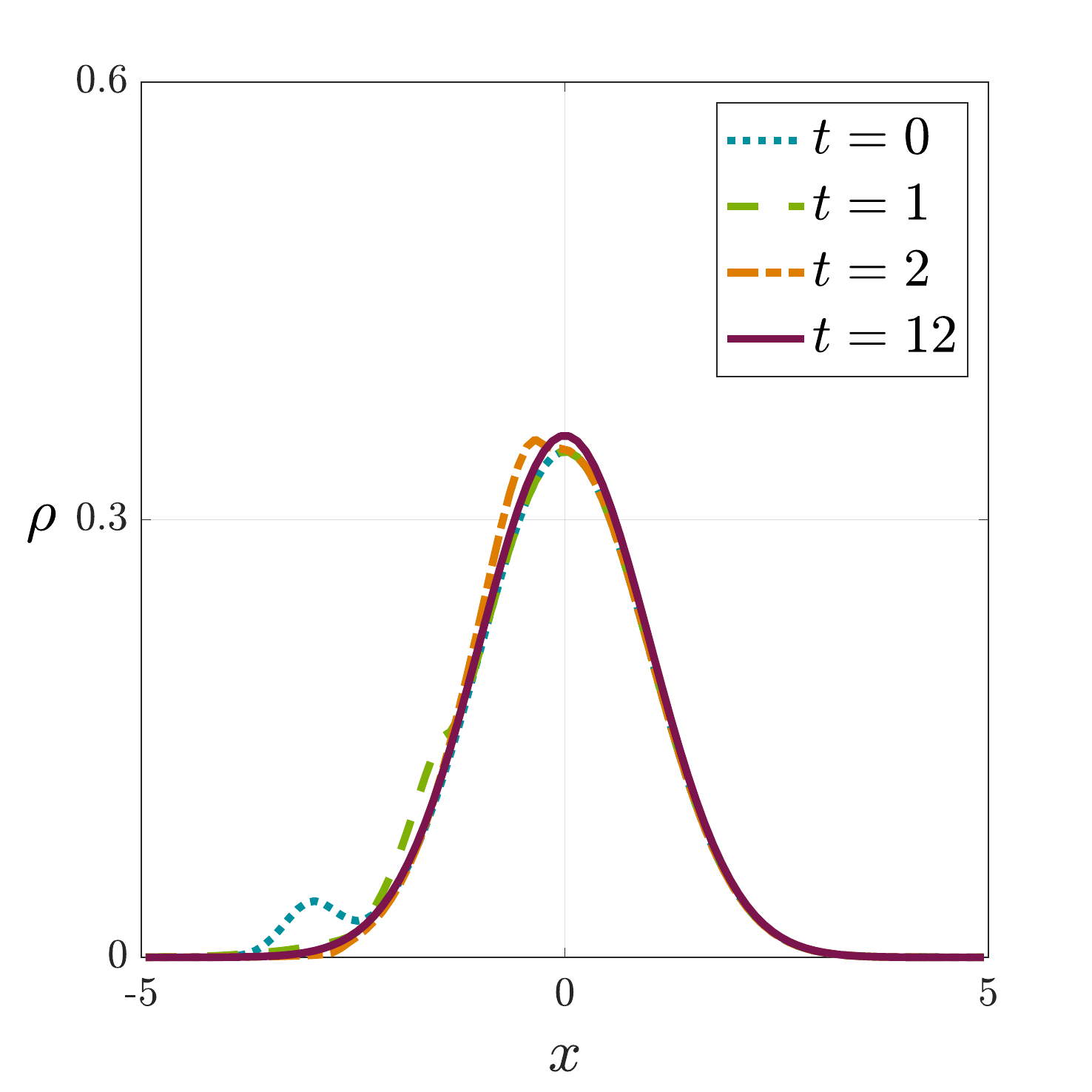}
}
\subfloat[Evolution of the momentum]{\protect\protect\includegraphics[scale=0.4]{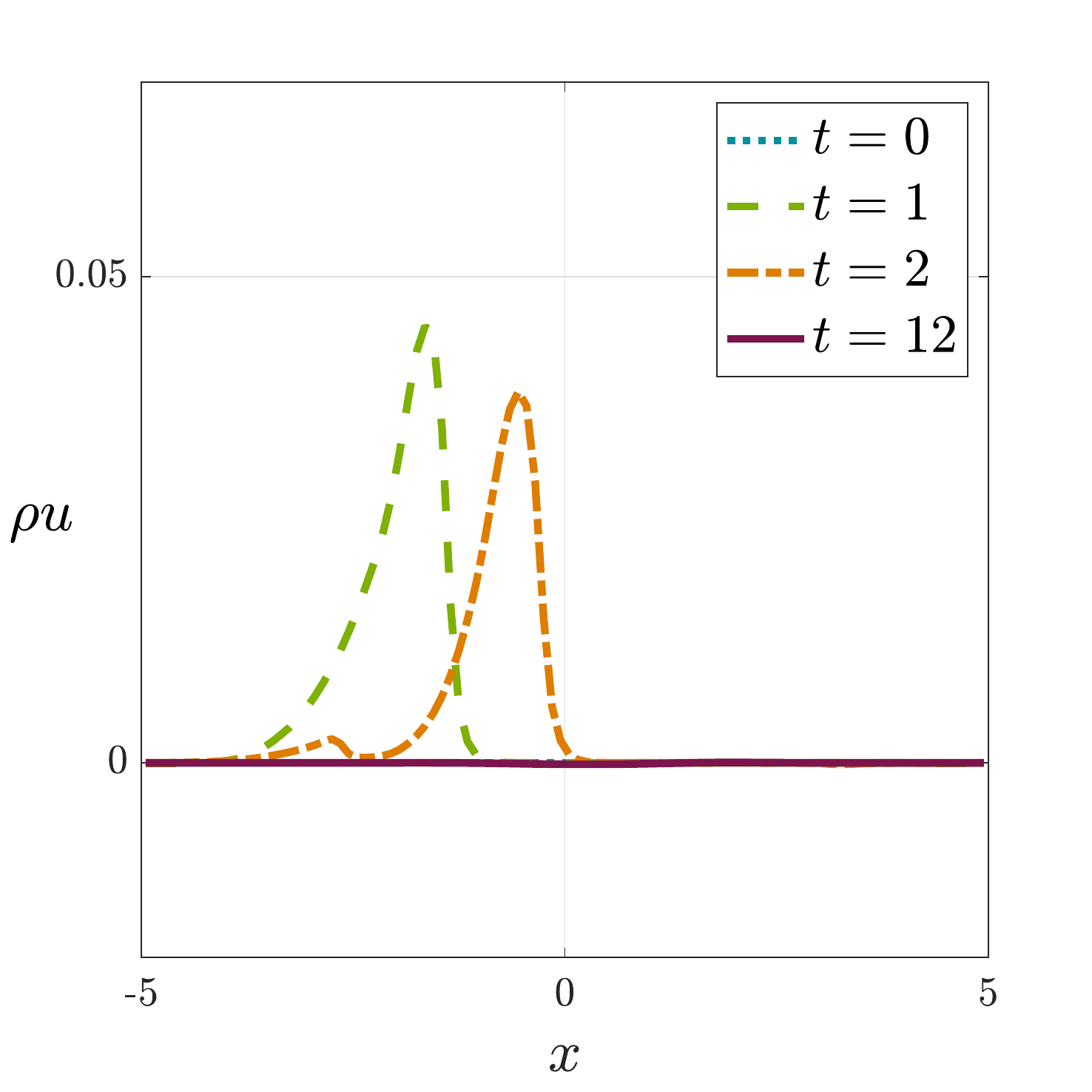}
}\\
\subfloat[Evolution of the variation of the free energy]{\protect\protect\includegraphics[scale=0.4]{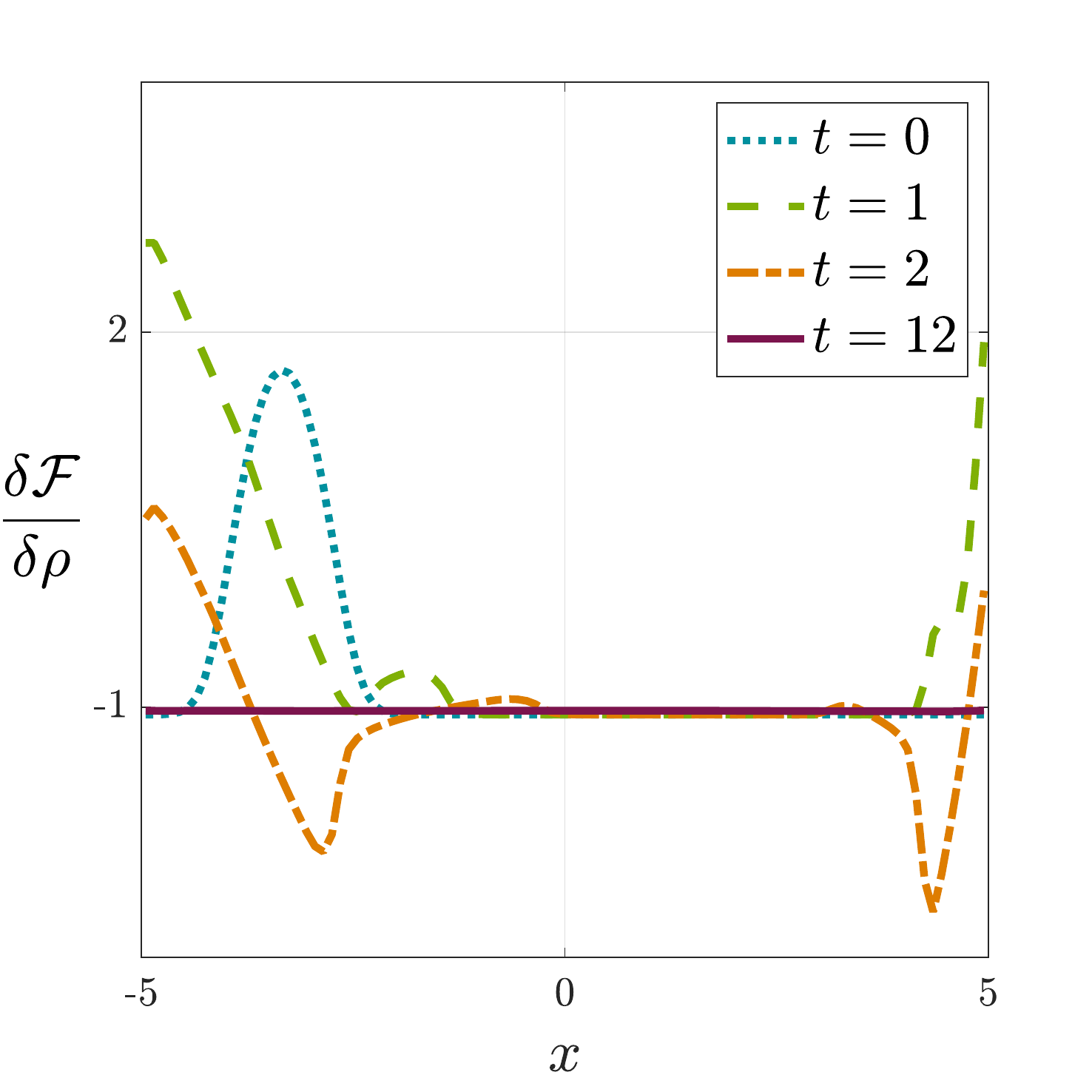}
}
\subfloat[Evolution of the total energy and free energy]{\protect\protect\includegraphics[scale=0.4]{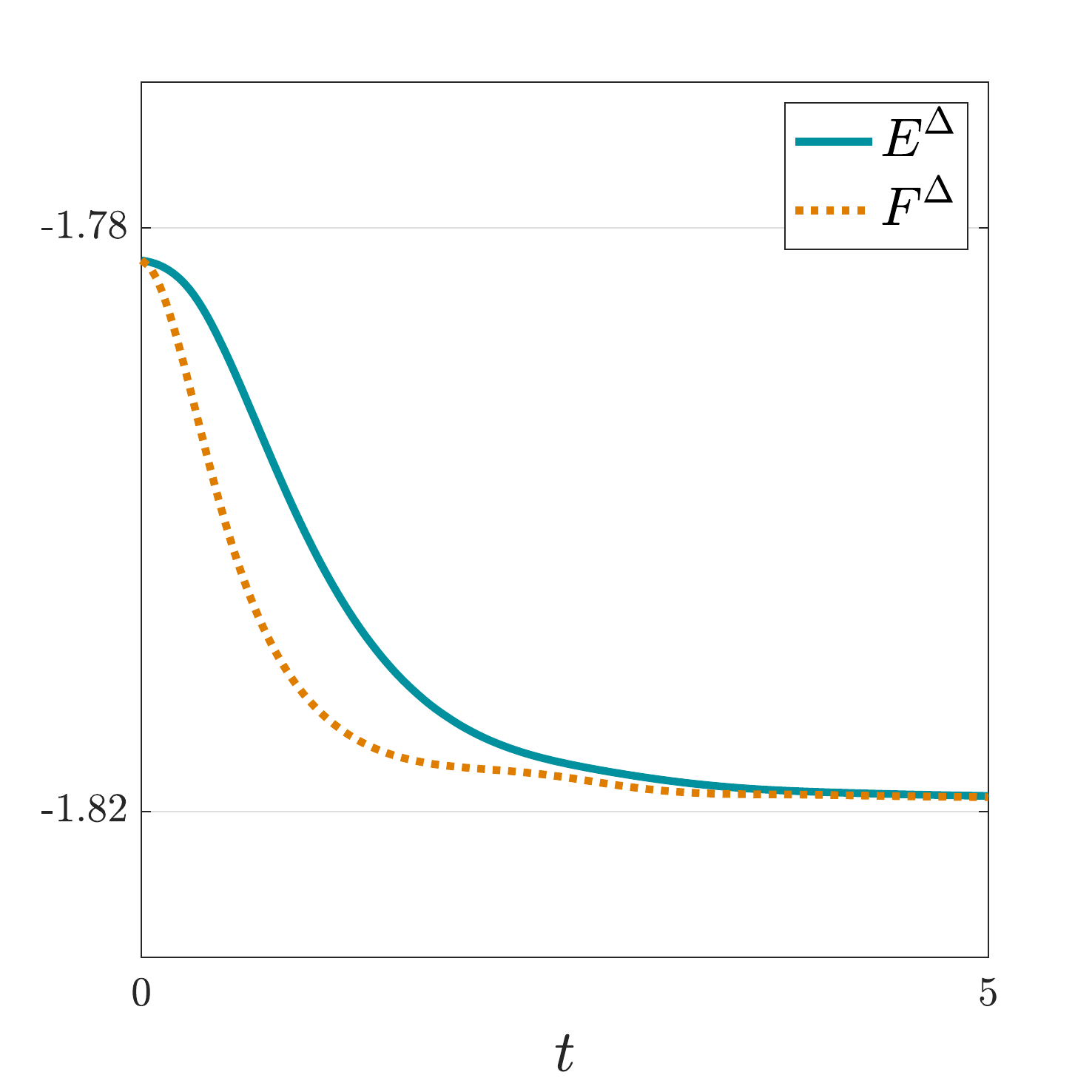}
}
\end{center}
\protect\protect\caption{\label{fig:idealpot} Temporal evolution of Example \ref{ex:idealpot}.}
\end{figure}

In figure \ref{fig:idealpot_comp} we visually illustrate the difference in accuracy between employing the third- and fifth-order schemes of this work versus the first-order scheme in our previous work \cite{carrillo2018wellbalanced}, which is summarized in Subsection \ref{subsec:first_order}. We display the density and momentum fluctuation profiles at two different times ($t=0.2$ and $t=0.4$), which result from substracting the initial conditions in \eqref{eq:icidealpot} to the numerical profiles obtained with the same mesh of 100 cells for the three schemes. The choice of measuring the fluctuations with respect to the initial condition is motivated by capturing the transient behavior. We also plot a reference profile obtained with the thrid-order scheme and 12600 cells. From figure \ref{fig:idealpot_comp} we observe the benefit of employing the high-order schemes in comparison to the first-order one, since they provide a numerical solution much closer to the reference profile for the same number of cells.

\begin{figure}[ht!]
\begin{center}
\subfloat[Density fluctuations at $t=0.2$]{\protect\protect\includegraphics[scale=0.4]{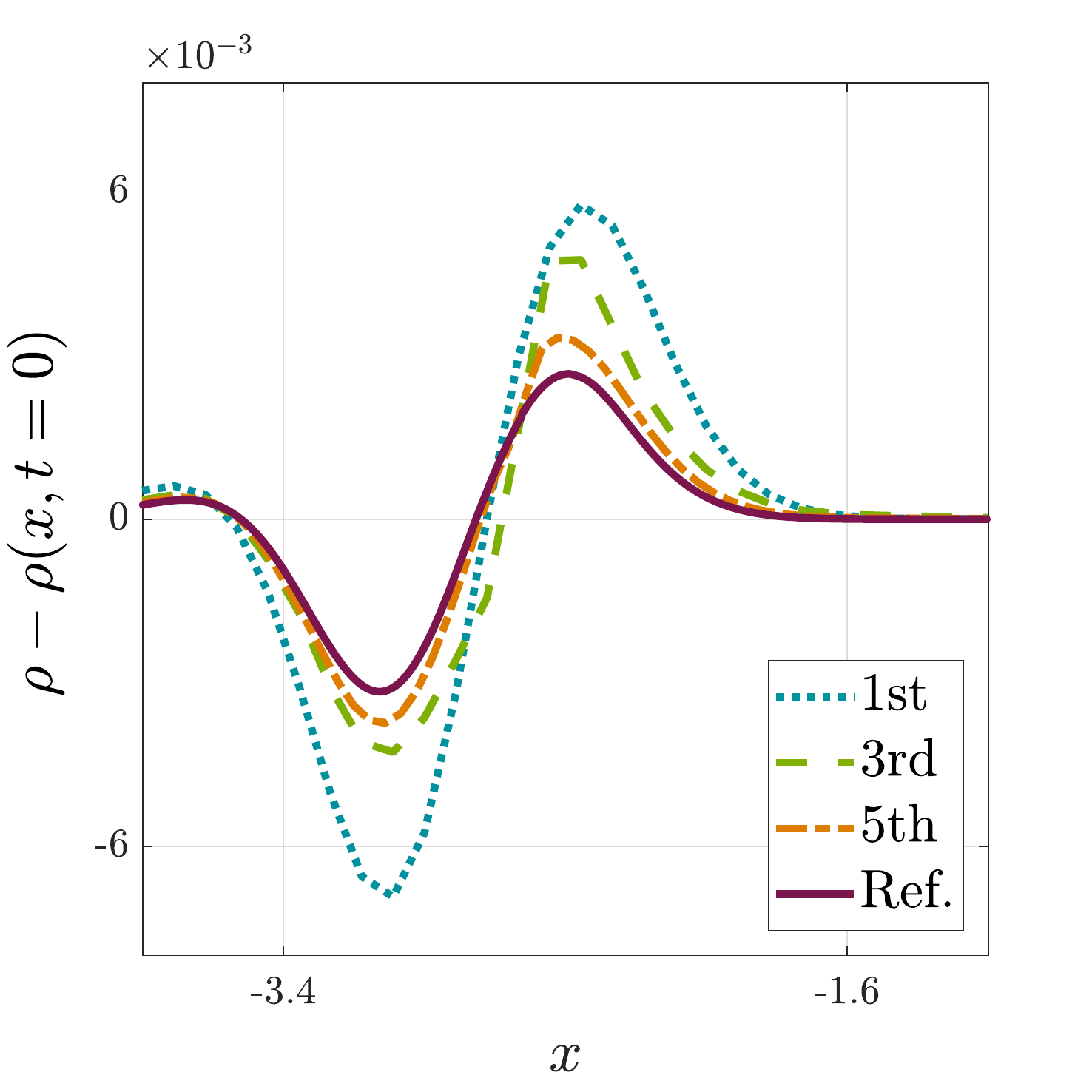}
}
\subfloat[Momentum fluctuations at $t=0.2$]{\protect\protect\includegraphics[scale=0.4]{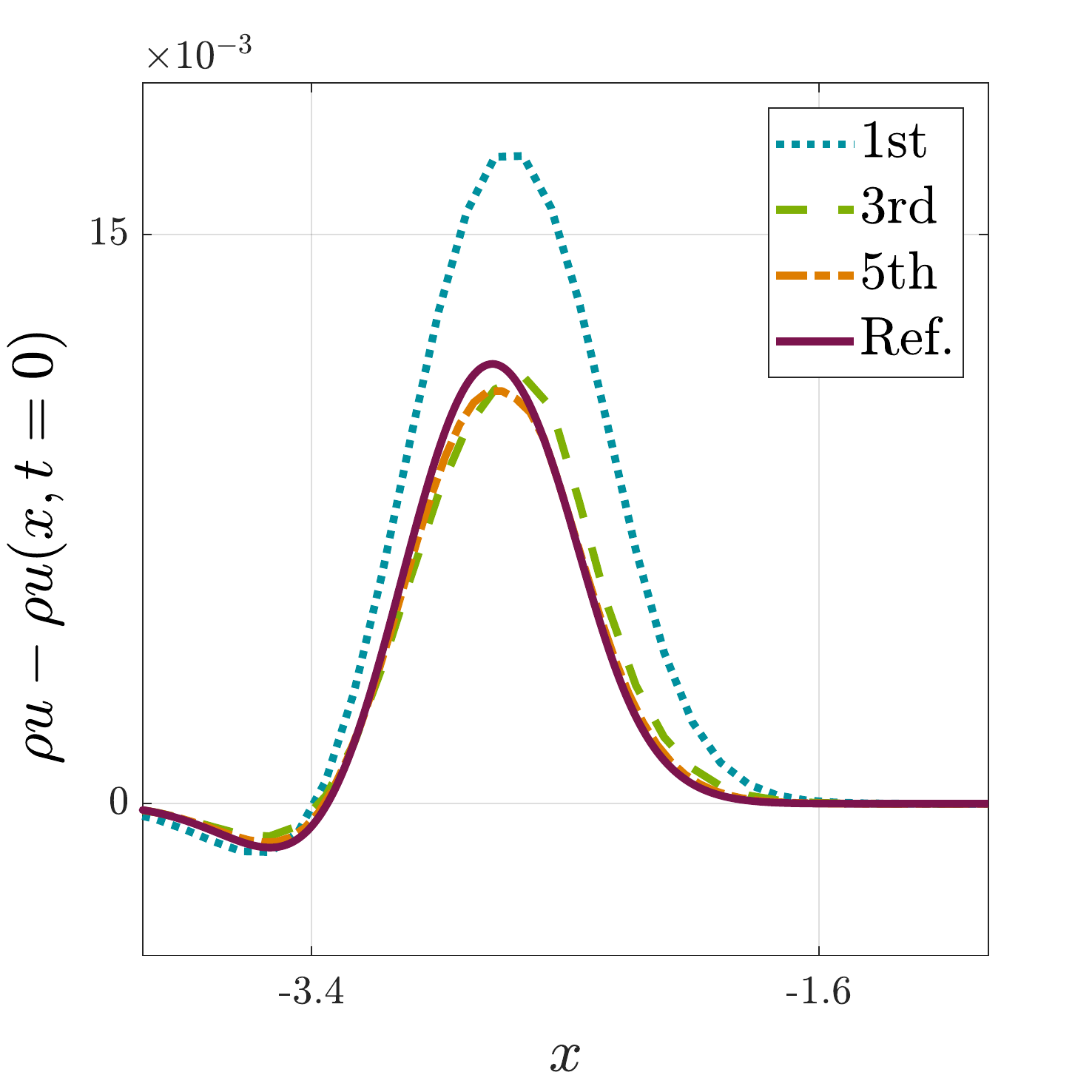}
}\\
\subfloat[Density fluctuations at $t=0.4$]{\protect\protect\includegraphics[scale=0.4]{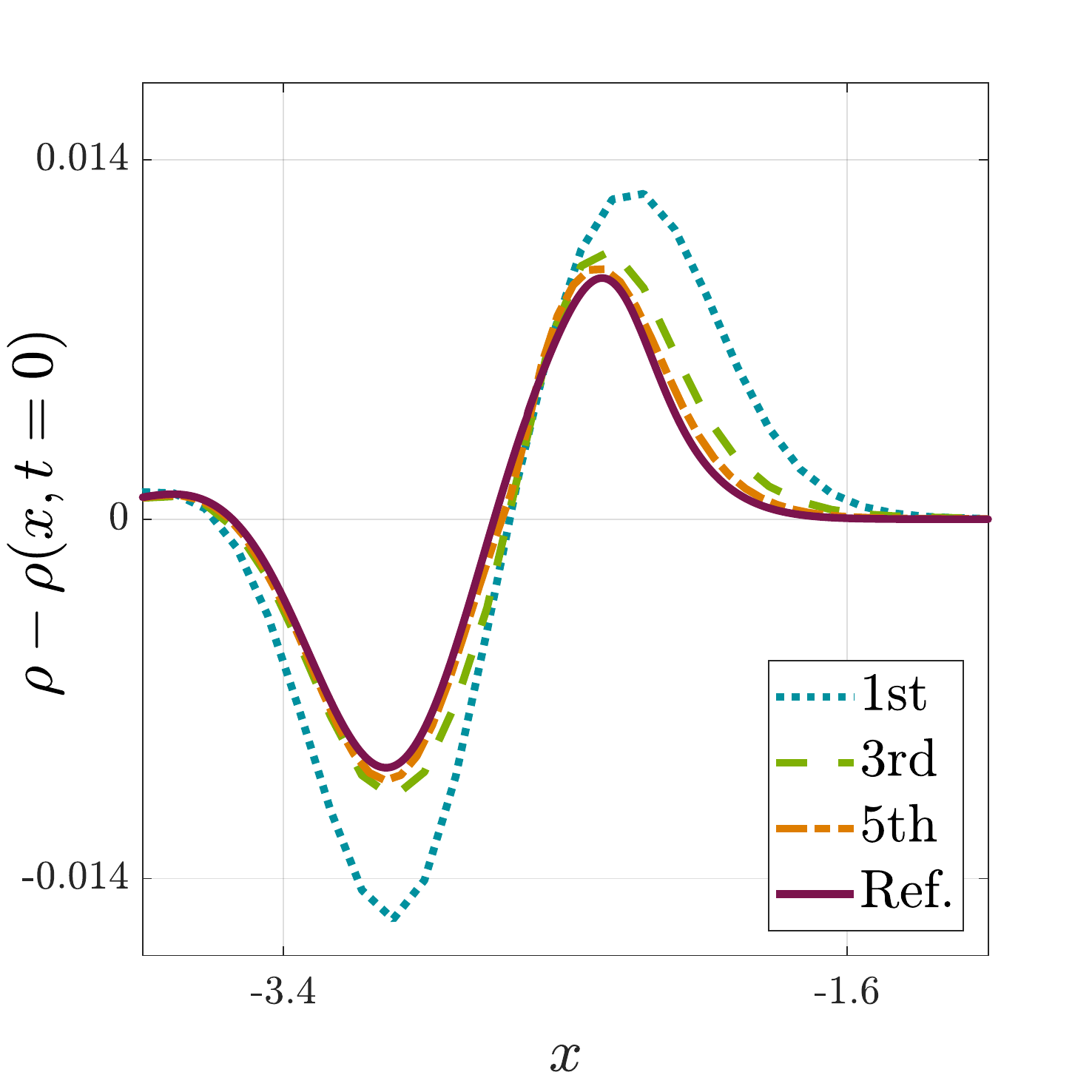}
}
\subfloat[Momentum fluctuations at $t=0.4$]{\protect\protect\includegraphics[scale=0.4]{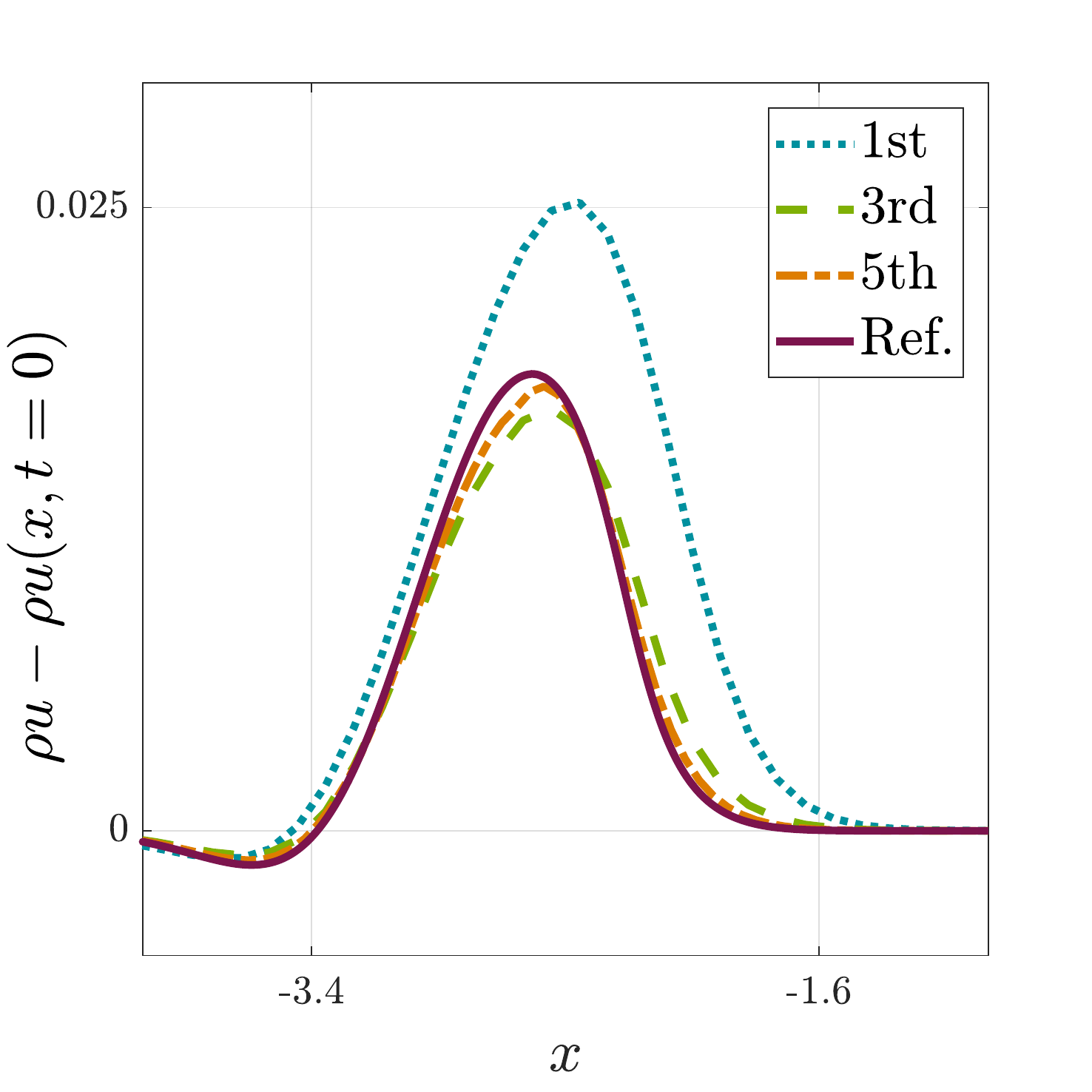}
}
\end{center}
\protect\protect\caption{\label{fig:idealpot_comp} Density and momentum fluctuations in Example \ref{ex:idealpot} for a first-, third- and fifth-order scheme with the same mesh of 100 cells. The reference solution is computed with the third-order scheme and 12600 cells.}
\end{figure}

\end{examplecase}

%
%

\begin{examplecase}[Generalized Euler-Poisson system: ideal-gas pressure and attractive kernel]\label{ex:idealker}

For this example we select an ideal-gas free energy with pressure $P(\rho)=\rho$ together with an interaction potential with a kernel
of the form $W(x)=\frac{x^2}{2}$. In this case the steady state aimed to be preserved satisfies

\begin{equation}\label{eq:constantidealker}
\frac{\delta \mathcal{F}}{\delta \rho}=\Pi'(\rho)+H(x,\rho)=\ln(\rho)+\frac{x^2}{2}\star \rho=\text{constant}\ \text{on}\ \mathrm{supp}(\rho)\ \text{and}\ u=0.
\end{equation}

Free energies of this type are common in Euler-Poisson systems, in which the
Euler equations for a compressible gas are coupled to a self-consistent force
field created by the gas particles \cite{hadvzic2019class}. This interaction
could be gravitational, leading to the modelling of Newtonian stars
\cite{binney2011galactic}, or electrostatic with repelling forces between the
particles is is the case of plasma \cite{guo2011global,
cordier2000quasineutral}. For Euler-Poisson systems the free energy contains
a function $S(t,x)$ which follows a Poisson-like equation, so that
\begin{equation}\label{eq:freeenergyPos}
\frac{\delta \mathcal{F}}{\delta \rho}=\Pi'(\rho)+S(t,x)\quad \text{and} \quad \partial_{xx}S(t,x)=c\rho,
\end{equation}
with $c$ being either $1$ for the gravitational case or -$1$ for the plasma
one. The Poisson equation for $S(t,x)$ can be
solved considering the fundamental solution of the Laplacian in one dimension
\cite{lattanzio2017gas}, which leads to $2S=-c\,|x|\star \rho$. Then, by
plugging this expression for $S$ in the variation of the free energy in
\eqref{eq:freeenergyPos}, one recovers the interaction potential $W(x)$ which
is convoluted with the density $\rho$. For a $S(t,x)$ following the Poisson
equation the interaction potential is $W(x)=-c\,|x|$, but for $c=-1$ one can
generalize it to a homogeneous kernel $W(x)=|x|^\alpha/\alpha$, where
$\alpha>-1$ and $W(x)=\ln|x|$ when $\alpha=0$ for convention. A popular
application of these more general kernels $W(x)$ is in the Keller-Segel
system for cells and bacteria \cite{bian2013dynamic, calvez2017equilibria,
carrillo2018ground} which we explore in example \ref{ex:KS}.

\begin{table}[!htbp]
\centering
\caption{Accuracy test for Example \ref{ex:idealker} with the third- and fifth-order schemes, at $t=0.1$}
\label{table:idealker}
\begin{tabular}{c c c c c }
\hline
\multirow{2}{*}{\begin{tabular}[c]{@{}c@{}}Number of \\ cells\end{tabular}} & \multicolumn{2}{c}{ Third-order} & \multicolumn{2}{c}{Fifth-order}   \\ \cline{2-5}
                                                                                & $L^1$ error    & order    & $L^1$ error           & order          \\ \hline
50                                                                              &               5.0109E-04 &     -     &      1.0913E-04                &     -             \\ \hline
100                                                                             &               1.2721E-04 &     1.98     &        5.0556E-06               &           4.43        \\ \hline
200                                                                             &               1.7573E-05 &    2.86      &        5.3713E-08               &         6.56         \\ \hline
400                                                                             &               2.3001E-06 &    2.93      &        2.3448E-10               &         4.52        \\ \hline
\end{tabular}
\end{table}

For Example 3.2 we select $\alpha=2$, leading to the interaction potential
$W(x)=\frac{x^2}{2}$ in the variation of the free energy in
\eqref{eq:constantidealker}. The steady state for a general mass $M_0$ is
equal to the steady state for example \ref{ex:idealpot} and satisfies
\eqref{eq:steadyidealpot}.
Notice that the particular choice of $W(x)=\frac{x^2}{2}$ and a symmetric initial  condition makes this Example analytically equivalent to the case of external quadratic potential in Example 3.1 with the same initial data, just expand the convolution and use symmetry. However, by treating it numerically as a convolution we are able to check the order of accuracy for interaction potentials.
For the order of accuracy test the initial
condition is a symmetric perturbation of the steady state in
\eqref{eq:steadyidealpot},
\begin{equation*}
\rho(x,t=0)=M_0\frac{e^{-x^2/2}+0.05*e^{-5(x+3)^2}+0.05*e^{-5(x-3)^2}}{\int_\R\left(e^{-x^2/2}+0.05*e^{-5(x+3)^2}+0.05*e^{-5(x-3)^2}\right)dx},\quad \rho u(x,t=0)=0,
\end{equation*}
with $x\in [-10,10]$ and $M_0$ equal to $1$ so that the total mass is also 1. The order of
accuracy test from this example is shown in table \ref{table:idealker}, while
the temporal evolution of the density, momentum, variation of the free energy
with respect to the density, total energy and free energy are depicted in
figure \ref{fig:idealker}. The third- and fifth-order of accuracy of the
numerical scheme in Section \ref{sec:numsch} are evident from table
\ref{table:idealker}. Figure \ref{fig:idealker} (A) shows that the density
remains symmetric at all times eventually reaching the steady state profile
in \eqref{eq:steadyidealpot}. It is also evident from figure
\ref{fig:idealker} (C) that the variation of free energy reaches a constant
value in the regions where the density is non-compactly supported, while
figure \ref{fig:idealker} (D) demonstrates that the total energy and free
energy exhibit a temporal decay.

\begin{figure}[ht!]
\begin{center}
\subfloat[Evolution of the density]{\protect\protect\includegraphics[scale=0.4]{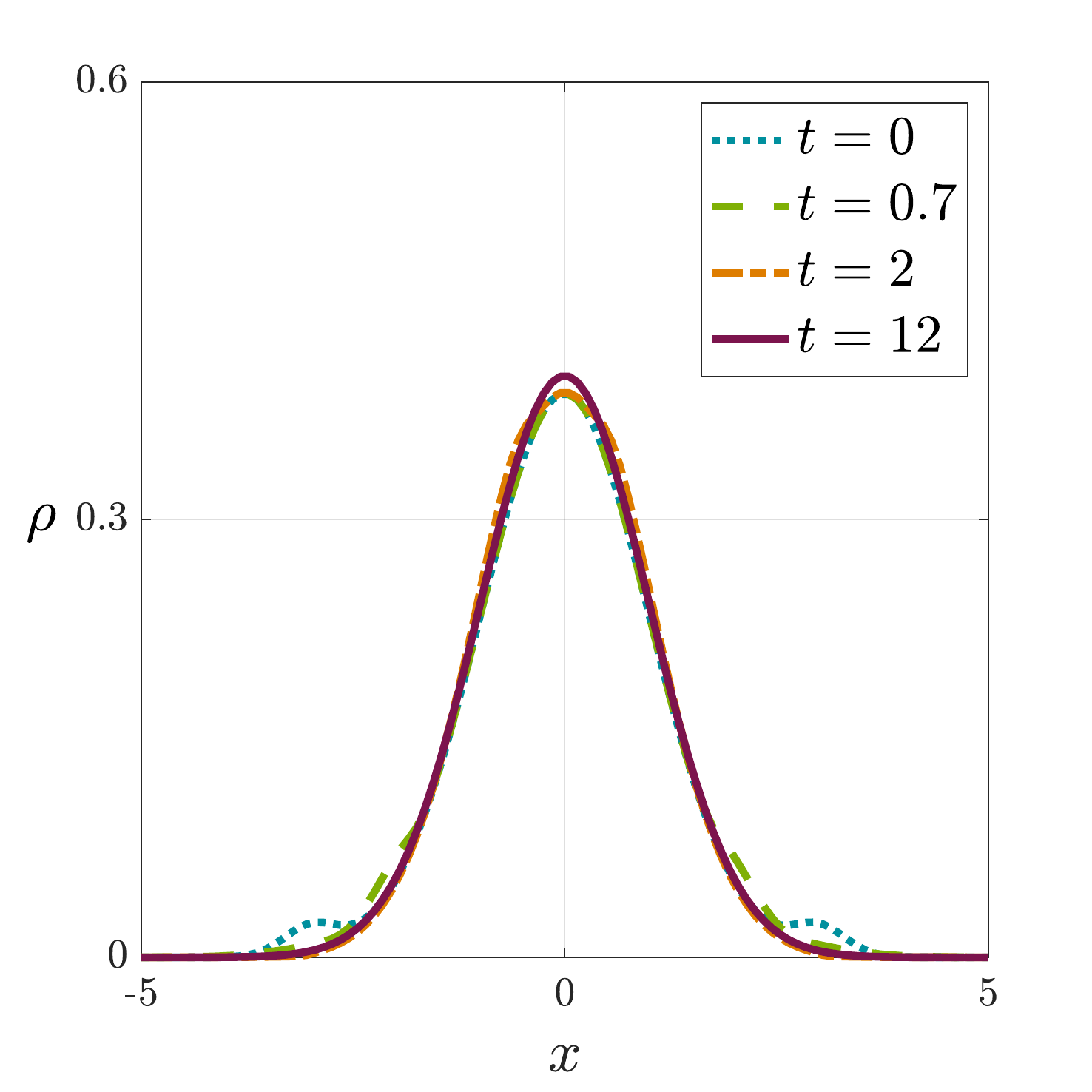}
}
\subfloat[Evolution of the momentum]{\protect\protect\includegraphics[scale=0.4]{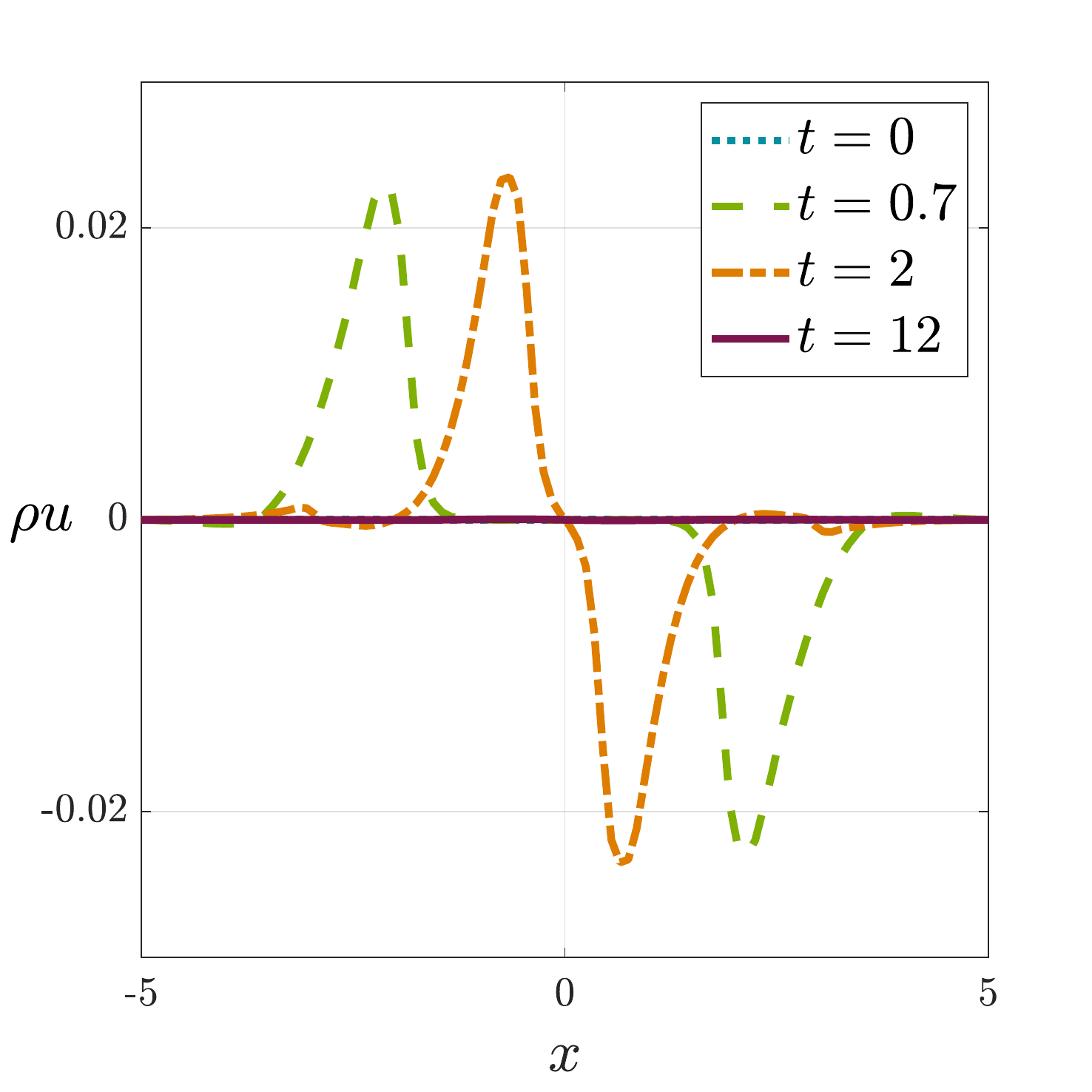}
}\\
\subfloat[Evolution of the variation of the free energy]{\protect\protect\includegraphics[scale=0.4]{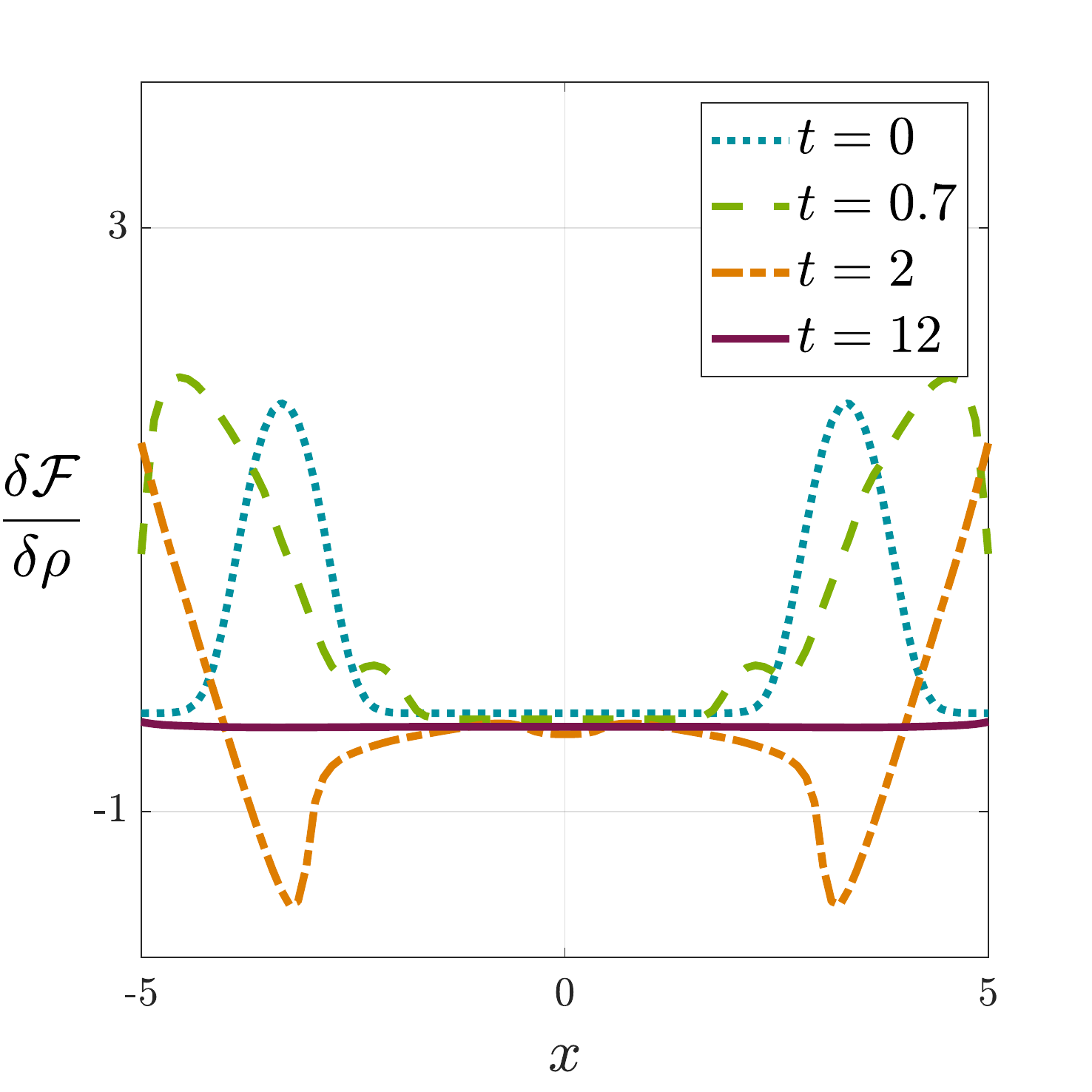}
}
\subfloat[Evolution of the total energy and free energy]{\protect\protect\includegraphics[scale=0.4]{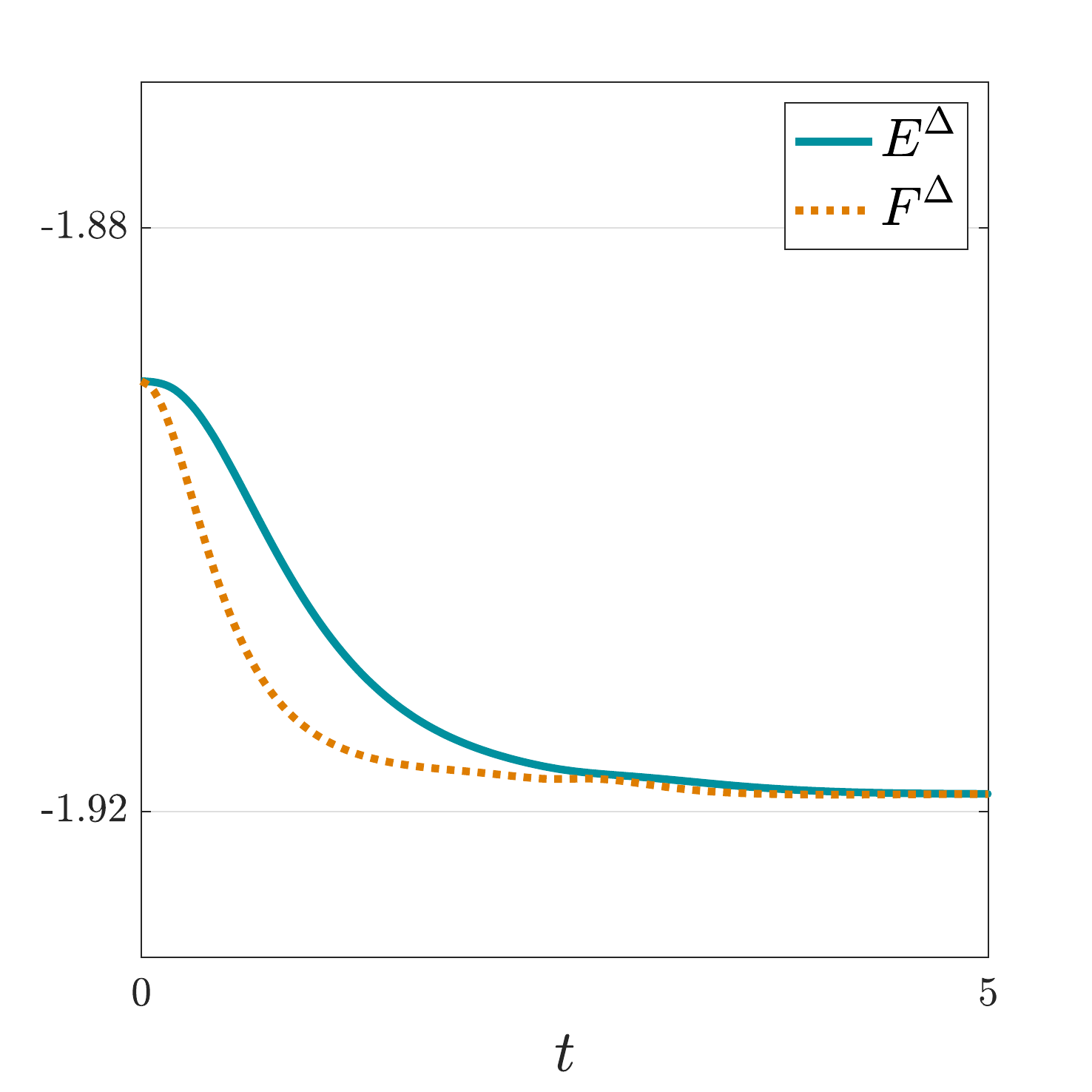}
}
\end{center}
\protect\protect\caption{\label{fig:idealker} Temporal evolution of Example \ref{ex:idealker}.}
\end{figure}

\end{examplecase}

%
%

\subsection{Numerical experiments and applications}\label{subsec:numexperiments}
Here we apply the finite-volume scheme we developed in Section
\ref{sec:numsch} to applications of the shallow-water system, a collective
behaviour system with Cucker-Smale and Motsch-Tadmor damping terms, and the
Keller-Segel model. Our scheme is useful to
run challenging numerical experiments for which analytical results are
limited in the literature, such as in the above applications.

%
%
\begin{examplecase}[Shallow water: pressure proportional to square of density and attractive potential]\label{ex:idealdoublewell}

In this example we select a pressure satisfying $P=\rho^2$ together with an
attractive external potential $V(x)$. This scenario corresponds to the
well-known shallow-water equations, which model free-surface gravity waves
whose wavelength is much larger than the characteristic bottom depth. The
choice of $P=\rho^2$ leads to the presence of dry regions during the
water-height evolution. These equations are applied in a wide range of
engineering and scientific applications involving free-surface flows
\cite{xing2014survey}, such as tsunami propagation \cite{castro2019third},
dam break and flooding problems \cite{cozzolino2018solution} and the
evolution of rivers and coastal areas \cite{churuksaeva2015mathematical}.

The main three challenges to accurately simulate the shallow-water equations
are the preservation of the steady states, the preservation of the
water-height positivity and the transitions between wet and dry areas. Many
authors have consequently proposed various numerical schemes addressing these
challenges, employing methodologies ranging from finite-difference and
finite-volume schemes to discontinuous Galerkin ones. The reader can find
more relevant references about high-order schemes \cite{castro13},
well-balanced reconstructions \cite{audusse2004fast}, density positivity
\cite{zhang2010maximum} and the simulation of the wet/dry front
\cite{gallardo2007well} in the introduction of this work and in the
comprehensive survey from Xing and Shu \cite{xing2014survey}.

\begin{figure}[ht!]
\begin{center}
\subfloat[Evolution of the density]{\protect\protect\includegraphics[scale=0.4]{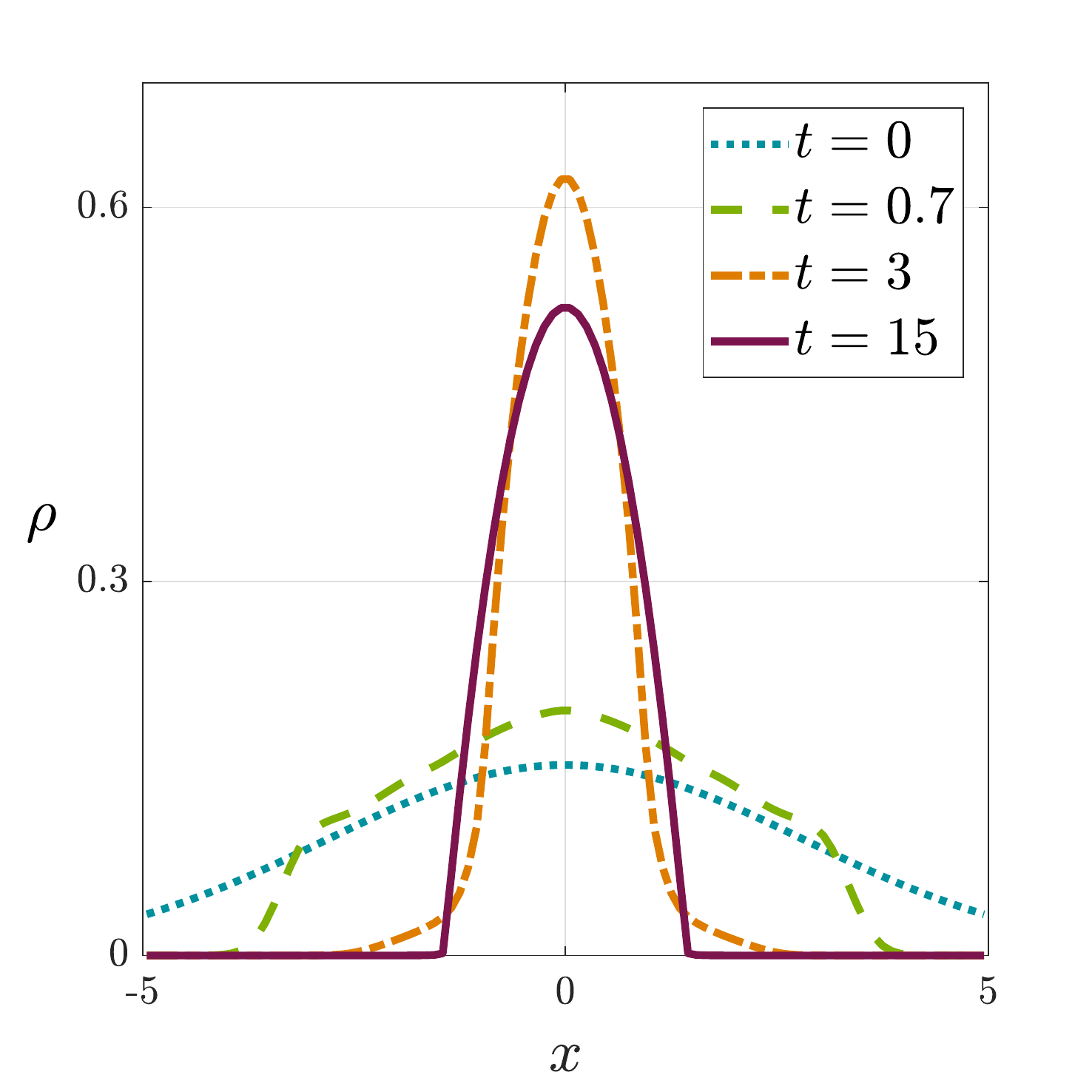}
}
\subfloat[Evolution of the momentum]{\protect\protect\includegraphics[scale=0.4]{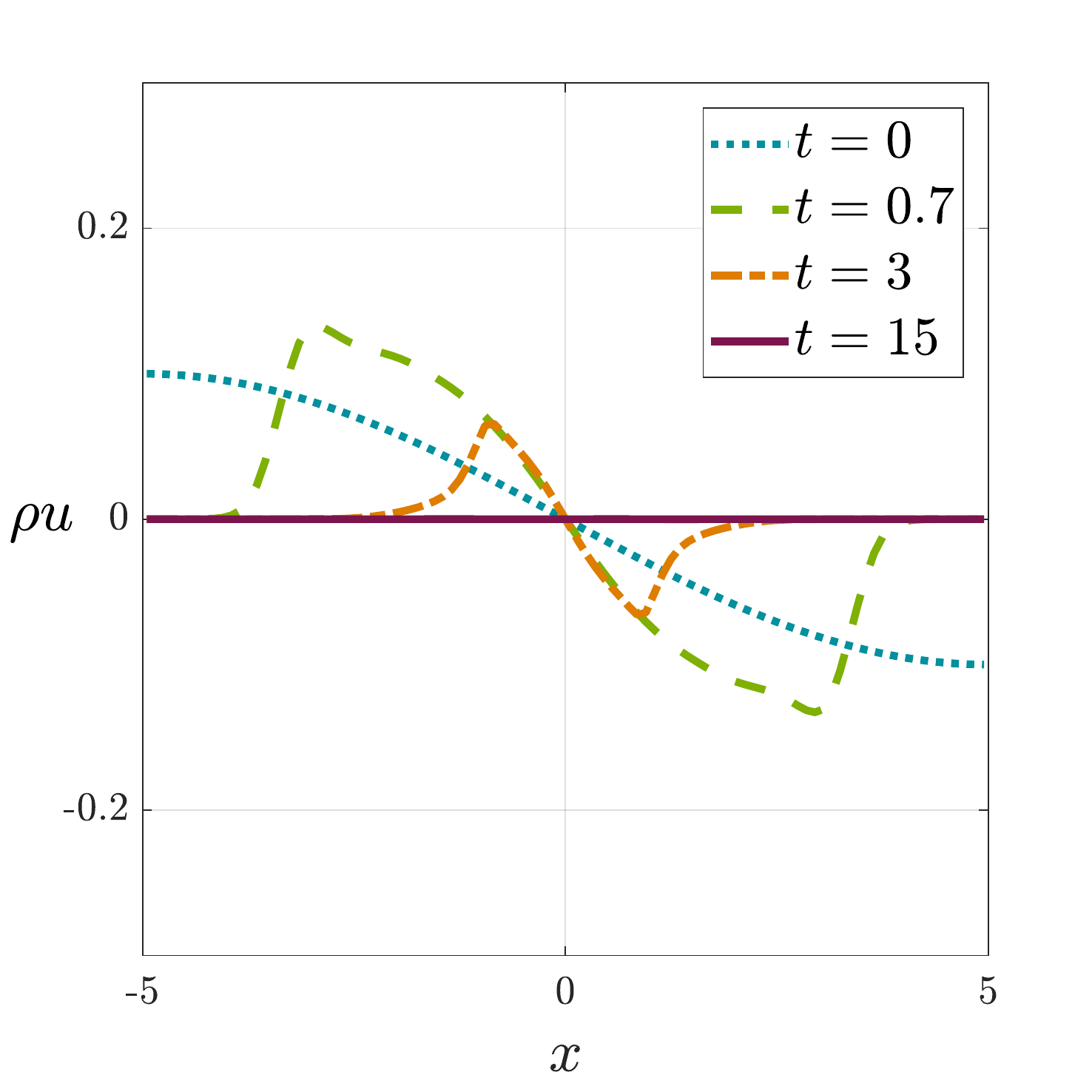}
}\\
\subfloat[Evolution of the variation of the free energy]{\protect\protect\includegraphics[scale=0.4]{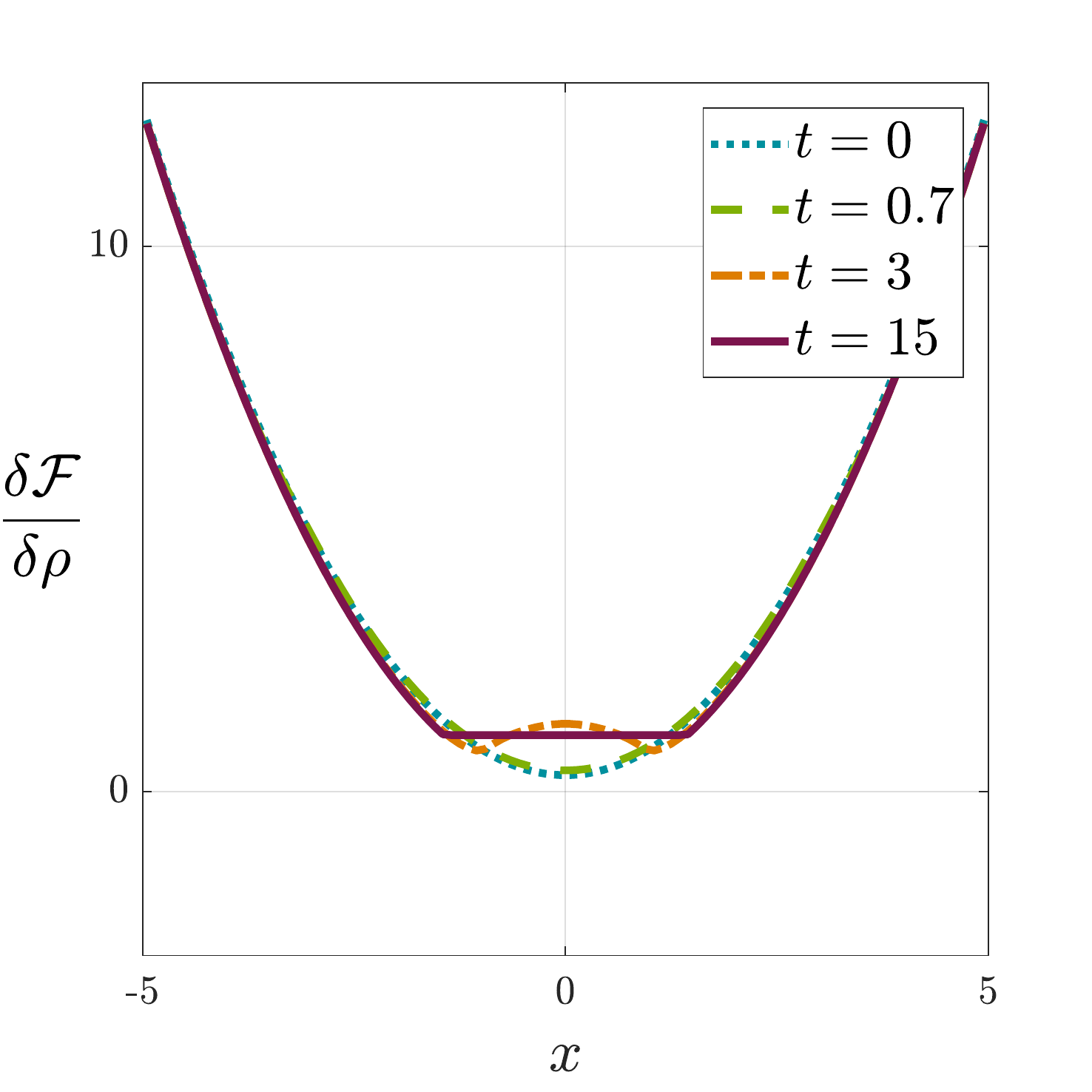}
}
\subfloat[Evolution of the total energy and free energy]{\protect\protect\includegraphics[scale=0.4]{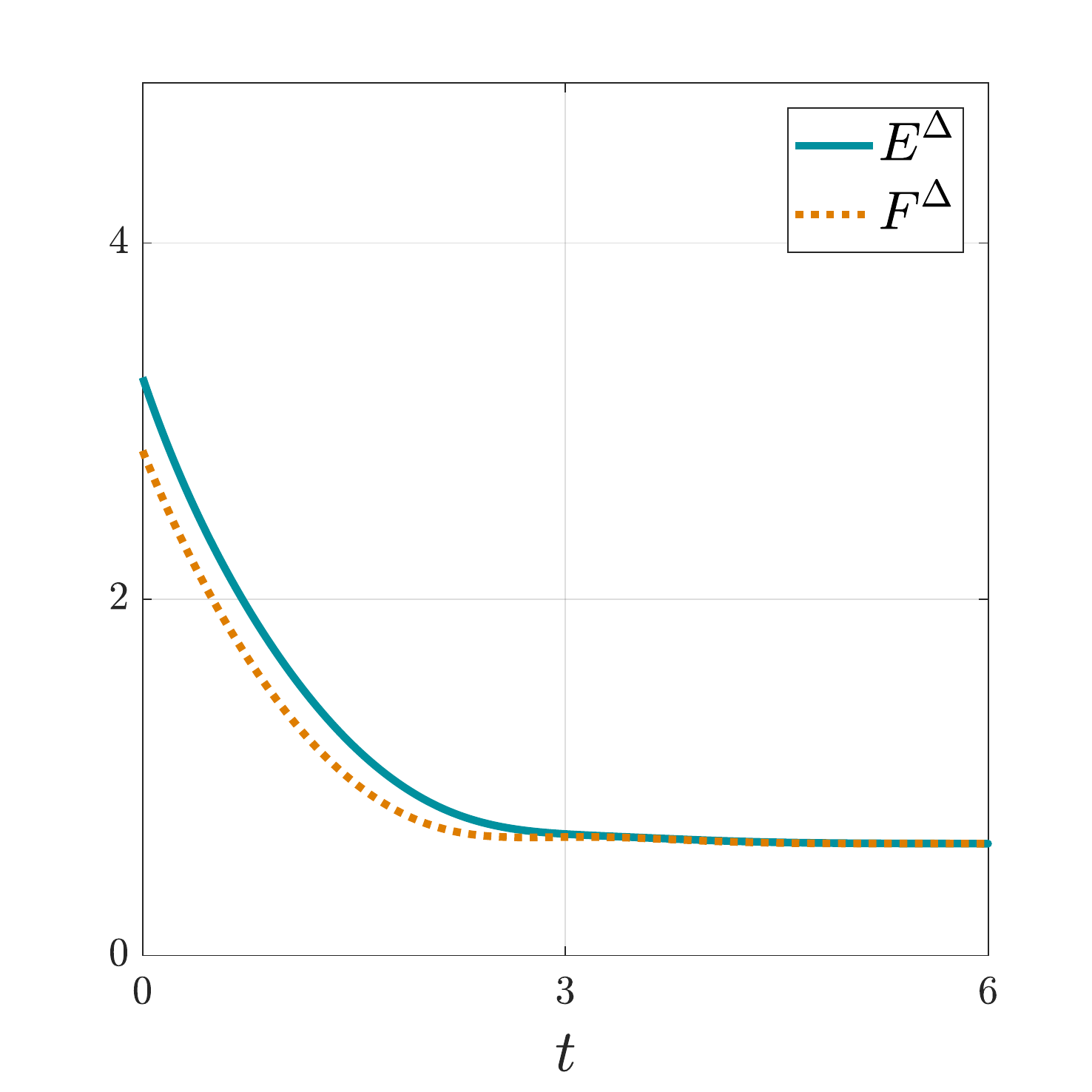}
}
\end{center}
\protect\protect\caption{\label{fig:p2well1_1} Temporal evolution with  single-well external potential and symmetric density in Example \ref{ex:idealdoublewell}.}
\end{figure}

For this example we aim to show that our numerical scheme accurately captures
the dry regions during the simulation and when reaching the steady states.
This is thanks to the combination of the positive-density reconstruction in
Appendix \ref{app:CWENO} and the choice of a kinetic numerical flux which is
able to handle vacuum regions \cite{perthame2001kinetic}. We show this by
conducting simulations with two different choices for the external potential
$V(x)$ with the following initial conditions for both cases
\begin{equation*}
\rho(x,t=0)=\frac{e^{-\lt(x-x_0\rt)^2/16}}{\int_\R e^{-\lt(x-x_0\rt)^2/16} dx}, \quad \rho u(x,t=0)=-0.1 \sin\lt(\frac{\pi x}{10}\rt),\quad x\in [-5,5],
\end{equation*}
with $x_0$ being the initial centre of mass. The steady states for the choice of pressure and external potentials of this example satisfy
\begin{equation}\label{eq:steadystatewell}
    \rho_{\infty}=(C(x)-V(x))_+,
\end{equation}
where $C(x)$ is a piecewise constant function being zero outside the support of the density.
\begin{figure}[ht!]
\begin{center}
\subfloat[Evolution of the density]{\protect\protect\includegraphics[scale=0.4]{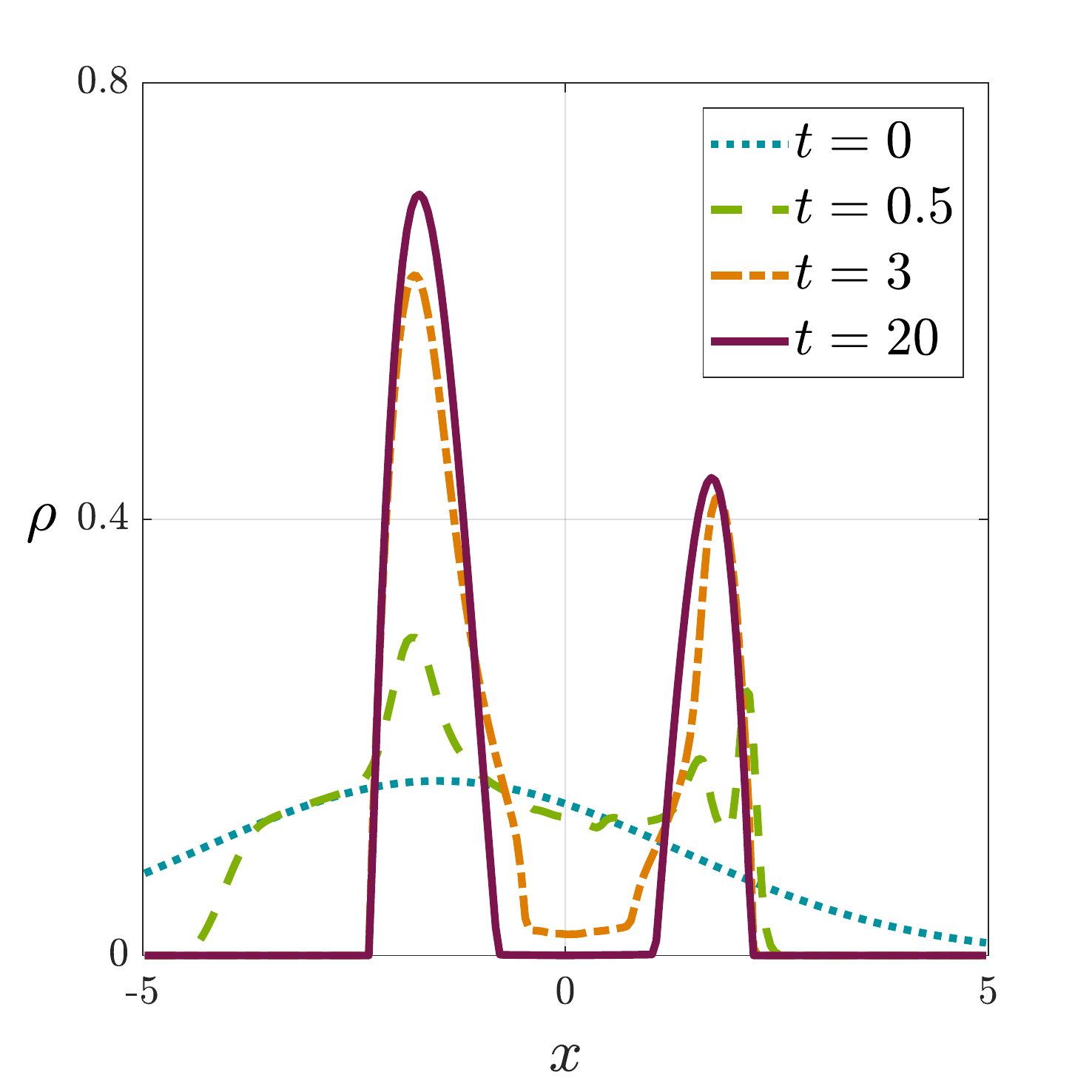}
}
\subfloat[Evolution of the momentum]{\protect\protect\includegraphics[scale=0.4]{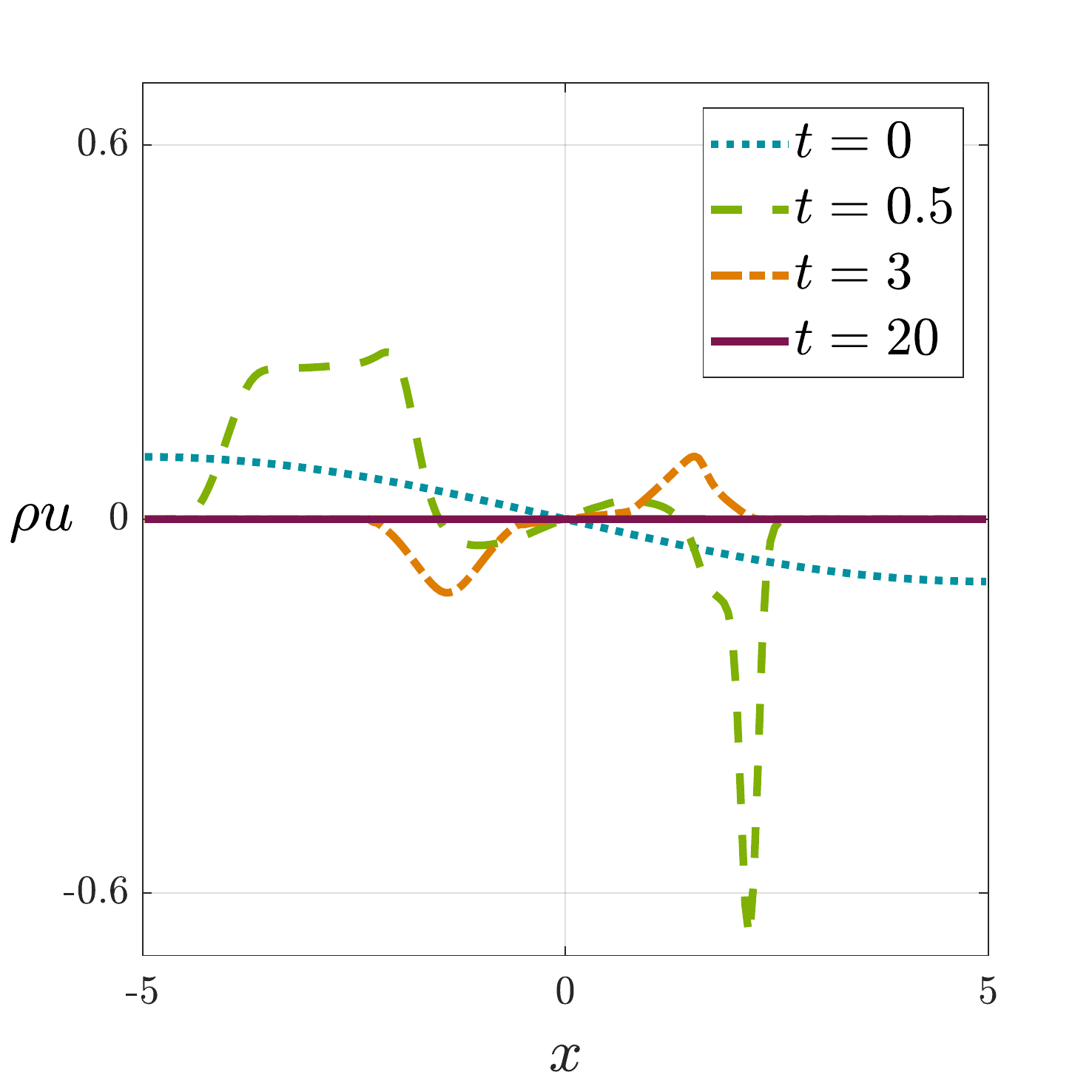}
}\\
\subfloat[Evolution of the variation of the free energy]{\protect\protect\includegraphics[scale=0.4]{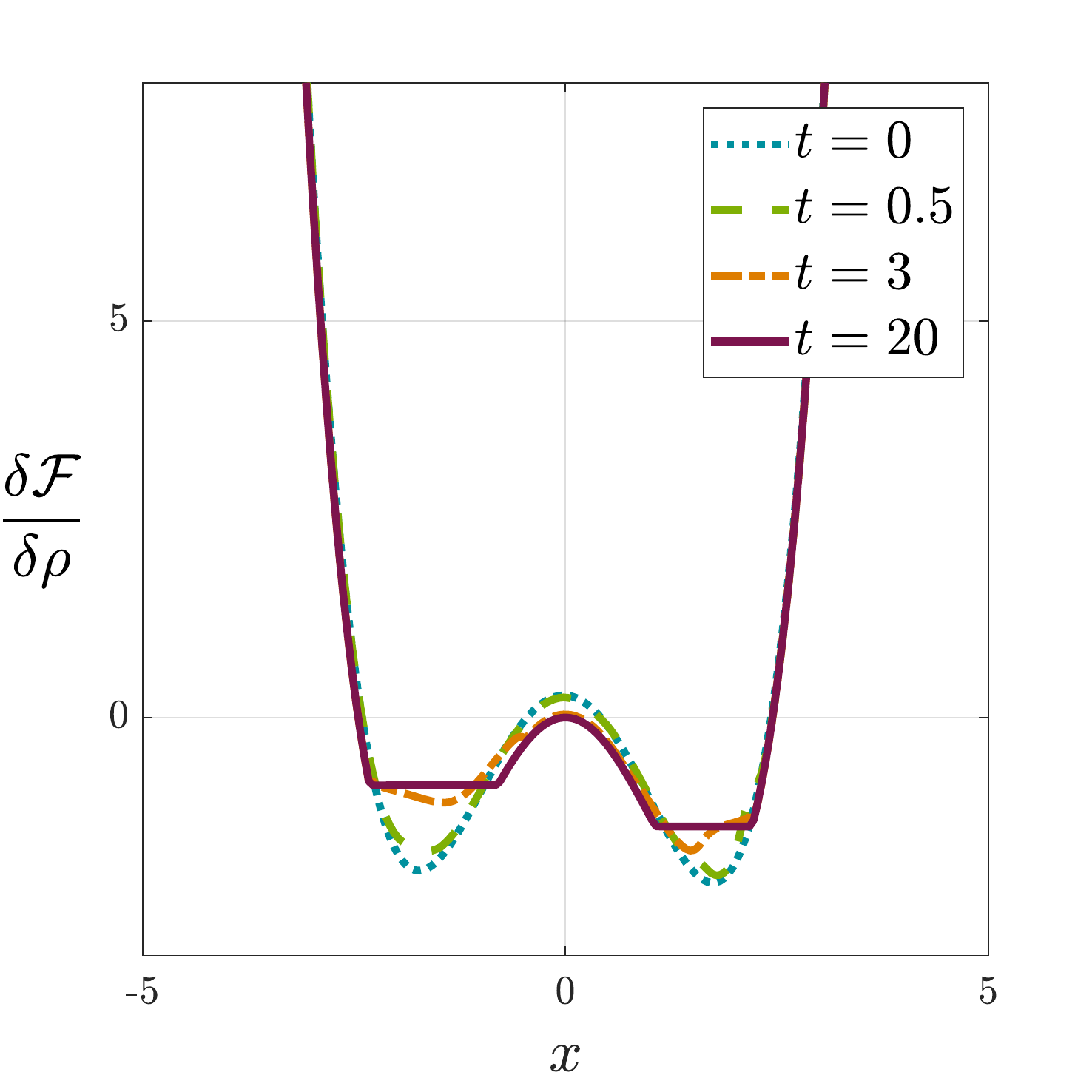}
}
\subfloat[Evolution of the total energy and free energy]{\protect\protect\includegraphics[scale=0.4]{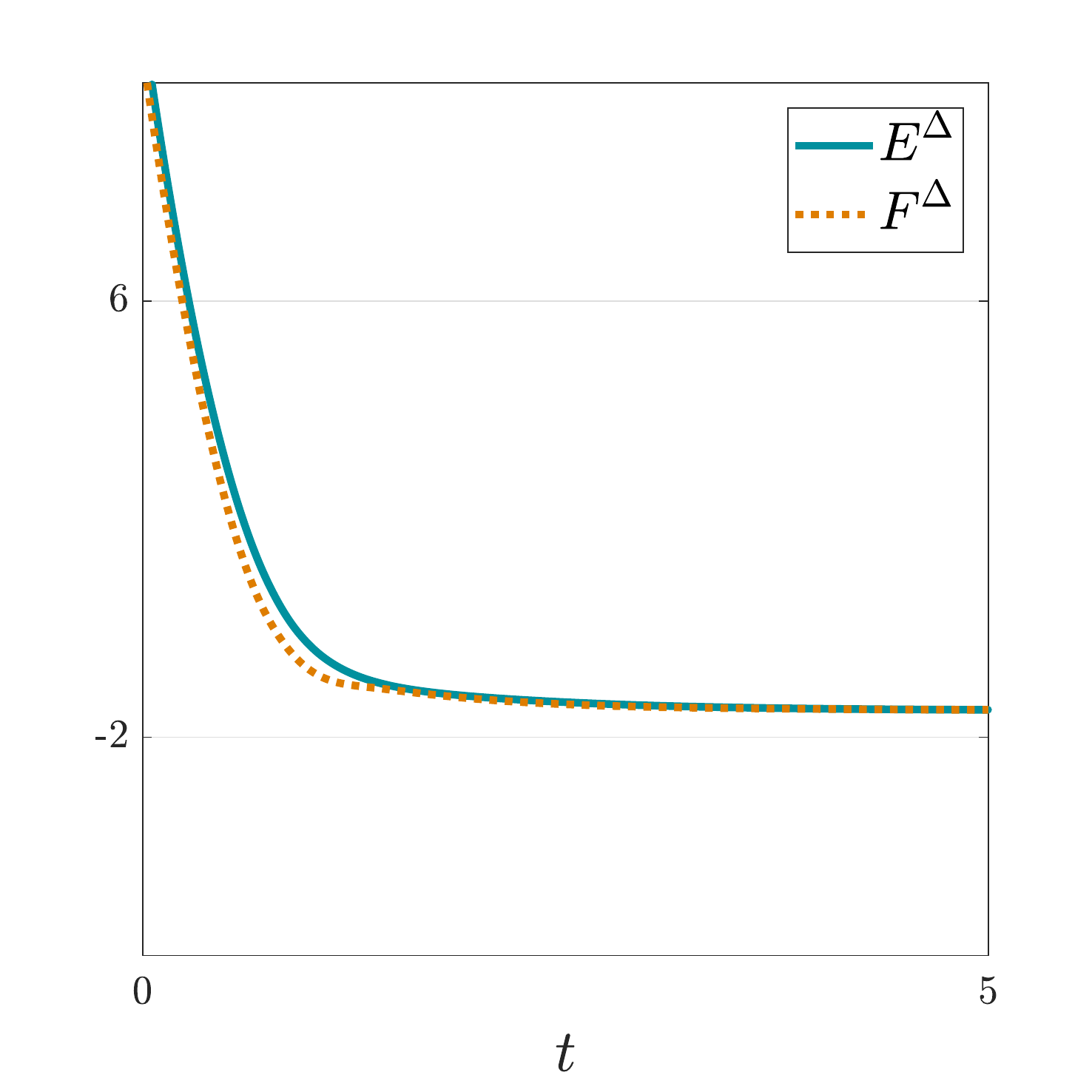}
}
\end{center}
\protect\protect\caption{\label{fig:p2well1_2} Temporal evolution with double-well external potential and asymmetric density in Example \ref{ex:idealdoublewell}.}
\end{figure}

The details of each simulation are:
\begin{enumerate}[label=\arabic*)]
    \item Single-well external potential and symmetric density:
        $V(x)=x^2/2$ and $x_0=0$. The results of this simulation are
        depicted in figure \ref{fig:p2well1_1}. Figure \ref{fig:p2well1_1}
        (a) shows the formation of the compact support of the density
        during the time evolution with the steady state taking the shape of
        a positive parabola and satisfying \eqref{eq:steadystatewell}. We
        also observe that the variation of the free energy in figure
        \ref{fig:p2well1_1} (c) reaches a constant value only in the
        support of the density, in agreement with the steady-state relation
        in \eqref{eq:steadyvarener}. We also note that in
        \ref{fig:p2well1_1} (d) the discrete total energy decreases in
        time, while the discrete free energy has a slight increase around
        $t=3$ due to an exchange of energy with the kinetic energy.

        \item Double-well external potential and asymmetric density:
            $V(x)=x^4/4-3\,x^2/2$ and $x_0=1.5$. The results of this
            simulation are depicted in figure \ref{fig:p2well1_2}. From the
            evolution of the density in figure \ref{fig:p2well1_2} (a) it
            is evident that two compactly-supported bumps of density are
            formed when reaching the steady state. This is due to the
            external potential having two wells. In addition, the mass in
            the bumps is not the same, since the initial density is not
            symmetric. It is also important to remark that, when reaching
            the steady state, the variation of the free energy in each
            compacted support of the density is constant but has different
            values. This is depicted in figure \ref{fig:p2well1_2} (c) and
            agrees with the steady state relation \eqref{eq:steadyvarener}.
            We refer the reader to our previous work
            \cite{carrillo2018wellbalanced} for similar simulations
            considering varied scenarios with double-well potentials.
\end{enumerate}
\end{examplecase}
%
%

\begin{examplecase}[Collective behaviour: comparison of linear, Cucker-Smale and Motsch-Tadmor dampings]\label{ex:damping}
In this example we explore the impact of adding linear and nonlinear damping
terms to the general system \eqref{eq:generalsys}. The motivation for the
nonlinear damping comes from the field of collective behaviour, in which a
large amount of interacting individuals or agents organize their dynamics by
influencing each other and without the presence of a leader. Most of the
literature in collective behaviour is based on individual based models (IBMs)
which are particle descriptions considering the three basic effects of
attraction, repulsion and alignment of the individuals. The combination of
these three effects has proven to be very versatile and extends beyond the
typical animal applications for schools of fish \cite{katz2011inferring},
herds of mammals \cite{giardina2008collective} or flocks or birds
\cite{hildenbrandt2010self}. Indeed, these models are now playing a critical
role in understanding complex phenomena including consensus and
spatio-temporal patterns in diverse problems ranging from the evolution of
human languages \cite{cucker2004modeling} to the prediction of criminal
behaviour \cite{short2008statistical} and space flight formation
\cite{perea2009extension}.

There are plenty of works in the literature addressing the mean-field
derivation of kinetic and other macroscopic models from the original particle
descriptions \cite{carrillo2010asymptotic, carrillo2010particle,ha2008from}.
These derived hydrodynamic equations agree with our general system
\eqref{eq:generalsys} and model the attraction and repulsion effects via the
interaction potential $W(x)$. The third effect for collective behaviour is
alignment which in our system \eqref{eq:generalsys} it is achieved by means
of the nonlinear and nonlocal damping of the RHS of the momentum equation.
The most popular approach for the velocity consensus is the Cucker-Smale (CS)
model \cite{cucker2007emergent,cucker2007mathematics} which adapts the
momentum of a particle depending on the momentum and distance of the other
particles. Several authors have proposed refined variations of the CS model,
and among them we remark the weighted-normalized model by Motsch and Tadmor
(MS) \cite{motsch2011new} (which will be referred to in the following as the
MS model). It basically corrects the CS model by eliminating the
normalization over the total number of agents, which leads to inaccurate
behaviours in far-from-equilibrium scenarios. Instead, the MT model
introduces the concept of relative distances between agents with the cost,
however, of destroying the symmetry of the original CS model. For further
details on flocking and alignment with the CS and related models we refer the
reader to \cite{carrillo2016critical,
carrillo2017review,minakowski2019singular,choi2017emergent}.

The objective of this example is to illustrate the differences of adding to
the general system \eqref{eq:generalsys} linear damping, the CS or the MT
model. The damping term for each of them is
\begin{equation}\label{eq:dampings}
    \begin{cases}
    -\gamma\rho \bm{u}& \text{if linear damping,}\\[6pt]
    -\rho\displaystyle\int_{\R^{d}}\psi(\bm{x}-\bm{y})(\bm{u}(\bm{x})-\bm{u}(\bm{y}))\rho(\bm{y})\,d\bm{y}& \text{if Cucker-Smale damping,}\\[12pt]
    -\frac{\rho}{\psi \star \rho}\displaystyle \int_{\R^{d}}\psi(\bm{x}-\bm{y})(\bm{u}(\bm{x})-\bm{u}(\bm{y}))\rho(\bm{y})\,d\bm{y}& \text{if Motsch-Tadmor damping,}\\[6pt]
    \end{cases}
\end{equation}
where $\psi(x)$ is a nonnegative symmetric smooth function, called the
communication function, satisfying for this example
\begin{equation*}
\psi(x)=\frac{1}{\left(1+|x|^2\right)^\frac{1}{4}}.
\end{equation*}
It should be noted that the CS damping term in \eqref{eq:dampings} would
reduce to linear damping if the communication function $\psi(x)$ was a
constant function $\psi(x)=1$. In addition, the difference between the CS and
MT models is the normalization over $\psi \star \rho$ that is added to the MT
model to ensure that the damping term is independent of the total mass of the
system.

The simulation for this example is chosen to specifically address a
particular drawback of the CS model. This occurs in the evolution of two
groups of agents separated by a certain significant distance and whose masses
have different degrees of magnitude. What happens with the CS model is that
the damping term for the small group of agents is negligible due to the
normalization over the total number of agents in the system. This means that
those agents do not seek alignment from the beginning of the simulation, and
as a result the convergence towards alignment is delayed. On the contrary,
with the MT model the normalization over $\psi \star \rho$ in
\eqref{eq:dampings} allows to take the relative distances between the agents
into account, and the small group of agents reacts much faster to the effect
of the rest of agents. In the simulation we also add a Morse-like interaction
potential \cite{carrillo2014explicit,carrillo2015finite} of the form
$W(x)=-e^{-|x|^2/2}/\sqrt{2\pi}$, which quickly decays at large distances and
does not add any attraction between the two groups of agents. Note that we
are forced to add this attraction term to balance the pressure and thus allow
for our well-balanced scheme in Section \ref{sec:numsch}.

This configuration is depicted in figure \ref{fig:damping}. Specifically,
figure \ref{fig:damping} (a) and (b) shows with blue the initial conditions
for the density and the momentum. On the one hand, in the density there are
two groups of agents with mass of $0.9$ and $0.1$, satisfying
\begin{equation*}
\rho(x,t=0)=0.9\,\frac{e^{-\lt(x+1\rt)^2/2}}{\int_\R e^{-\lt(x+1\rt)^2/2} dx}+0.1\,\frac{e^{-\lt(x-11\rt)^2}}{\int_\R e^{-\lt(x-11\rt)^2} dx},\quad x\in [-5,14],
\end{equation*}
while on the other hand for the momentum the two groups have opposite velocity signs, in agreement with
\begin{equation*}
\rho u(x,t=0)=\begin{cases}
\,\,\,\;2\,\rho(x,t=0) & \text{if }x<5,\\
-2\,\rho(x,t=0) & \text{if }x\geq 5.
\end{cases}
\end{equation*}

\begin{figure}[ht!]
\begin{center}
\subfloat[Evolution of the density]{\protect\protect\includegraphics[scale=0.4]{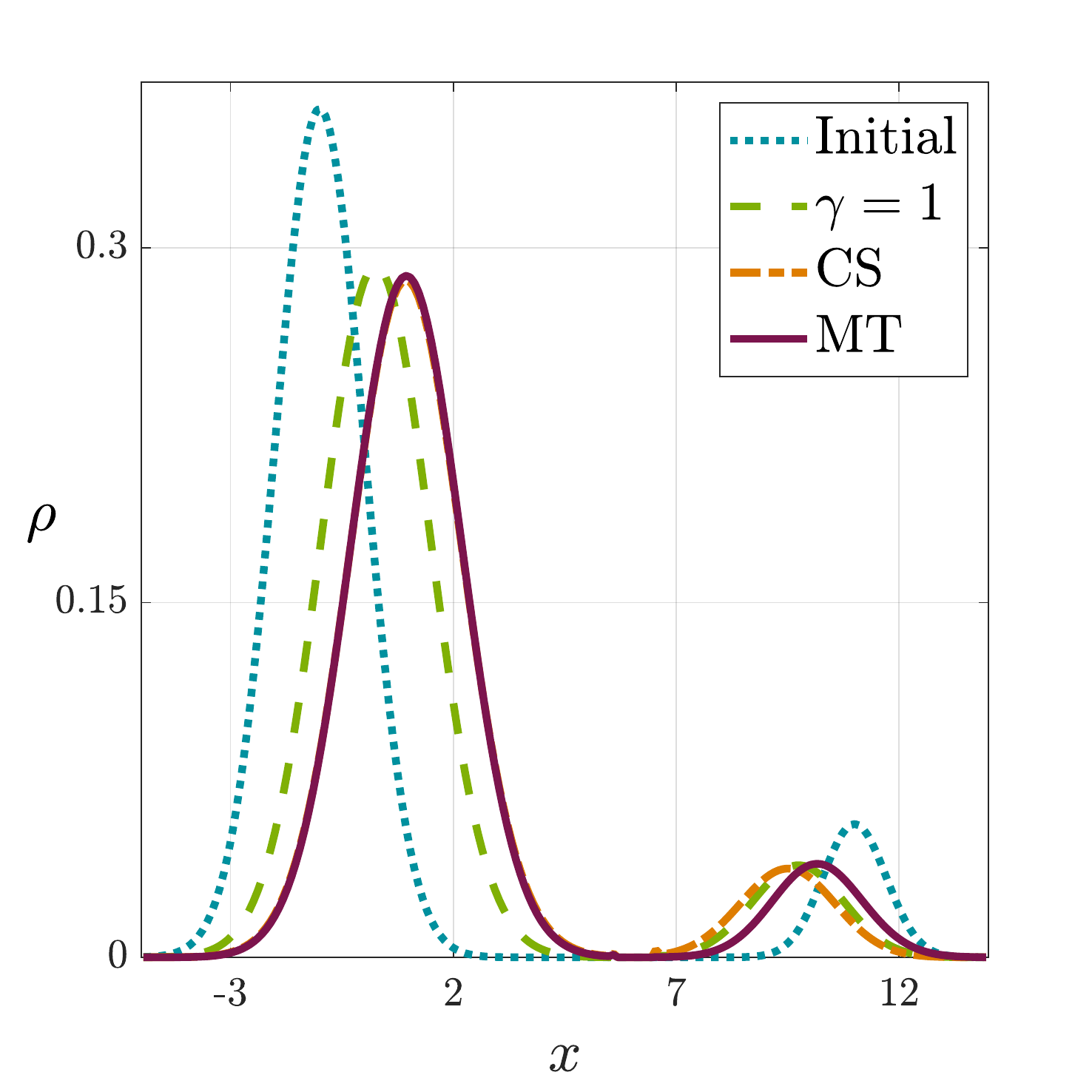}
\llap{\shortstack{%
        \includegraphics[scale=.14,trim={0 0 1.4cm 0cm},clip]{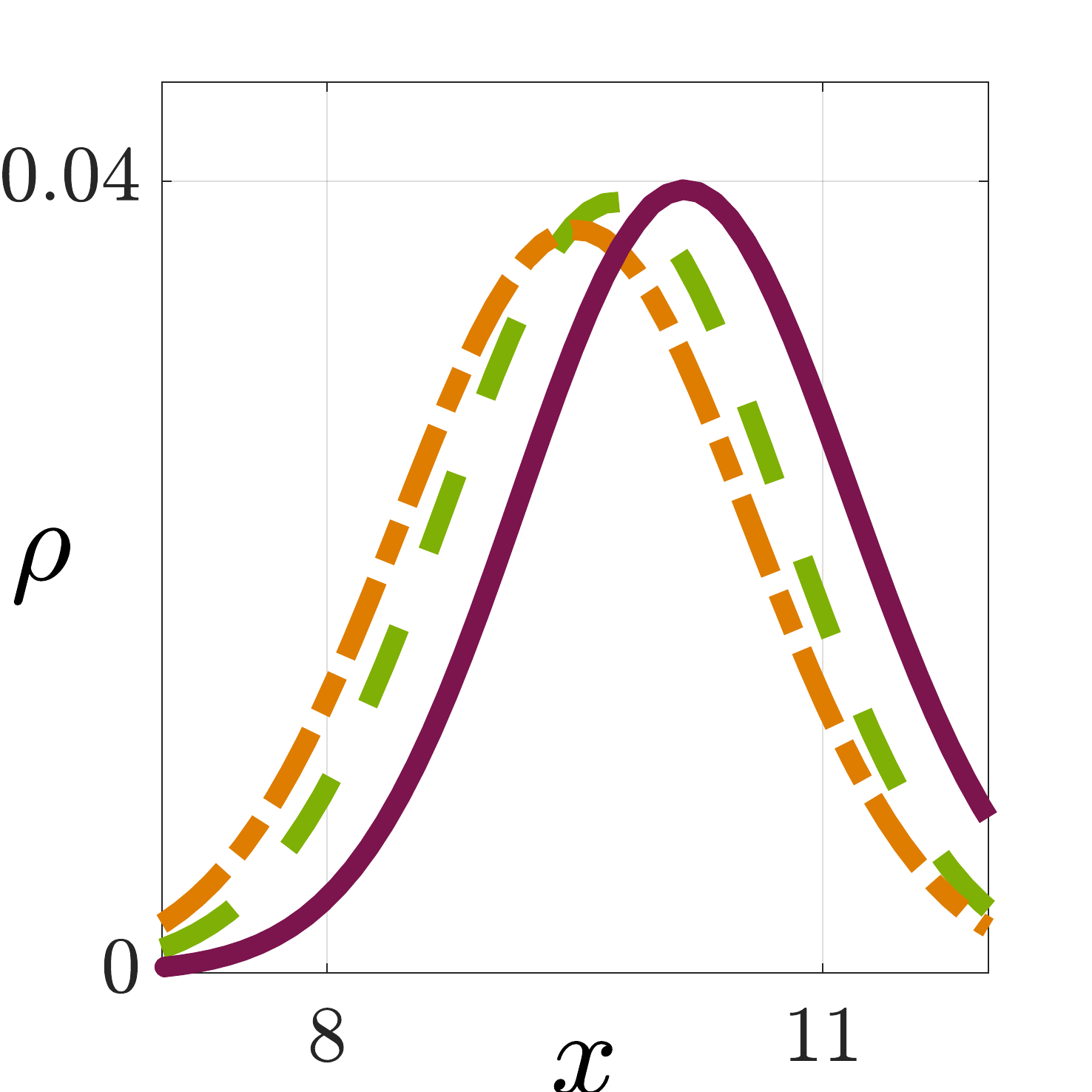}\\
        \rule{0ex}{0.65in}%
      }
  \rule{0.268in}{0ex}}
}
\subfloat[Evolution of the momentum]{\protect\protect\includegraphics[scale=0.4]{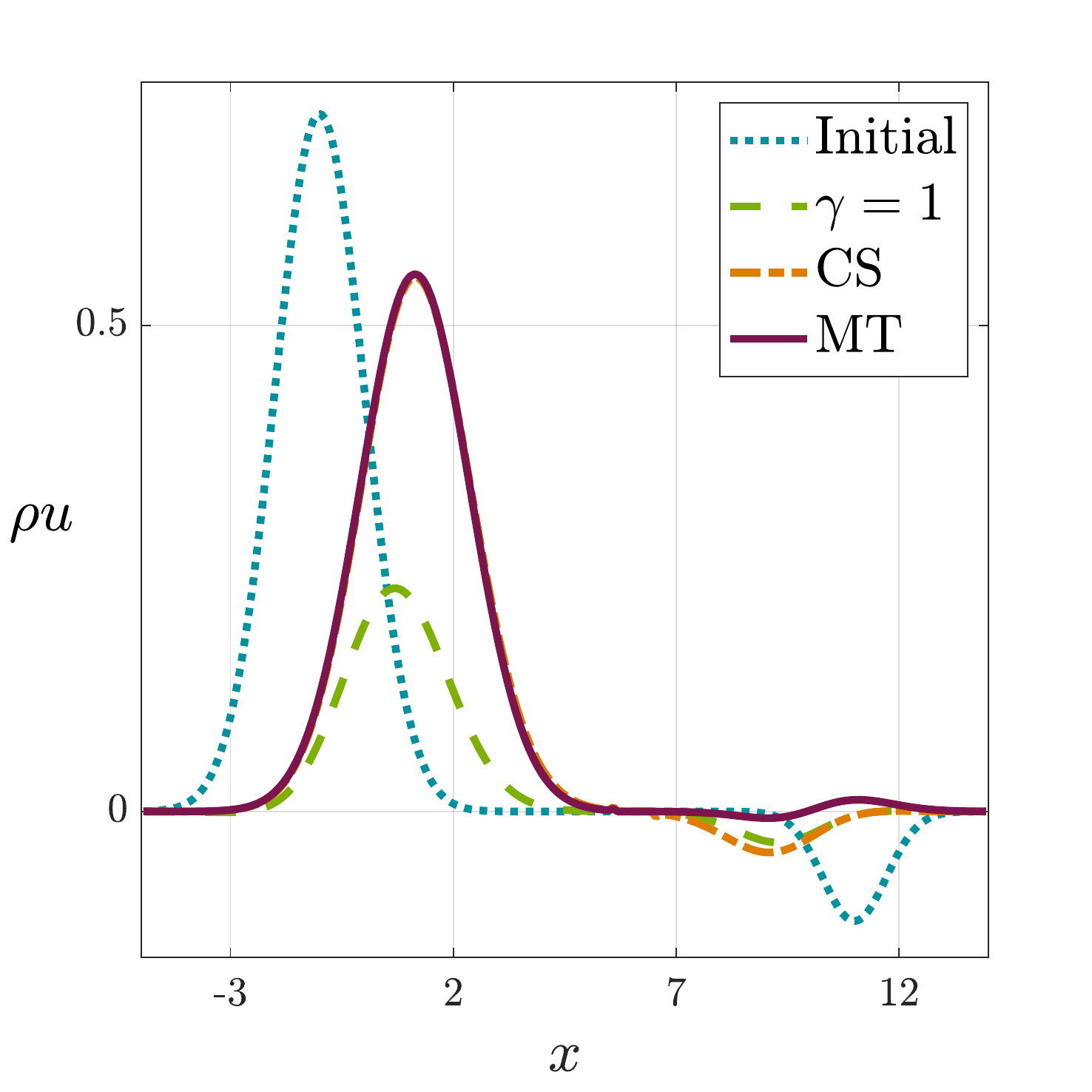}
\llap{\shortstack{%
        \includegraphics[scale=.14,trim={0 0 1.4cm 0cm},clip]{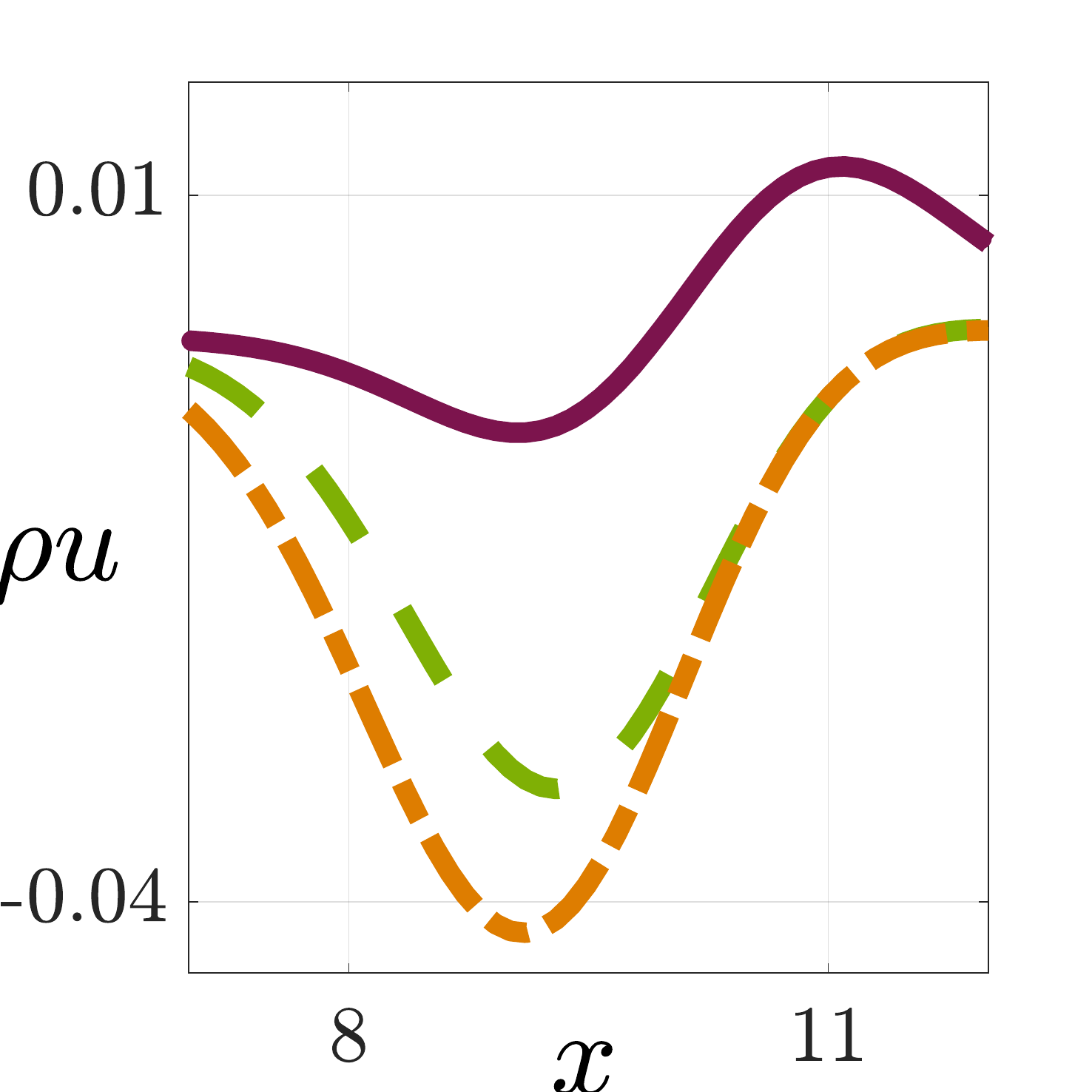}\\
        \rule{0ex}{0.65in}%
      }
  \rule{0.268in}{0ex}}
}\\
\subfloat[Evolution of the discrete free energy]{\protect\protect\includegraphics[scale=0.4]{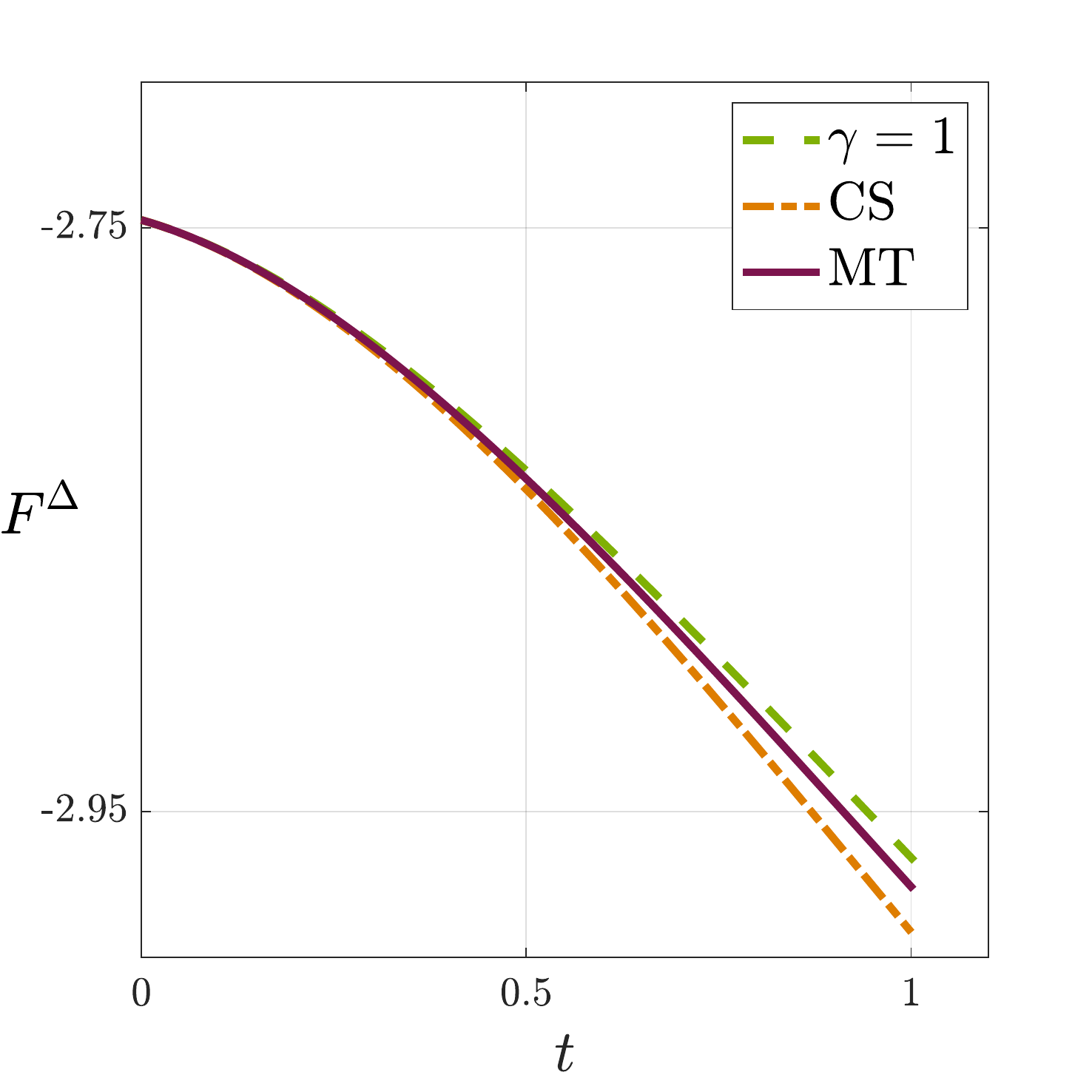}
}
\subfloat[Evolution of the discrete total energy]{\protect\protect\includegraphics[scale=0.4]{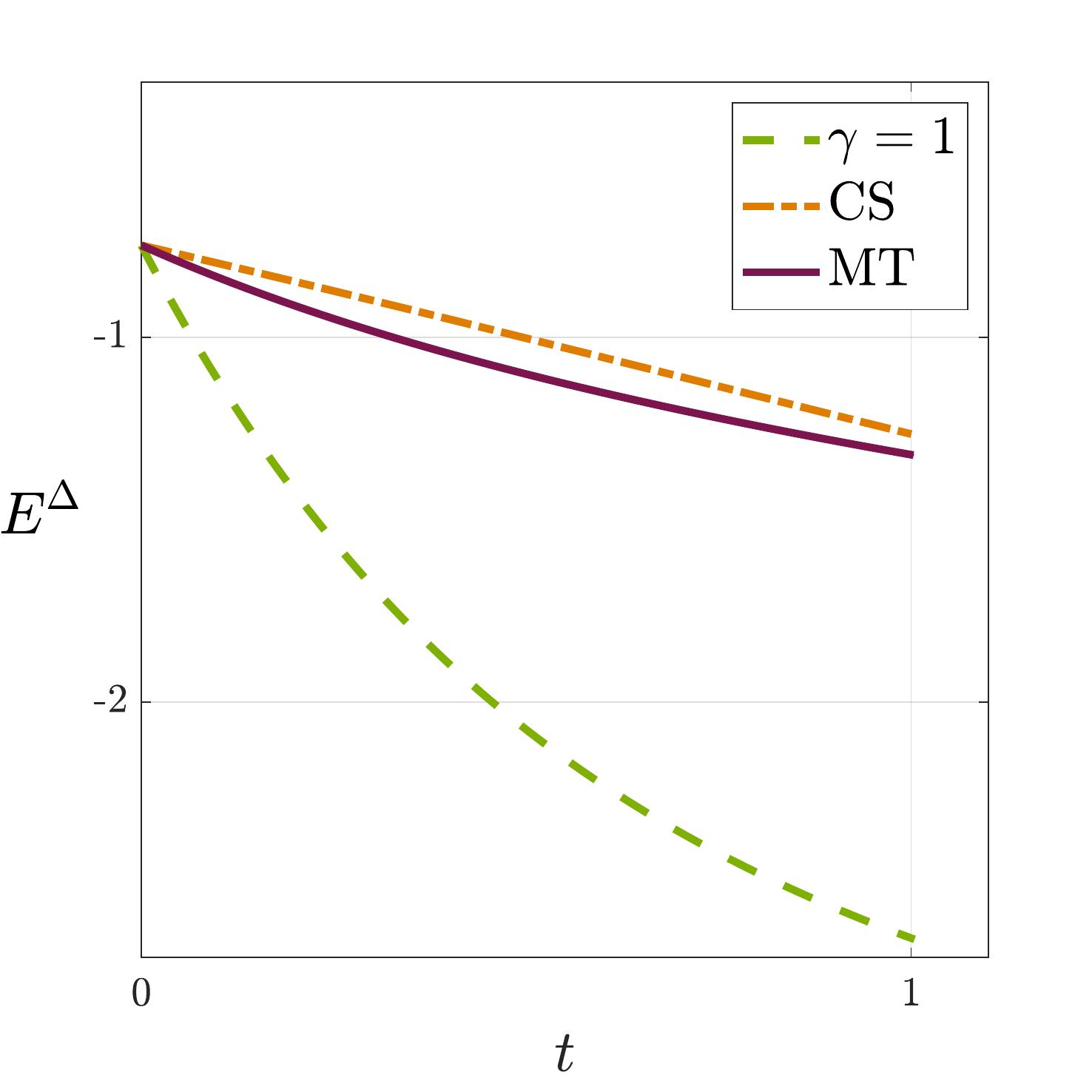}
}

\end{center}
\protect\protect\caption{\label{fig:damping} Simulation of Example  \ref{ex:damping} until $t=1$. $\gamma=1$ denotes the linear damping simulation, CS the Cucker-Smale simulation and MT the Motsch-tadmor simulation.}
\end{figure}

What we expect to happen in this situation is that the large group imposes
its velocity sign over the small group, so eventually all the agents align
with positive velocity. From the momentum simulation in figure
\ref{fig:damping} (b) we observe that after $t=1$ the MT model has already
changed the velocity sign of the small group from negative to positive, while
for the CS model the velocity is still negative. In general, linear damping
is the one that dissipates more momentum, as depicted in the momentum plot of
figure \ref{fig:damping} (b) and in the total energy plot of figure
\ref{fig:damping} (d). The free energy and total energy decay are similar for
both the CS and MT models. A similar numerical experiment was already conducted in \cite{carrillo2017review} using particle methods.

\end{examplecase}

%
%
\begin{examplecase}[Hydrodynamic Keller-Segel system]\label{ex:KS}
The Keller-Segel model has been widely employed for chemotactic aggregation
of biological populations such as cells, bacteria or insects. It models how
the production of a particular chemical by these organisms leads to
long-range attraction and eventually results in self-organization. Its first
formulation was proposed in \cite{keller1970initiation} and consisted in a
drift-diffusion equation for the density (which is obtained in the overdamped
limit of our system \eqref{eq:generalsys}) coupled with a diffusion equation
for the chemical concentration.

In this example we are interested in the hydrodynamic extension of the
Keller-Segel model proposed in \cite{chavanis2007kinetic} and takes into
account the inertia of the biological entities and has been proposed in
\cite{chavanis2007kinetic}. It follows the same structure as the generalized
Euler-Poisson system in the example \ref{ex:idealker} with the free energy
satisfying \eqref{eq:freeenergyPos} and the chemical concentration usually
taken as $S=W(x) \star \rho$ (see
\cite{carrillo2018ground,calvez2017geometry}). The homogeneous kernel $W(x)$
follows $W(x)=|x|^\alpha/\alpha$, where $\alpha>-1$ and $W(x)=\ln|x|$ when
$\alpha=0$ for convention. The difference with example \ref{ex:idealker} is
that here the pressure follows $P(\rho)=\rho^m$ with $m\geq1$, thus allowing
for compactly-supported steady states and vacuum in the density if $m>1$. We
refer the reader to \cite{bellomo2015toward} for more information about the
Keller-Segel model and the diffusion equation for the chemical concentration.

\begin{figure}[ht!]
\begin{center}
\subfloat[Evolution of the density]{\protect\protect\includegraphics[scale=0.4]{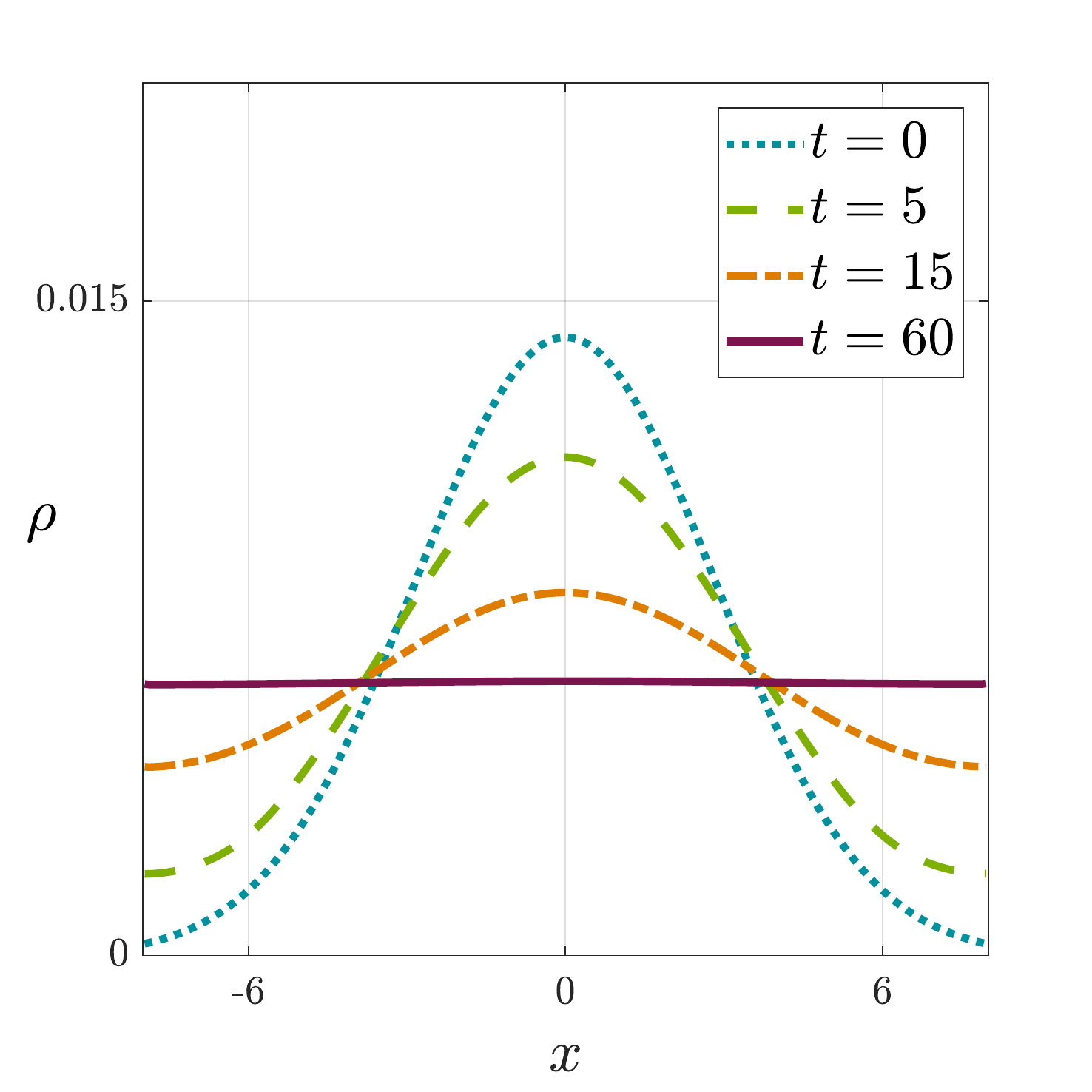}
}
\subfloat[Evolution of the momentum]{\protect\protect\includegraphics[scale=0.4]{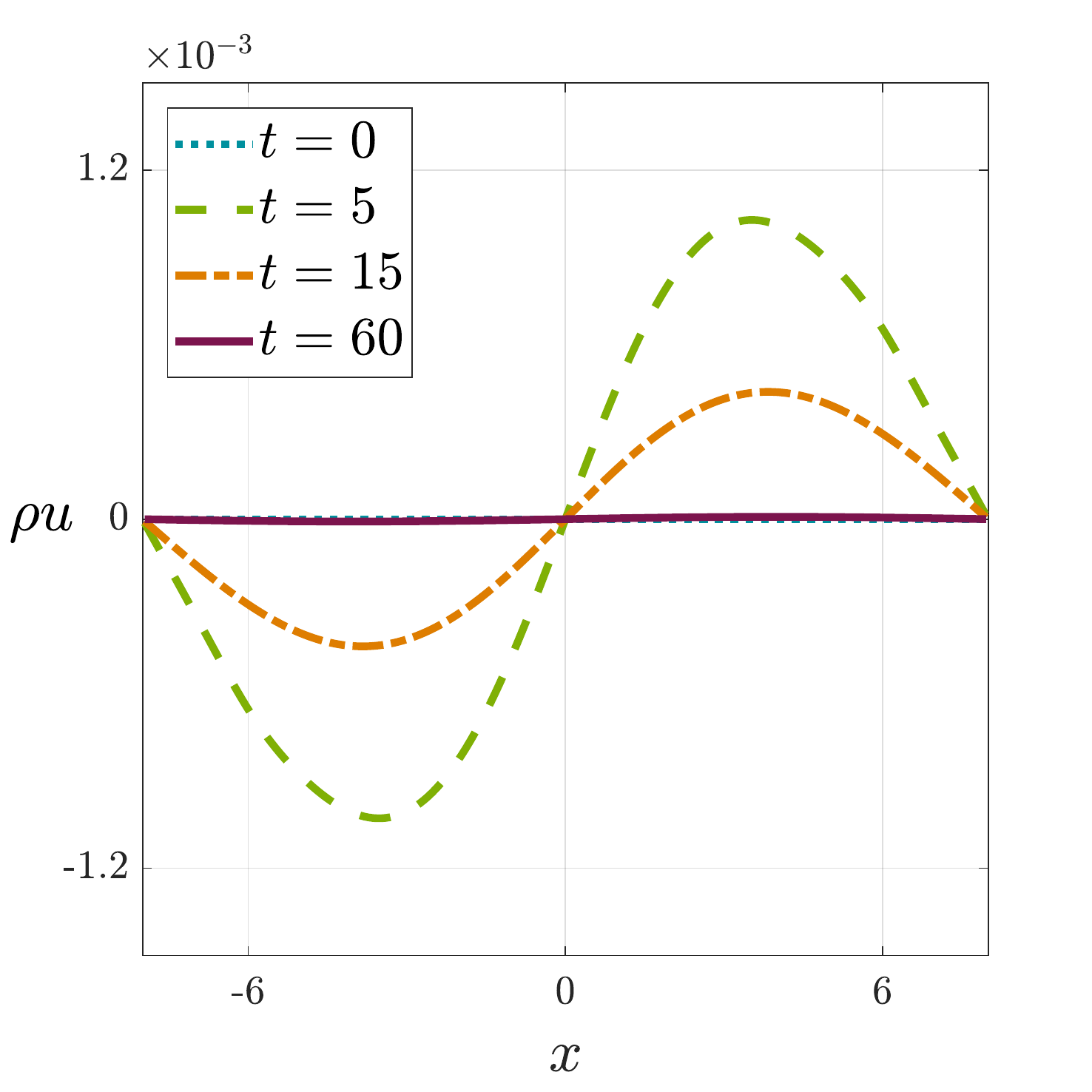}
}\\
\subfloat[Evolution of the variation of the free energy]{\protect\protect\includegraphics[scale=0.4]{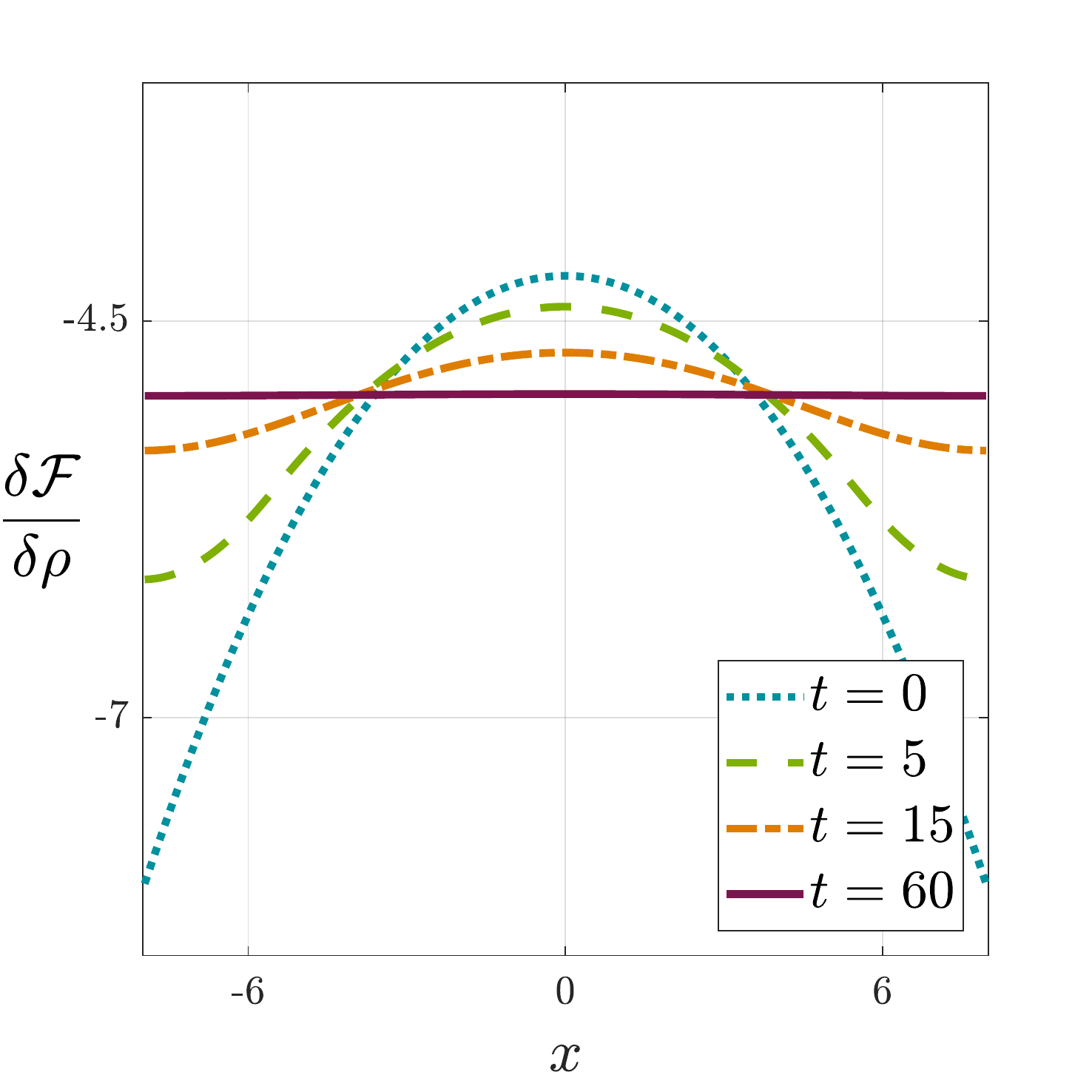}
}
\subfloat[Evolution of the total energy and free energy]{\protect\protect\includegraphics[scale=0.4]{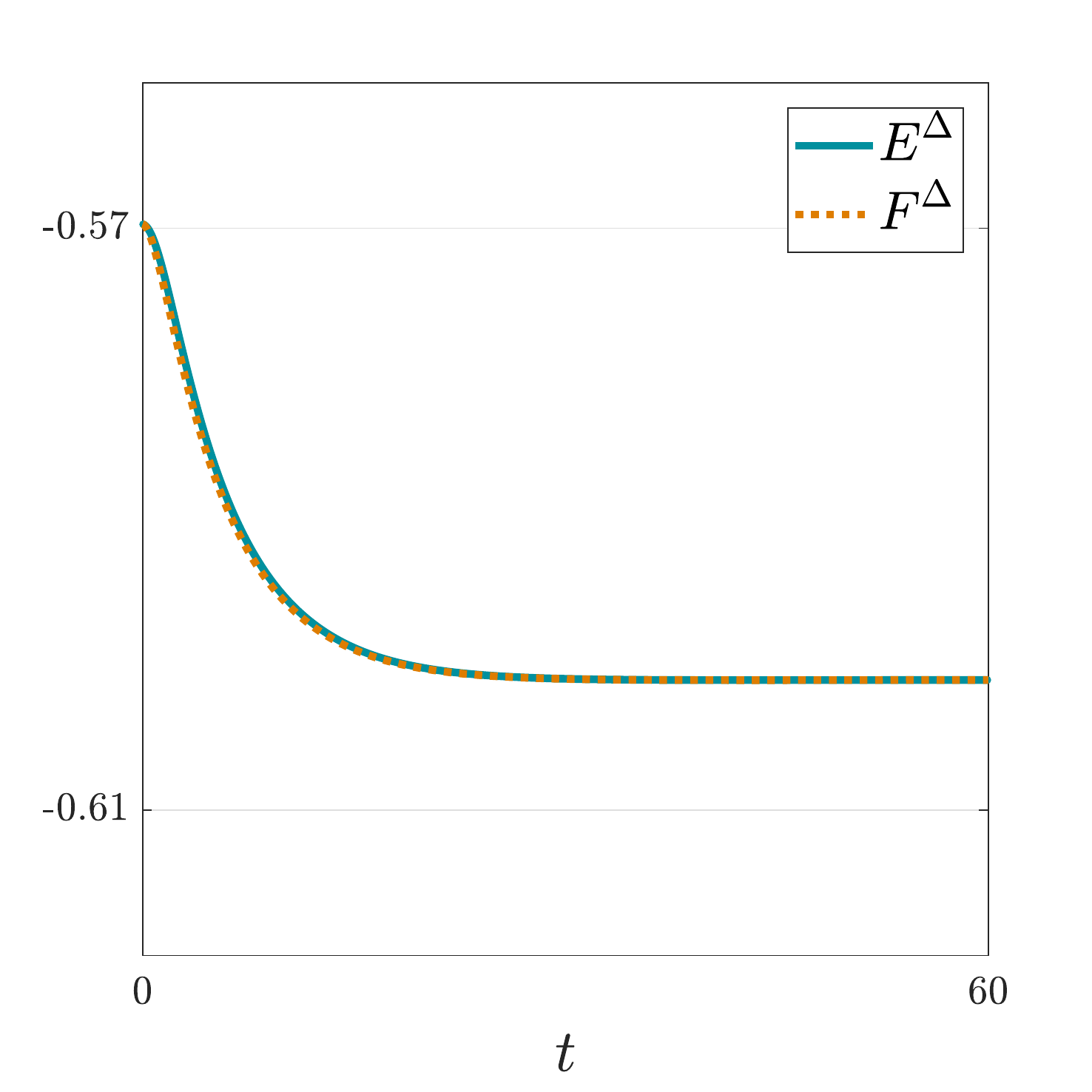}
}
\end{center}
\protect\protect\caption{\label{fig:KS_1} Temporal evolution of the Keller-Segel system with $P=\rho$, $W(x)=\ln|x|$ and initial conditions \eqref{eq:iniKS} with $M=0.1$. Global-in-time solution.}
\end{figure}

In our previous work \cite{carrillo2018wellbalanced} we applied our first-
and second-order well-balanced scheme to investigate the competition between
the attraction from the local kernel $W(x)$ and the repulsion caused by the
diffusion of the pressure $P(\rho)$. For the overdamped Keller-Segel model
there are basically three possible regimes \cite{calvez2017equilibria,
calvez2017geometry}, which result from adequately tuning the parameters
$\alpha$ in the kernel $W(x)$ and $m$ in the pressure $P(\rho)$: diffusion
dominated regime ($m > 1 -\alpha$), balanced regime ($m = 1 -\alpha$) where a
critical mass separates self-similar and blow-up behaviour, and
aggregation-dominated regime ($m < 1 -\alpha$). Results with the momentum
equation included, and thus inertia, are still quite limited in the
literature, with only some specific scenarios studied
\cite{carrillo2016critical, carrillo2018longtime}. In our previous work
\cite{carrillo2018wellbalanced} we investigated the role of inertia for a
choice of parameters of $\alpha=0.5$, $m = 1.5$ and $\alpha=-0.5$, $m = 1.3$,
which led to a diffusion-dominated and aggregation-dominated regimes,
respectively.

For this example we aim to explore the case of $\alpha=0$ which leads to the
singular potential $W(x)=\ln|x|$. Initially, for the two first simulations of
this example we set $m=1$ so that $P(\rho)=\rho$. In the overdamped limit
this scenario corresponds to the balanced regime since $m = 1 -\alpha$, and
there is a critical mass separating the global-in-time from the finite-time
blowup solution. We aim to run two simulations with identical initial
conditions which differ only in a multiplicative constant for the density
which allows to set a different mass of the system. The objective is to find
global-in-time and finite-time blowup solutions by only changing the mass of
the system. With that aim we set the initial conditions as
\begin{equation}\label{eq:iniKS}
\rho(x,t=0)=M\frac{e^{-\lt(x\rt)^2/16}}{\int_\R e^{-\lt(x\rt)^2/16} dx}, \quad \rho u(x,t=0)=0,\quad x\in [-8,8],
\end{equation}
with $M$ being the mass of the system.

For the first simulation we set the mass of the system to be $M=0.1$. The
results are shown in figure \ref{fig:KS_1} where the numerical solution is
clearly global-in-time and diffusion-dominated. Eventually the steady state
is reached,
\begin{equation*}
\frac{\delta \mathcal{F}}{\delta \rho}=\Pi'(\rho)+H(x,\rho)=\ln(\rho)+\ln|x|\star \rho=\text{constant}\ \text{on}\ \mathrm{supp}(\rho)\ \text{and}\ u=0.
\end{equation*}
From figure \ref{fig:KS_1} (a) we observe that the solution is completely
diffused and the final profile for the density is uniform. From figure
\ref{fig:KS_1} (c) we notice that the variation of the free energy with
respect to the density is constant once the steady state is reached, and from
figure \ref{fig:KS_1} (d) we remark how the total and free energy decay
during the temporal evolution.

\begin{figure}[ht!]
\begin{center}
\subfloat[Evolution of the density]{\protect\protect\includegraphics[scale=0.4]{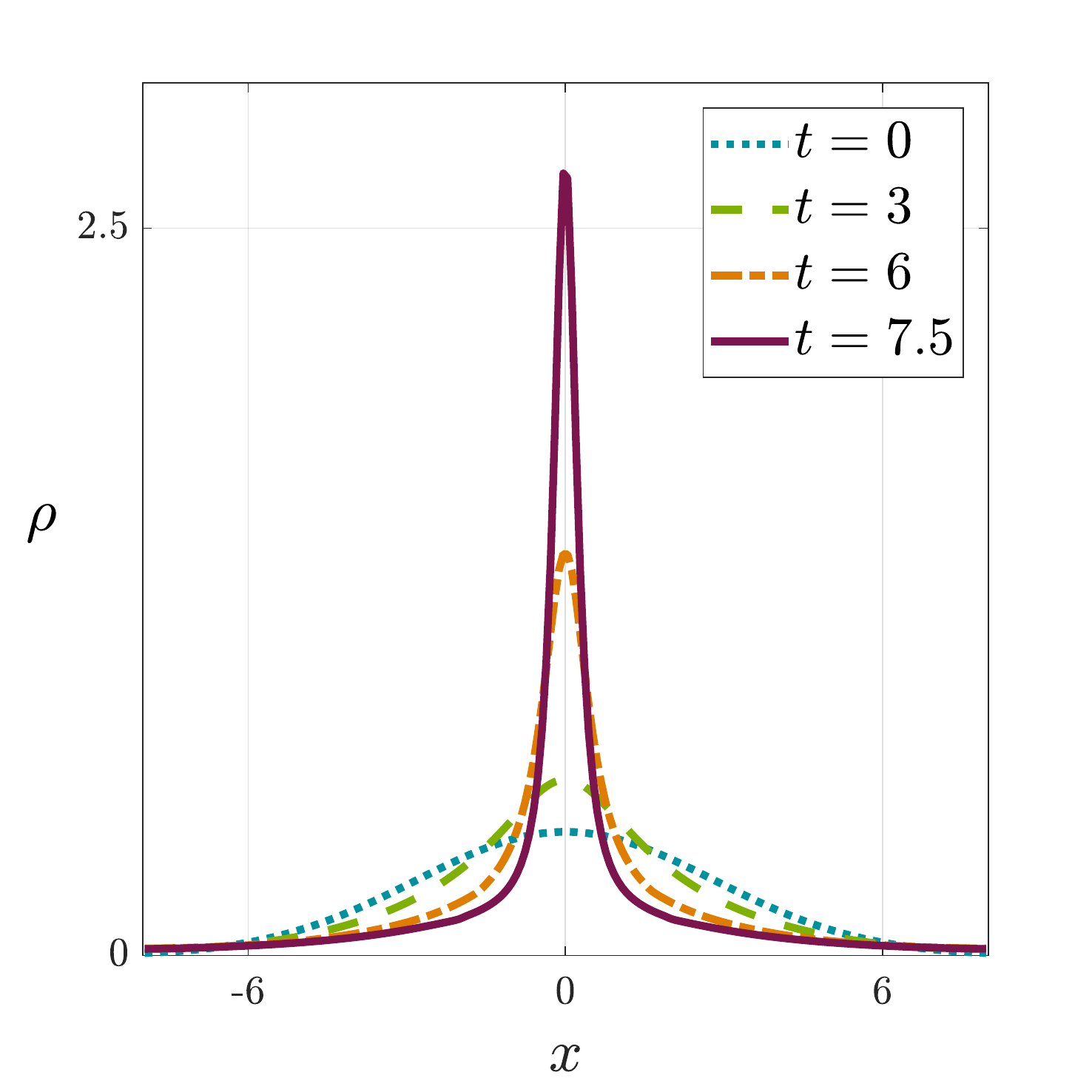}
}
\subfloat[Evolution of the momentum]{\protect\protect\includegraphics[scale=0.4]{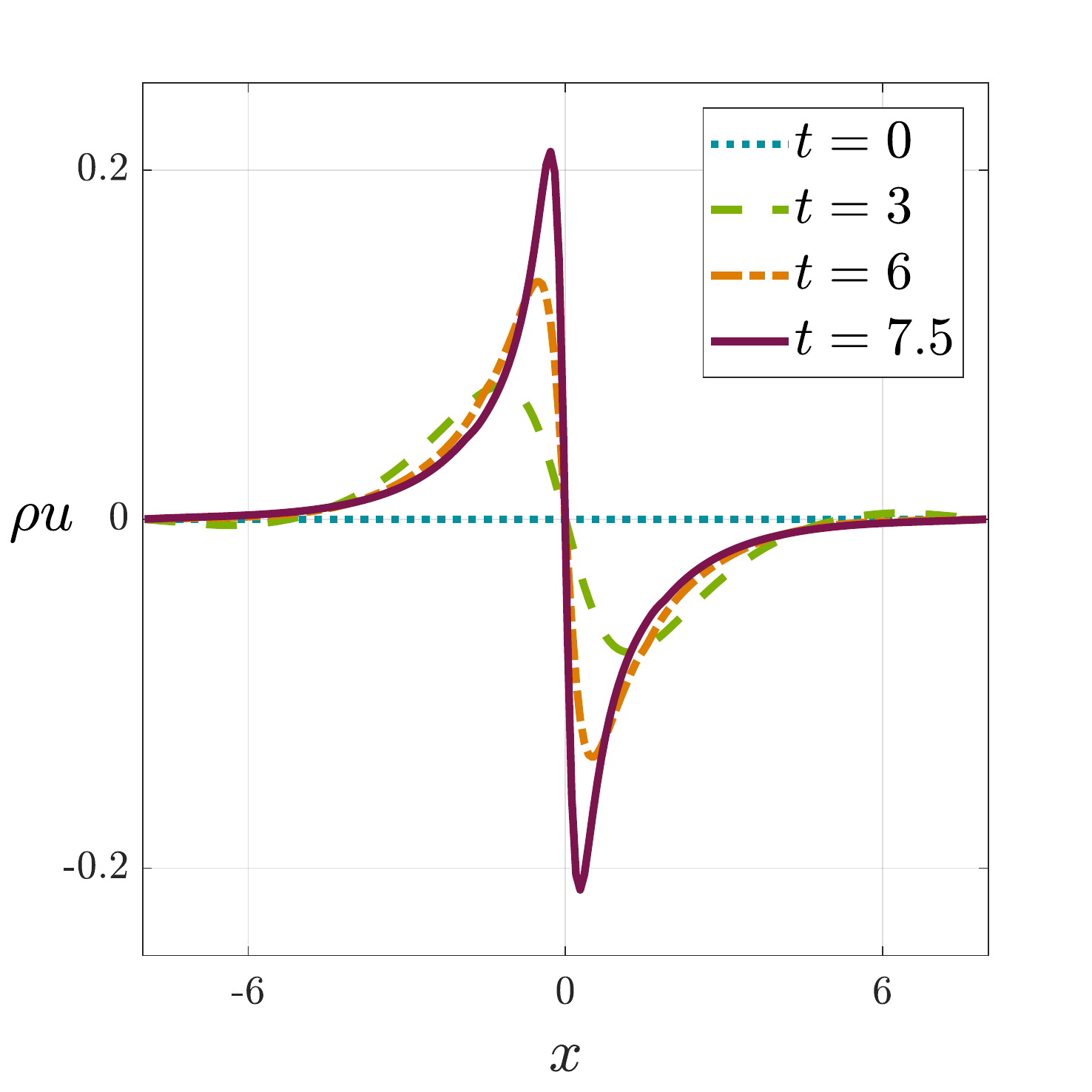}
}\\
\subfloat[Evolution of the variation of the free energy]{\protect\protect\includegraphics[scale=0.4]{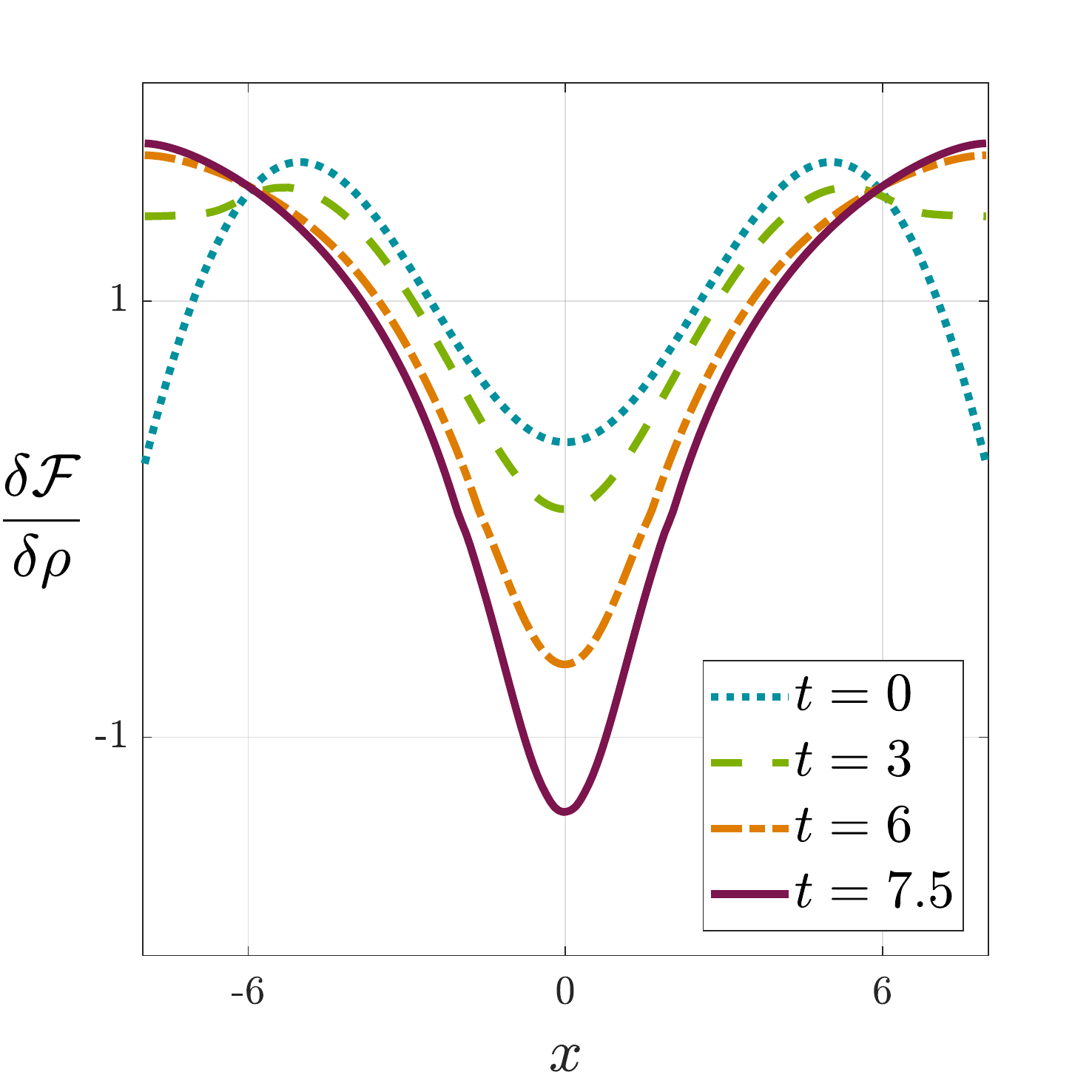}
}
\subfloat[Evolution of the total energy and free energy]{\protect\protect\includegraphics[scale=0.4]{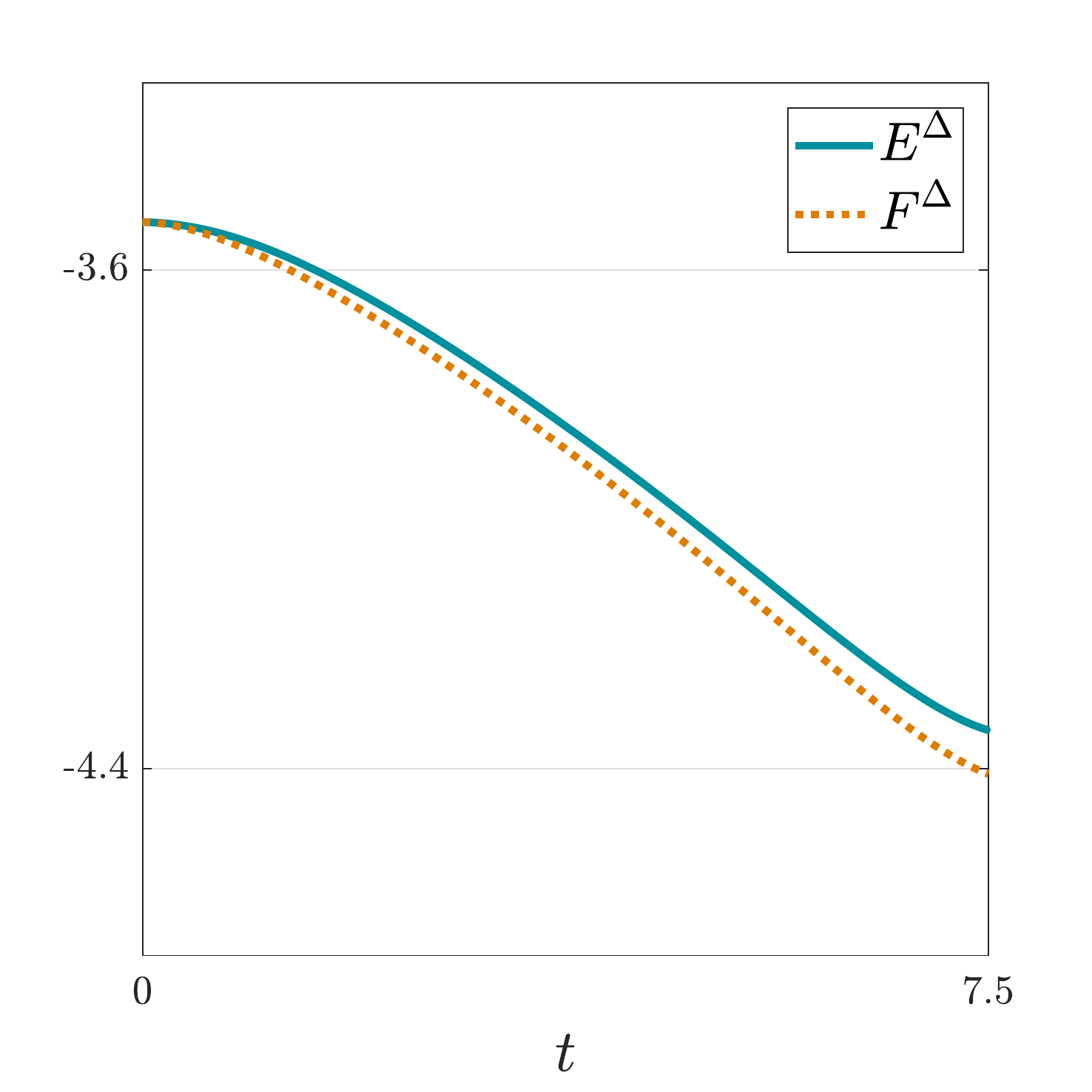}
}
\end{center}
\protect\protect\caption{\label{fig:KS_2} Temporal evolution of the Keller-Segel system with $P=\rho$, $W(x)=\ln|x|$ and initial conditions \eqref{eq:iniKS} with $M=3$. Finite-time blowup solution.}
\end{figure}

For the second simulation we select a mass of $M=3$ while keeping the same
initial conditions as in \eqref{eq:iniKS}. As displayed in figure
\ref{fig:KS_2}, the solution now presents a finite-time blowup around
$t\approx 7.5$, leading to an aggregation-dominated behaviour. Figure
\ref{fig:KS_2} (a) reveals that the density is concentrated towards the
middle of the domain while figure \ref{fig:KS_2} (b) shows that the momentum
presents an infinite slope when reaching the blowup. From figure
\ref{fig:KS_2} (d) we notice that the total and free energy temporally decay
until the blowup occurs.

\begin{figure}[ht!]
\begin{center}
\subfloat[Evolution of the density]{\protect\protect\includegraphics[scale=0.4]{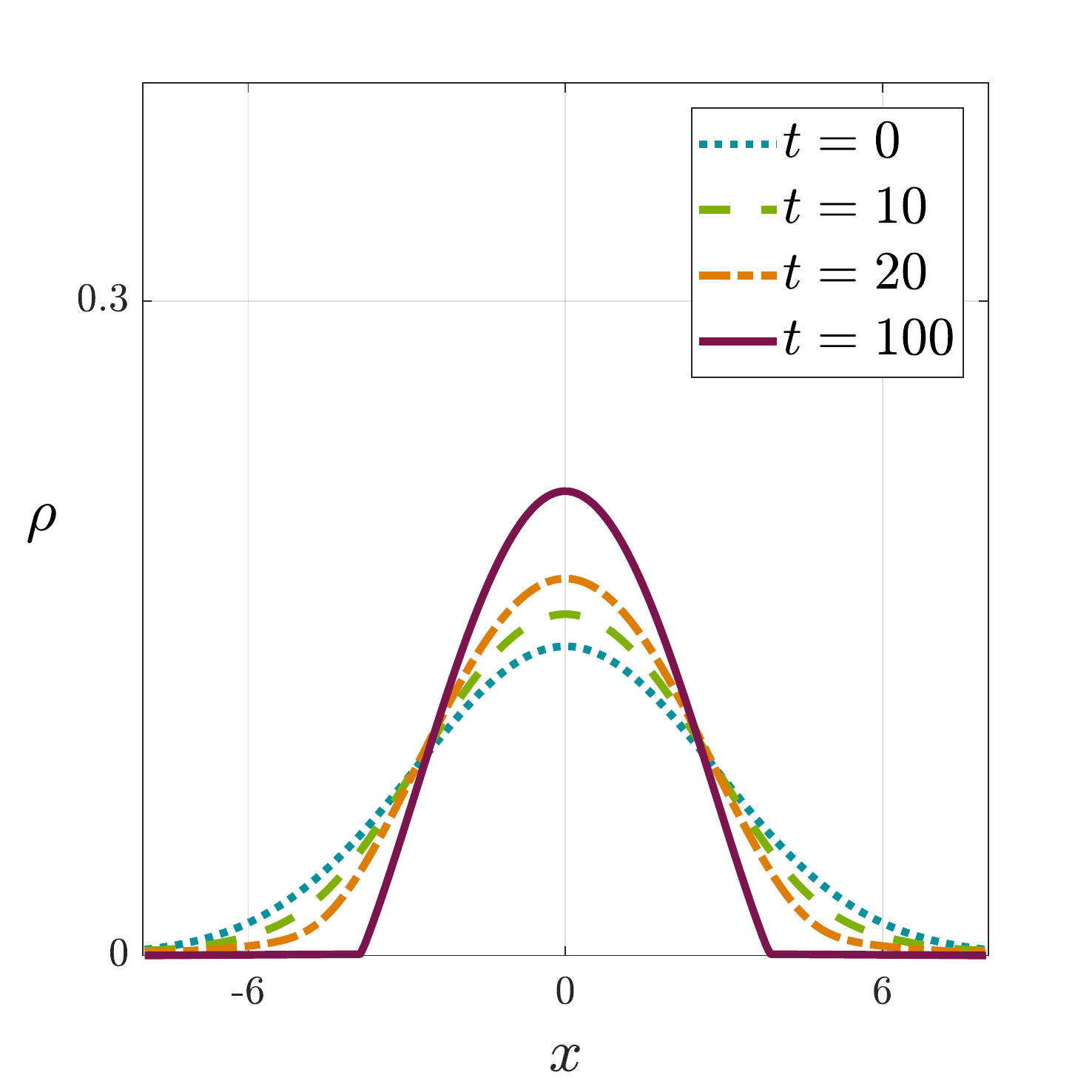}
}
\subfloat[Evolution of the momentum]{\protect\protect\includegraphics[scale=0.4]{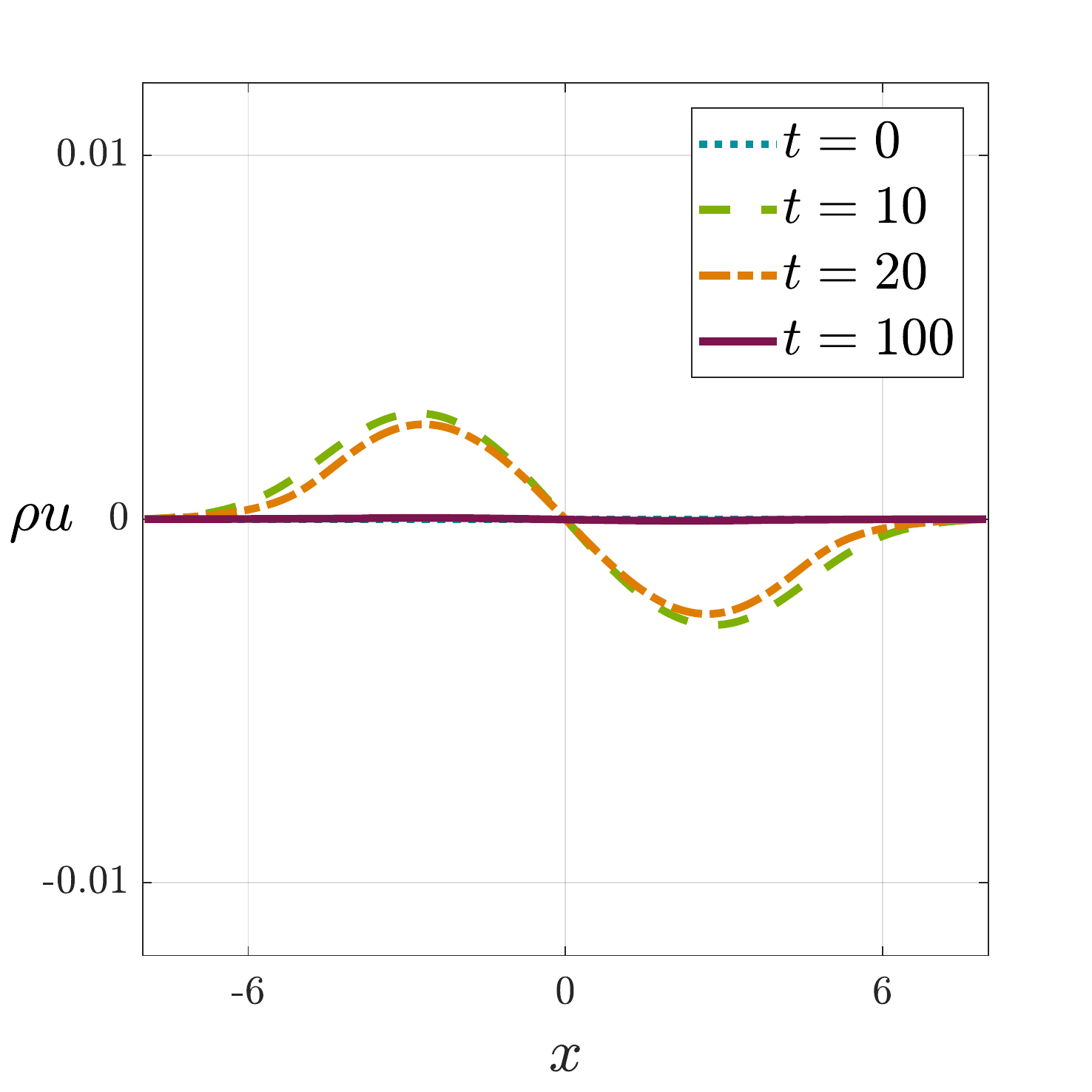}
}\\
\subfloat[Evolution of the density]{\protect\protect\includegraphics[scale=0.4]{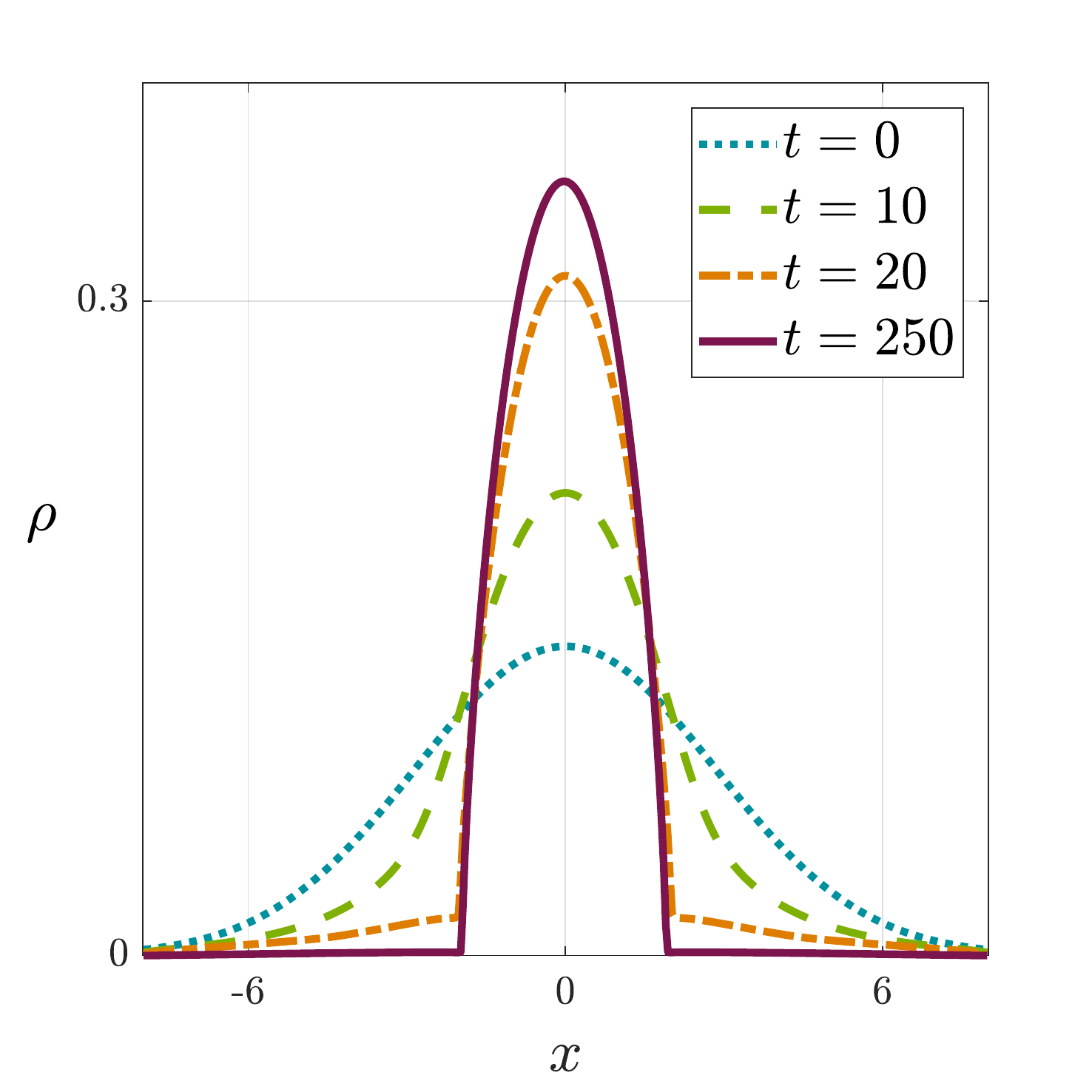}
}
\subfloat[Evolution of the total energy and free momentum]{\protect\protect\includegraphics[scale=0.4]{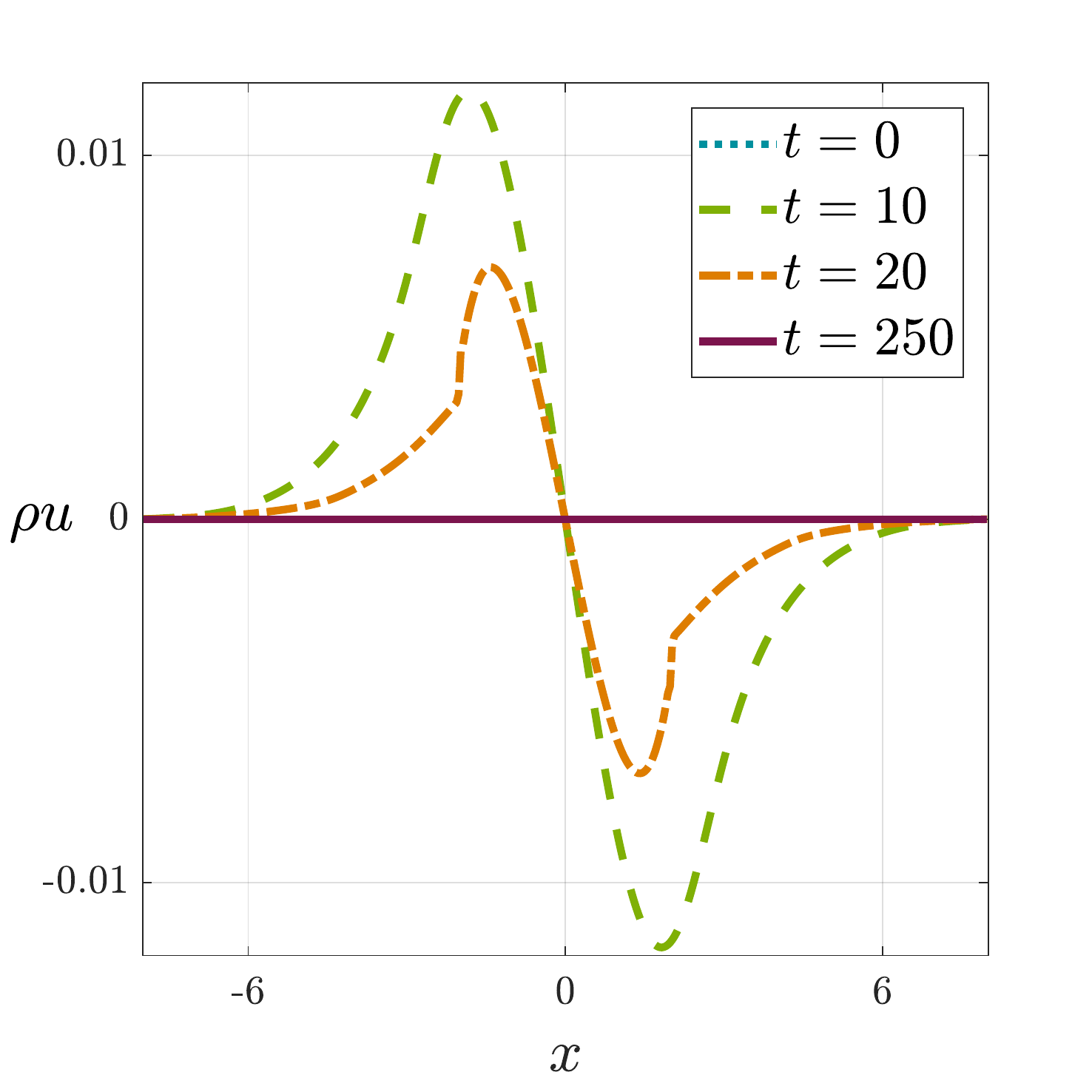}
}
\end{center}
\protect\protect\caption{\label{fig:KS_3} Temporal evolution of the Keller-Segel system with $W(x)=\ln|x|$ and initial conditions \eqref{eq:iniKS} with $M=1$. Compactly-supported steady state. (a)-(b) $P=3\rho^2$, steady state $t=250$, (c)-(d) $P=3\rho^{2.5}$, steady state $t=100$.}
\end{figure}

Finally, we also aim to compare diffusion-dominated solutions where $m>1$
leading to steady states that are compactly supported. For this purpose we
set the initial conditions to be \eqref{eq:iniKS} with a mass of $M=1$. For
comparison we look at two scenarios with $P=3\rho^2$ and $P=3\rho^{2.5}$ so
that the exponent $m$ is different. In figure \ref{fig:KS_3} we depict the
final steady states that arise from the two choices of $m$. From figures
\ref{fig:KS_3} (a) and (c) we observe that the final compactly-supported
density profiles have slightly different shapes due to the balances between
the attraction from the local kernel $W(x)$ and the repulsion caused by the
diffusion of the pressure $P(\rho)$.

\end{examplecase}

%
%
\section*{Acknowledgements}
Jos\'e A. Carrillo was partially supported by EPSRC grant number
EP/P031587/1. Manuel J. Castro acknowledges financial support from the
Spanish Government and FEDER through the coordinated Research project
RTI2018-096064-B-C21 and from Junta de Andaluc\'ia and FEDER through the
project UMA18-FEDERJA-161. Serafim Kalliadasis acknowledges financial support
from the Engineering and Physical Sciences Research Council (EPSRC) of the UK
through Grant No. EP/L020564. And Sergio~P.~Perez acknowledges financial
support from the Imperial College President's PhD Scholarship and thanks
Universidad de M\'alaga for hospitality during a visit in October 2019.

%
%

\appendix

%
%

\section{Details about the fifth-order Gaussian quadrature}\label{app:gauss}

In this section we detail the procedure to approximate an integral via the fifth-order Gaussian quadrature. Briefly, Gaussian quadratures of $n$ points yield exact values of integrals for polynomials of degree up to $2n-1$. In our case we implement Gaussian quadratures of $3$ points, so that the spatial error is of the order $\mathcal{O}(\Delta x^5)$. In this way we do not limit the order of the high-order finite volume schemes, since the order of this quadrature is always higher or equal than ones in the reconstructions of the scheme.

The approximation of a function $f(x)$ within an interval $[-1,1]$ via a three-point Gaussian quadrature \cite{abramowitz1988handbook} satisfies 
\begin{equation*}
    \int_{-1}^{1}f(x)\,dx=\frac{5}{9}f\lt(-\sqrt{\frac{3}{5}}\rt)+\frac{8}{9}f\lt(0\rt)+\frac{5}{9}f\lt(\sqrt{\frac{3}{5}}\rt).
\end{equation*}
In our case, this integration is performed within each of the finite volume cells in  $[x_{i-1/2},x_{i+1/2}]$ centred at $x_i$ and with size $\Delta x$. The cell averages in the finite volume schemes are also divided over $\Delta x$. As a result, the transformation of weights and spatial coordinates for the Gaussian quadrature from $[-1,1]$ to $[x_{i-1/2},x_{i+1/2}]$ results in
\begin{equation*}
    \frac{1}{\Delta x}\int_{x_{i-1/2}}^{x_{i+1/2}}f(x)\,dx=\frac{5}{18}f\lt(x_i-\frac{\Delta x}{2}\sqrt{\frac{3}{5}}\rt)+\frac{4}{9}f\lt(x_i\rt)+\frac{5}{18}f\lt(x_i+\frac{\Delta x}{2}\sqrt{\frac{3}{5}}\rt).
\end{equation*}
From this last expression we get that the coefficients $\alpha_j$ for $j\in\{1,2,3\}$ satisfy
\begin{equation*}
    \alpha_1=\frac{5}{18},\quad \alpha_2=\frac{4}{9},\quad\alpha_3=\frac{5}{18},
\end{equation*}
with the spatial nodes $x_i^j$ within the cell $i$ for the evaluations of the integrand located at
\begin{equation*}
    x_i^1=x_i-\frac{\Delta x}{2}\sqrt{\frac{3}{5}},\quad x_i^2=x_i,\quad
    x_i^3=x_i+\frac{\Delta x}{2}\sqrt{\frac{3}{5}}.
\end{equation*}

%
%

\section{Details about the positive-density CWENO reconstruction}\label{app:CWENO}

In this section we proceed to summarize the third- and fifth-order CWENO reconstructions of a generic function $g(x)$ whose cell averages $\{g_i\}$, defined as in \eqref{eq:cellU}, are taken as input. Each cell has size $\Delta x$, is centred at $\{x_i\}$ and is contained in the region $\left[x_{i-1/2},x_{i+1/2}\right]$. These reconstructions are applied in \eqref{eq:reconsCWENO} to compute the high-order reconstructions $R_i^{\rho}(x)$, $R_i^{\rho u}(x)$ and $R_i^{K}(x)$. For further details about the CWENO algorithm, we refer the reader to \cite{cravero2016accuracy,levy2000compact,levy1999central,capdeville2008central}.

In subsection \ref{subsec:thirdCWENO} we detail the third-order reconstruction, in subsection \ref{subsec:fifthCWENO} we proceed with the fifth-order reconstruction, and finally in subsection \ref{subsec:positiveCWENO} we end up by summarizing the positive-density limiters from \cite{zhang2010maximum}, which are essential to prove the positivity of the overall finite volume scheme. 

%
%

\subsection{Third-order CWENO reconstruction}\label{subsec:thirdCWENO}

The third-order CWENO reconstruction from \cite{levy1999central} satisfies
\begin{equation}\label{eq:CWENO3}
    R_i^g(x)=\underline{g_i}+\underline{g_i'}(x-x_i)+\frac{1}{2}\underline{g_i''}(x-x_i)^2,
\end{equation}
with $\underline{g_i}$, $\underline{g_i'}$ and $\underline{g_i''}$ resulting from
\begin{equation}\label{eq:coeffCWENO3}
\begin{split}
\underline{g_i} &= w_{i-1}^i\lt(\Tilde{g}_{i-1}+\Delta x\,\Tilde{g}_{i-1}'+\frac{1}{2}\Delta x^2\, \Tilde{g}_{i-1}''\rt)+w_i^i\,\Tilde{g}_j+w_{i+1}^i \lt(\Tilde{g}_{i+1}-\Delta x\,\Tilde{g}_{i+1}'+\frac{1}{2}\Delta x^2\,\Tilde{g}_{i+1}''\rt),\\[3pt] 
\underline{g_i'} &= w_{i-1}^i\lt(\Tilde{g}_{i-1}'+\Delta x\, \Tilde{g}_{i-1}''\rt)+w_i^i\,\Tilde{g}_i'+w_{i+1}^i\lt(\Tilde{g}_{i+1}'-\Delta x\,\Tilde{g}_{i+1}''\rt), \\[3pt]
\underline{g_i''} &= w_{i-1}^i\,\Tilde{g}_{i-1}''+w_i^i\,\Tilde{g}_i''+w_{i+1}^i\,\Tilde{g}_{i+1}'',\\[3pt]
\end{split}
\end{equation}
and $\Tilde{g}_k$, $\Tilde{g}_k'$ and $\Tilde{g}_k''$, for $k=\{i-1,i,i+1\}$, being computed as
\begin{equation*}
\begin{split}
        \Tilde{g}_k &= g_k-\frac{g_{k-1}-2g_k+g_{k+1}}{24},\quad 
        \Tilde{g}_k' = \frac{g_{k+1}-g_{k-1}}{2 \Delta x},
        \quad \Tilde{g}_k'' = \frac{g_{k+1}-2 g_k+ g_{k-1}}{\Delta x^2}.
\end{split}
\end{equation*}

The weights $w_k^i$ appearing in \eqref{eq:coeffCWENO3}, for $k=\{i-1,i,i+1\}$, satisfy
\begin{equation}\label{eq:weightsCWENO3}
    w_k^i=\frac{\alpha_k^i}{\alpha_{i-1}^i+\alpha_i^i+\alpha_{i+1}^i},\quad\text{where}\quad \alpha_k^i=\frac{C_k}{\lt(\epsilon+IS_k^i\rt)^p}.
\end{equation}
The constants $C_{i-1}$, $C_{i}$, $C_{i+1}$, $\epsilon$ and $p$ for $\alpha_k^i$ in \eqref{eq:weightsCWENO3} are
\begin{equation}\label{eq:constantsCWEN03}
    C_{i-1}=\frac{3}{16},\quad C_{i}=\frac{5}{8}, \quad C_{i+1}=\frac{3}{16}, \quad \epsilon=10^{-6},\quad p=3.
\end{equation}
Finally, the smoothness indicators $IS_k^i$ for $\alpha_k^i$ in \eqref{eq:weightsCWENO3}, where   $k=\{i-1,i,i+1\}$, result from
\begin{equation}\label{eq:ISCWENO3}
\begin{split}
    IS_{i-1}^i &= \frac{13}{12} \lt(g_{i-2}-2g_{i-1}+g_i\rt)^2+\frac{1}{4}\lt(g_{i-2}-4g_{i-1}+3g_i\rt)^2,\\
    IS_{i}^i &= \frac{13}{12} \lt(g_{i-1}-2 g_i +g_{i+1}\rt)^2+\frac{1}{4}\lt(g_{i-1}-g_{i+1}\rt)^2,\\
    IS_{i+1}^i &= \frac{13}{12} \lt(g_i-2g_{i+1}+g_{i+2}\rt)^2+\frac{1}{4}\lt(3g_i-4g_{i+1}+g_{i+2}\rt)^2.\\
    \end{split}
\end{equation}
The proposed choice of the constants in \eqref{eq:constantsCWEN03} and the smoothness indicators in \eqref{eq:ISCWENO3} is based on the original work for third-order CWENO reconstruction in \cite{levy1999central}. The reader can find about other more refined choices in \cite{levy2000compact} and later works \cite{semplice2016adaptive,arandiga2011analysis,kolb2014full,cravero2016accuracy}. 
%
%

\subsection{Fifth-order CWENO reconstruction}\label{subsec:fifthCWENO}

The fifth-order CWENO reconstruction from \cite{capdeville2008central} satisfies
\begin{equation}\label{eq:CWENO5}
    R_i^g(x)=g_{opt}(x)+\sum_{k\in\{1,2,3,c\}}\lt(w^i_k-C_k\rt)g_k(x),
\end{equation}
with $g_{opt}$, $g_1$, $g_2$, $g_3$ and $g_c$ resulting from
\begin{equation}\label{eq:polyCWENO5}
\begin{gathered}
g_{opt}(x) = \sum_{j=1}^5 a_{j} (x-x_i)^{j-1},\quad g_{1}(x) = \sum_{j=1}^3 b_{j} (x-x_i)^{j-1},\quad g_{2}(x) = \sum_{j=1}^3 c_{j} (x-x_i)^{j-1},\\[3pt]
\quad g_{3}(x) = \sum_{j=1}^3 d_{j} (x-x_i)^{j-1},\quad g_c(x) = \lt(g_{opt}(x)-C_1\,g_{1}(x)-C_2\,g_{2}(x)-C_3\,g_{3}(x)\rt)/C_c .
\end{gathered}
\end{equation}
The coefficients $a_{j}$ for $j\in\{1,2,3,4,5\}$ which appear in the optimal polynomial $g_{opt}(x)$ in \eqref{eq:polyCWENO5} are taken as
\begin{equation*}
\begin{gathered}
a_1 = \frac{1067}{960}g_i-\frac{29}{480}\lt(g_{i+1}+g_{i-1}\rt)+\frac{3}{640}\lt(g_{i+2}+g_{i-2}\rt),\quad a_2 = \frac{34\lt(g_{i+1}-g_{i-1}\rt)+5\lt(g_{i-2}-g_{i+2}\rt)}{48\Delta x},\\[6pt]
a_3 = \frac{g_{i-2}+22\,g_i+g_{i+2}-12\lt(g_{i+1}+g_{i-1}\rt)}{-16\Delta x^2},\quad a_4 = \frac{2\lt(g_{i+1}-g_{i-1}\rt)+\lt(g_{i-2}-g_{i+2}\rt)}{-12 \Delta x^3},\\[6pt]
a_5 = \frac{g_{i-2}+6g_i+g_{i+2}-4\lt(g_{i+1}+g_{i-1}\rt)}{24\Delta x^4}.
\end{gathered}
\end{equation*}

The rest of the coefficients $b_j$, $c_j$ and $d_j$ for $j\in\{1,2,3\}$ which appear in the polynomials $g_1$, $g_2$ and $g_3$, respectively, follow from
\begin{equation*}
\begin{gathered}
b_1 = \frac{23}{24}g_i+\frac{1}{12}\lt(g_{i-1}-\frac{1}{2}g_{i-2}\rt),\quad b_2 = \frac{3g_i-4g_{i-1}+g_{i-2}}{2\Delta x}, \quad b_3 = \frac{g_i-2g_{i-1}+g_{i-2}}{2\Delta x^2},\\[6pt]
c_1 = \frac{13}{12}g_i+\frac{1}{24}\lt(g_{i-1}+g_{i+1}\rt),\quad c_2 = \frac{g_{i+1}-g_{i-1}}{2\Delta x}, \quad c_3 = \frac{g_{i+1}-2g_{i}+g_{i-1}}{2\Delta x^2},\\[6pt]
d_1 = \frac{23}{24}g_i+\frac{1}{12}\lt(g_{i+1}-\frac{1}{2}g_{i+2}\rt),\quad d_2 = \frac{3g_i-4g_{i+1}+g_{i+2}}{-2\Delta x}, \quad d_3 = \frac{g_i-2g_{i+1}+g_{i+2}}{2\Delta x^2}.
\end{gathered}
\end{equation*}

The weights $w^i_k$ for $k \in\{1,2,3,c\}$ in the fifth-order CWENO reconstruction \eqref{eq:CWENO5} satisfy
\begin{equation}\label{eq:weightsCWENO5}
    w_k^i=\frac{\alpha_k^i}{\sum_{k\in\{1,2,3,c\}}\alpha^i_k},\quad\text{where}\quad \alpha_k^i=\frac{C_k}{\lt(\epsilon+IS_k^i\rt)^p}.
\end{equation}

The constants $C_{1}$, $C_{2}$, $C_{3}$, $C_{c}$, $\epsilon$ and $p$ for $\alpha_k^i$ in \eqref{eq:weightsCWENO5} are
\begin{equation*}
    C_{1}=\frac{1}{8},\quad C_{2}=\frac{1}{4}, \quad C_{3}=\frac{1}{8}, \quad C_{c}=\frac{1}{2}, \quad \epsilon=10^{-6},\quad p=2.
\end{equation*}

Finally, the smoothness indicators $IS_k^i$ which appear in the computation of $\alpha_k^i$ in \eqref{eq:weightsCWENO5}, with   $k\in\{1,2,3,c\}$, result from
\begin{equation*}
\begin{gathered}
    IS_{1}^i = b_2^2\Delta x^2+\frac{13}{3}b_3^2\Delta x^4,\quad
    IS_{2}^i = c_2^2\Delta x^2+\frac{13}{3}c_3^2\Delta x^4,\quad IS_{3}^i = d_2^2\Delta x^2+\frac{13}{3}d_3^2\Delta x^4,\\
    IS_{4}^i = a_2^2\Delta x^2+\lt(\frac{13}{3}a_3^2+\frac{1}{2}a_2\,a_4\rt)\Delta x^4.\\
    \end{gathered}
\end{equation*}

%
%

\subsection{Positive-density CWENO reconstruction}\label{subsec:positiveCWENO}

The third- and fifth-order CWENO reconstructions in \eqref{eq:CWENO3} or \eqref{eq:CWENO5}, respectively, can be modified to yield positive values for evaluations at specific spatial points of the finite volume cell. In our case we are interested in obtaining positive values of the density at the points required by the numerical scheme in \eqref{eq:hofinal}. Namely, those points are: the boundaries of the cell, computed in \eqref{eq:recons_b} and employed for in \eqref{eq:numflux_2}, \eqref{eq:fluxsource_2} and \eqref{eq:fluxhom_2}, and the quadrature points to compute the integrals in the nonlinear damping term in \eqref{eq:hofinal}, the distribution of the high-order corrections in the source term in \eqref{eq:hofinal} and described in appendix \ref{app:source}, and the update of $K_i^{n+1}$ in \eqref{eq:Kn+1}. We refer the reader to the appendices \ref{app:gauss} and \ref{app:source} for the details of the location of the quadrature points. 

Zhang and Shu \cite{zhang2010maximum} proposed a methodology to construct maximum-principle-satisfying high-order schemes. Here we apply their work in \cite{zhang2010maximum} to modify the CWENO reconstructions in the subsections \ref{subsec:thirdCWENO} and \ref{subsec:fifthCWENO}, so that they preserve the positivity of the density. The procedure is the following:

\begin{enumerate}[label=\arabic*)]
    \item Construct the third- or fifth-order CWENO polynomial in \eqref{eq:CWENO3} or \eqref{eq:CWENO5}, respectively, for all the finite volume cells.
    \item Evaluate the reconstructed polynomials at the spatial points required by the numerical scheme in \eqref{eq:hofinal}: the boundaries of the cells and quadrature points. We denote as $M$ the total number those spatial points. 
    \item For each cell $i$ compute the minimum $m_i$ of the evaluations of the reconstructed polynomial at the required spatial points $x_i^j$, for $j\in \{1,\ldots,M\}$, so that $m_i=\min_{j\in \{1,\ldots,M\}} R_i^g(x_i^j)$.
    \item Modify the reconstructed polynomial in \eqref{eq:CWENO3} or \eqref{eq:CWENO5} so that $m_i\geq 0$, in the following way:
    \begin{equation*}
        \widetilde{R_i^g}(x)=\theta \lt(R_i^g(x)-g_i\rt)+g_i,\quad \text{with}\quad \theta=\min \lt\{\frac{\lt|0-g_i\rt|}{\lt|m_i-g_i\rt|},1\rt\}.
    \end{equation*}
    \item Evaluate all the $M$ spatial points within each cell with the modified reconstructed polynomial $\widetilde{R_i^g}(x)$.
    \item Apply the following CFL condition for the time step $\Delta t$, depending on the numerical flux employed and where $\alpha_j$ is the quadrature weight of the spatial point $x_i^j$, for $j\in \{1,\ldots,M\}$:
    \begin{equation}\label{eq:CFL}
\Delta t= \begin{cases}
\text{CFL} \frac{\Delta x\,\min_{{j\in \{1,\ldots,M\}}} \alpha_j}{\max_{\forall\left(U_{i+\frac{1}{2}}^-,U_{i+\frac{1}{2}}^+\right)}\left\{\left|u+\sqrt{P'(\rho)}\right|,\left|u-\sqrt{P'(\rho)}\right|\right\}}, & \text{if Lax-Friedrich flux},\\[1cm]
\text{CFL} \frac{\Delta x\,\min_{{j\in \{1,\ldots,M\}}} \alpha_j}{\max_{\forall\left(U_{i+\frac{1}{2}}^-,U_{i+\frac{1}{2}}^+\right)}\left\{|u|+3^{\frac{m-1}{4}}\right\}}, & \text{if kinetic flux}.
\end{cases}
\end{equation}
The quadrature weights $\alpha_j$ employed in this work are specified in the appendices \ref{app:gauss} and \ref{app:source}. For the details about the numerical fluxes we refer the reader to our previous work \cite{carrillo2018wellbalanced}.

\end{enumerate}

%
%

\section{Details about the integration for the high-order corrections}\label{app:source}

In this section we follow \cite{noelle2006well,noelle2007high} to propose a fourth- and sixth-order quadrature for the high-order corrections in the source term, which appear in the last term of \eqref{eq:hofinal} and satisfy
\begin{equation}\label{eq:correc}
    I_i=\int_{x_{i-\frac{1}{2}}}^{x_{i+\frac{1}{2}}}  \left(S\left(R^{\rho}_i(x),R^{H}_i(x)\right)-S\left(U_i^*(x),H_i^*(x)\right)\right)\,dx.
\end{equation}

The first step is to define a general trapezoidal numerical quadrature $I_i^m$ for the integral $I_i$ in \eqref{eq:correc}, which employs $m$ points $x_i^j=x_{i-1/2}+(j-1)\Delta x/m$ of the cell $i$, with $j\in\{1,\ldots,m\}$. Such integral yields
\begin{equation}
    I_i^m=\sum_{j=1}^{m-1}\lt[\frac{R^{\rho}_i(x_i^j)+R^{\rho}_i(x_i^{j+1})}{2}\left(R^{H}_i(x_i^{j+1})-R^{H}_i(x_i^{j})\right)-\frac{R^{\rho}_i(x_i^j)+R^{\rho}_i(x_i^{j+1})}{2}\left(H^*_i(x_i^{j+1})-H^*_i(x_i^{j})\right)\rt].
\end{equation}
The integral $I_i^m$ is a second-order approximation of $I_i$ in \eqref{eq:correc}, and its asymptotic form satisfies \cite{noelle2006well}
\begin{equation}\label{eq:int_error}
    I_i^m=I_i+c_1\lt(\frac{\Delta x}{m}\rt)^2+c_2\lt(\frac{\Delta x}{m}\rt)^4+\ldots
\end{equation}
The strategy to obtain the fourth- and sixth-order schemes relies in computing the integral $I_i$ in \eqref{eq:correc} as a linear combination of $I_i^m$ for different $m$, such that the desired errors in \eqref{eq:int_error} are cancelled. The required formulas of integration are:

\begin{enumerate}[label=\alph*)]
    \item Fourth-order quadrature employing $I_i^1$ and $I_i^2$, so that
    \begin{equation*}
        I_i=\frac{4I_i^2-I_i^1}{3}+\mathcal{O}(\Delta x^4).
    \end{equation*}
    \item Sixth-order quadrature employing $I_i^1$, $I_i^2$ and $I_i^3$, so that
    \begin{equation*}
        I_i=\frac{81}{40}I_i^3-\frac{16}{15}I_i^2+\frac{1}{24}I_i^1+\mathcal{O}(\Delta x^6).
    \end{equation*}
\end{enumerate}

As a remark, the order of these quadratures is maintained as long as the order of the reconstructions for the density in \eqref{eq:reconsCWENO} and the potential in \eqref{eq:reconsH} is greater than or at least equal to the order of the quadrature formulas. Otherwise the order of the quadrature is diminished and matches the order of the reconstruction.

%
%

\bibliographystyle{siam}
\bibliography{High-Order-bib}

\begin{thebibliography}{10}

\bibitem{abramowitz1988handbook}
{\sc M.~Abramowitz, I.~A. Stegun, and R.~H. Romer}, {\em Handbook of
  mathematical functions with formulas, graphs, and mathematical tables}, 1988.

\bibitem{arandiga2011analysis}
{\sc F.~Ar{\`a}ndiga, A.~Baeza, A.~Belda, and P.~Mulet}, {\em Analysis of
  {WENO} schemes for full and global accuracy}, SIAM Journal on Numerical
  Analysis, 49 (2011), pp.~893--915.

\bibitem{audusse2004fast}
{\sc E.~Audusse, F.~Bouchut, M.-O. Bristeau, R.~Klein, and B.~t. Perthame},
  {\em {A fast and stable well-balanced scheme with hydrostatic reconstruction
  for shallow water flows}}, SIAM Journal on Scientific Computing, 25 (2004),
  pp.~2050--2065.

\bibitem{audusse2015very}
{\sc E.~Audusse, C.~Chalons, and P.~Ung}, {\em A very simple well-balanced
  positive and entropy-satisfying scheme for the shallow-water equations},
  Commun. Math. Sci, 13 (2015), pp.~1317--1332.

\bibitem{bailo1811fully}
{\sc R.~Bailo, J.~A. Carrillo, and J.~Hu}, {\em Fully discrete
  positivity-preserving and energy-decaying schemes for aggregation-diffusion
  equations with a gradient flow structure}, Preprint arxiv, 2018 (1811).

\bibitem{bellomo2015toward}
{\sc N.~Bellomo, A.~Bellouquid, Y.~Tao, and M.~Winkler}, {\em Toward a
  mathematical theory of {K}eller--{S}egel models of pattern formation in
  biological tissues}, Mathematical Models and Methods in Applied Sciences, 25
  (2015), pp.~1663--1763.

\bibitem{bermudez1994upwind}
{\sc A.~Berm{\'{u}}dez and M.~E. V{\'{a}}zquez}, {\em {Upwind methods for
  hyperbolic conservation laws with source terms}}, {C}omputers \& {F}luids, 23
  (1994), pp.~1049--1071.

\bibitem{bian2013dynamic}
{\sc S.~Bian and J.-G. Liu}, {\em Dynamic and steady states for
  multi-dimensional {K}eller-{S}egel model with diffusion exponent $m> 0$},
  Communications in Mathematical Physics, 323 (2013), pp.~1017--1070.

\bibitem{binney2011galactic}
{\sc J.~Binney and S.~Tremaine}, {\em Galactic dynamics}, Princeton university
  press, 2011.

\bibitem{calvez2017equilibria}
{\sc V.~Calvez, J.~A. Carrillo, and F.~Hoffmann}, {\em Equilibria of
  homogeneous functionals in the fair-competition regime}, Nonlinear Analysis,
  159 (2017), pp.~85--128.

\bibitem{calvez2017geometry}
{\sc V.~Calvez, J.~A. Carrillo, and F.~Hoffmann}, {\em The geometry of
  diffusing and self-attracting particles in a one-dimensional fair-competition
  regime}, in Nonlocal and nonlinear diffusions and interactions: new methods
  and directions, Springer, 2017, pp.~1--71.

\bibitem{capdeville2008central}
{\sc G.~Capdeville}, {\em A central {WENO} scheme for solving hyperbolic
  conservation laws on non-uniform meshes}, Journal of Computational Physics,
  227 (2008), pp.~2977--3014.

\bibitem{carrillo2015finite}
{\sc J.~A. Carrillo, A.~Chertock, and Y.~Huang}, {\em {A finite-volume method
  for nonlinear nonlocal equations with a gradient flow structure}},
  Communications in Computational Physics, 17 (2015), pp.~233--258.

\bibitem{carrillo2017review}
{\sc J.~A. Carrillo, Y.-P. Choi, and S.~P. Perez}, {\em {A review on
  attractive--repulsive hydrodynamics for consensus in collective behavior}},
  in Act. Part. Vol. 1, Springer, 2017, pp.~259--298.

\bibitem{carrillo2016critical}
{\sc J.~A. Carrillo, Y.-P. Choi, E.~Tadmor, and C.~Tan}, {\em Critical
  thresholds in 1d {E}uler equations with non-local forces}, Mathematical
  Models and Methods in Applied Sciences, 26 (2016), pp.~185--206.

\bibitem{carrillo2017weak}
{\sc J.~A. Carrillo, E.~Feireisl, P.~Gwiazda, and
  A.~{\'{S}}wierczewska-Gwiazda}, {\em {Weak solutions for {E}uler systems with
  non-local interactions}}, Journal of the London Mathematical Society, 95
  (2017), pp.~705--724.

\bibitem{carrillo2010asymptotic}
{\sc J.~A. Carrillo, M.~Fornasier, J.~Rosado, and G.~Toscani}, {\em Asymptotic
  flocking dynamics for the kinetic cucker--smale model}, SIAM Journal on
  Mathematical Analysis, 42 (2010), pp.~218--236.

\bibitem{carrillo2010particle}
{\sc J.~A. Carrillo, M.~Fornasier, G.~Toscani, and F.~Vecil}, {\em Particle,
  kinetic, and hydrodynamic models of swarming}, in Mathematical modeling of
  collective behavior in socio-economic and life sciences, Springer, 2010,
  pp.~297--336.

\bibitem{carrillo2018ground}
{\sc J.~A. Carrillo, F.~Hoffmann, E.~Mainini, and B.~Volzone}, {\em Ground
  states in the diffusion-dominated regime}, Calculus of variations and partial
  differential equations, 57 (2018), p.~127.

\bibitem{carrillo2014explicit}
{\sc J.~A. Carrillo, Y.~Huang, and S.~Martin}, {\em {Explicit flock solutions
  for quasi-morse potentials}}, European Journal of Applied Mathematics, 25
  (2014), pp.~553--578.

\bibitem{carrillo2018wellbalanced}
{\sc J.~A. Carrillo, S.~Kalliadasis, S.~P. Perez, and C.-W. Shu}, {\em
  Well-balanced finite volume schemes for hydrodynamic equations with general
  free energy}, 2020.
\newblock to appear.

\bibitem{carrillo2018longtime}
{\sc J.~A. Carrillo, A.~Wr{\'o}blewska-Kami{\'n}ska, and E.~Zatorska}, {\em On
  long-time asymptotics for viscous hydrodynamic models of collective behavior
  with damping and nonlocal interactions}, Mathematical Models and Methods in
  Applied Sciences, 29 (2018), pp.~31--63.

\bibitem{castro2017well}
{\sc M.~J. Castro, T.~M. de~Luna, and C.~Par{\'e}s}, {\em Well-balanced schemes
  and path-conservative numerical methods}, in Handbook of Numerical Analysis,
  vol.~18, Elsevier, 2017, pp.~131--175.

\bibitem{castro13}
{\sc M.~J. Castro, J.~A. L{\'o}pez-Garc{\'i}a, and C.~Par{\'e}s}, {\em High
  order exactly well-balanced numerical methods for shallow water systems},
  Journal of Computational Physics, 246 (2013), pp.~242 -- 264.

\bibitem{castro2007well}
{\sc M.~J. Castro, A.~Pardo~Milan{\'e}s, and C.~Par{\'e}s}, {\em Well-balanced
  numerical schemes based on a generalized hydrostatic reconstruction
  technique}, Mathematical Models and Methods in Applied Sciences, 17 (2007),
  pp.~2055--2113.

\bibitem{castro2020}
{\sc M.~J. Castro and C.~Par{\'e}s}, {\em Well-balanced high-order finite
  volume methods for systems of balance laws}, Journal of Scientific Computing,
  82 (2020), p.~48.

\bibitem{castro2019third}
{\sc M.~J. Castro and M.~Semplice}, {\em Third-and fourth-order well-balanced
  schemes for the shallow water equations based on the {CWENO} reconstruction},
  International Journal for Numerical Methods in Fluids, 89 (2019),
  pp.~304--325.

\bibitem{chalons2018high}
{\sc C.~Chalons, P.~Goatin, and L.~M. Villada}, {\em High-order numerical
  schemes for one-dimensional nonlocal conservation laws}, SIAM Journal on
  Scientific Computing, 40 (2018), pp.~A288--A305.

\bibitem{chavanis2007kinetic}
{\sc P.-H. Chavanis and C.~Sire}, {\em Kinetic and hydrodynamic models of
  chemotactic aggregation}, Physica A: Statistical Mechanics and its
  Applications, 384 (2007), pp.~199--222.

\bibitem{cheng2019new}
{\sc Y.~Cheng, A.~Chertock, M.~Herty, A.~Kurganov, and T.~Wu}, {\em A new
  approach for designing moving-water equilibria preserving schemes for the
  shallow water equations}, Journal of Scientific Computing, 80 (2019),
  pp.~538--554.

\bibitem{chertock2018well}
{\sc A.~Chertock, S.~Cui, A.~Kurganov, {\c{S}}.~N. {\"O}zcan, and E.~Tadmor},
  {\em Well-balanced schemes for the {E}uler equations with gravitation:
  Conservative formulation using global fluxes}, Journal of Computational
  Physics, 358 (2018), pp.~36--52.

\bibitem{choi2017emergent}
{\sc Y.-P. Choi, S.-Y. Ha, and Z.~Li}, {\em Emergent dynamics of the
  cucker--smale flocking model and its variants}, in Active Particles, Volume
  1, Springer, 2017, pp.~299--331.

\bibitem{churuksaeva2015mathematical}
{\sc V.~Churuksaeva and A.~Starchenko}, {\em Mathematical modeling of a river
  stream based on a shallow water approach}, Procedia Computer Science, 66
  (2015), pp.~200--209.

\bibitem{cordier2000quasineutral}
{\sc S.~Cordier and E.~Grenier}, {\em Quasineutral limit of an
  {E}uler-{P}oisson system arising from plasma physics}, Communications in
  Partial Differential Equations, 25 (2000), pp.~1099--1113.

\bibitem{cozzolino2018solution}
{\sc L.~Cozzolino, V.~Pepe, L.~Cimorelli, A.~D'Aniello, R.~Della~Morte, and
  D.~Pianese}, {\em The solution of the dam-break problem in the porous shallow
  water equations}, Advances in Water Resources, 114 (2018), pp.~83--101.

\bibitem{cravero2016accuracy}
{\sc I.~Cravero and M.~Semplice}, {\em On the accuracy of {WENO} and {CWENO}
  reconstructions of third order on nonuniform meshes}, Journal of Scientific
  Computing, 67 (2016), pp.~1219--1246.

\bibitem{cucker2007emergent}
{\sc F.~Cucker and S.~Smale}, {\em {Emergent behavior in flocks}}, IEEE
  Transactions on Automatic Control, 52 (2007), pp.~852--862.

\bibitem{cucker2007mathematics}
\leavevmode\vrule height 2pt depth -1.6pt width 23pt, {\em {On the mathematics
  of emergence}}, Japanese J. Math., 2 (2007), pp.~197--227.

\bibitem{cucker2004modeling}
{\sc F.~Cucker, S.~Smale, and D.-X. Zhou}, {\em Modeling language evolution},
  Foundations of Computational Mathematics, 4 (2004), pp.~315--343.

\bibitem{filbet2006finite}
{\sc F.~Filbet}, {\em A finite volume scheme for the
  {P}atlak--{K}eller--{S}egel chemotaxis model}, Numerische Mathematik, 104
  (2006), pp.~457--488.

\bibitem{filbet2005approximation}
{\sc F.~Filbet and C.-W. Shu}, {\em {Approximation of hyperbolic models for
  chemosensitive movement}}, SIAM Journal on Scientific Computing, 27 (2005),
  pp.~850--872.

\bibitem{fjordholm2012arbitrarily}
{\sc U.~S. Fjordholm, S.~Mishra, and E.~Tadmor}, {\em Arbitrarily high-order
  accurate entropy stable essentially nonoscillatory schemes for systems of
  conservation laws}, SIAM Journal on Numerical Analysis, 50 (2012),
  pp.~544--573.

\bibitem{gallardo2007well}
{\sc J.~M. Gallardo, C.~Par{\'e}s, and M.~J. Castro}, {\em On a well-balanced
  high-order finite volume scheme for shallow water equations with topography
  and dry areas}, Journal of Computational Physics, 227 (2007), pp.~574--601.

\bibitem{gamba2003percolation}
{\sc A.~Gamba, D.~Ambrosi, A.~Coniglio, A.~{De Candia}, S.~{Di Talia},
  E.~Giraudo, G.~Serini, L.~Preziosi, and F.~Bussolino}, {\em {Percolation,
  morphogenesis, and Burgers dynamics in blood vessels formation}}, Physical
  Review Letters, 90 (2003), p.~118101.

\bibitem{gascon2001construction}
{\sc L.~Gasc{\'o}n and J.~Corber{\'a}n}, {\em Construction of second-order tvd
  schemes for nonhomogeneous hyperbolic conservation laws}, Journal of
  computational physics, 172 (2001), pp.~261--297.

\bibitem{giardina2008collective}
{\sc I.~Giardina}, {\em Collective behavior in animal groups: theoretical
  models and empirical studies}, HFSP journal, 2 (2008), pp.~205--219.

\bibitem{giesselmann2017relative}
{\sc J.~Giesselmann, C.~Lattanzio, and A.~E. Tzavaras}, {\em {Relative energy
  for the Korteweg theory and related Hamiltonian flows in gas dynamics}},
  Archive for Rational Mechanics and Analysis, 223 (2017), pp.~1427--1484.

\bibitem{goddard2012general}
{\sc B.~D. Goddard, A.~Nold, N.~Savva, P.~G. A, and S.~Kalliadasis}, {\em
  General dynamical density functional theory for classical fluids}, Physical
  Review Letters, 109 (2012), p.~120603.

\bibitem{goddard2012unification}
{\sc B.~D. Goddard, A.~Nold, N.~Savva, P.~Yatsyshin, and S.~Kalliadasis}, {\em
  {Unification of dynamic density functional theory for colloidal fluids to
  include inertia and hydrodynamic interactions: derivation and numerical
  experiments}}, Journal of Physics: Condensed Matter, 25 (2012), p.~35101.

\bibitem{gosse2000well}
{\sc L.~Gosse}, {\em A well-balanced flux-vector splitting scheme designed for
  hyperbolic systems of conservation laws with source terms}, Computers \&
  {M}athematics with {A}pplications, 39 (2000), pp.~135--159.

\bibitem{gosse2012asymptotic}
{\sc L.~Gosse}, {\em Asymptotic-preserving and well-balanced schemes for the 1d
  cattaneo model of chemotaxis movement in both hyperbolic and diffusive
  regimes}, Journal of Mathematical Analysis and Applications, 388 (2012),
  pp.~964--983.

\bibitem{gottlieb1998total}
{\sc S.~Gottlieb and C.-W. Shu}, {\em {Total variation diminishing Runge-Kutta
  schemes}}, Math. Comput. Am. Math. Soc., 67 (1998), pp.~73--85.

\bibitem{greenberg1996well}
{\sc J.~M. Greenberg and A.-Y. LeRoux}, {\em {A well-balanced scheme for the
  numerical processing of source terms in hyperbolic equations}}, SIAM Journal
  on Numerical Analysis, 33 (1996), pp.~1--16.

\bibitem{guo2011global}
{\sc Y.~Guo and B.~Pausader}, {\em Global smooth ion dynamics in the
  {E}uler-{P}oisson system}, Communications in Mathematical Physics, 303
  (2011), pp.~89--125.

\bibitem{ha2008from}
{\sc S.-Y. Ha and E.~Tadmor}, {\em From particle to kinetic and hydrodynamic
  descriptions of flocking}, Kinetic \& Related Models, 1 (2008), p.~415.

\bibitem{hadvzic2019class}
{\sc M.~Had{\v{z}}i{\'c} and J.~J. Jang}, {\em A class of global solutions to
  the {E}uler--{P}oisson system}, Communications in Mathematical Physics, 370
  (2019), pp.~475--505.

\bibitem{hildenbrandt2010self}
{\sc H.~Hildenbrandt, C.~Carere, and C.~K. Hemelrijk}, {\em Self-organized
  aerial displays of thousands of starlings: a model}, Behavioral Ecology, 21
  (2010), pp.~1349--1359.

\bibitem{katz2011inferring}
{\sc Y.~Katz, K.~Tunstr{\o}m, C.~C. Ioannou, C.~Huepe, and I.~D. Couzin}, {\em
  Inferring the structure and dynamics of interactions in schooling fish},
  Proceedings of the National Academy of Sciences, 108 (2011),
  pp.~18720--18725.

\bibitem{keller1970initiation}
{\sc E.~F. Keller and L.~A. Segel}, {\em Initiation of slime mold aggregation
  viewed as an instability}, Journal of theoretical biology, 26 (1970),
  pp.~399--415.

\bibitem{klingenberg2019arbitrary}
{\sc C.~Klingenberg, G.~Puppo, and M.~Semplice}, {\em Arbitrary order finite
  volume well-balanced schemes for the {E}uler equations with gravity}, SIAM
  Journal on Scientific Computing, 41 (2019), pp.~A695--A721.

\bibitem{kolb2014full}
{\sc O.~Kolb}, {\em On the full and global accuracy of a compact third order
  {WENO} scheme}, SIAM Journal on Numerical Analysis, 52 (2014),
  pp.~2335--2355.

\bibitem{lattanzio2017gas}
{\sc C.~Lattanzio and A.~E. Tzavaras}, {\em From gas dynamics with large
  friction to gradient flows describing diffusion theories}, Communications in
  Partial Differential Equations, 42 (2017), pp.~261--290.

\bibitem{levy1999central}
{\sc D.~Levy, G.~Puppo, and G.~Russo}, {\em Central {WENO} schemes for
  hyperbolic systems of conservation laws}, ESAIM: Mathematical Modelling and
  Numerical Analysis, 33 (1999), pp.~547--571.

\bibitem{levy2000compact}
\leavevmode\vrule height 2pt depth -1.6pt width 23pt, {\em Compact central
  {WENO} schemes for multidimensional conservation laws}, SIAM Journal on
  Scientific Computing, 22 (2000), pp.~656--672.

\bibitem{marche2007evaluation}
{\sc F.~Marche, P.~Bonneton, P.~Fabrie, and N.~Seguin}, {\em Evaluation of
  well-balanced bore-capturing schemes for 2d wetting and drying processes},
  International Journal for Numerical Methods in Fluids, 53 (2007),
  pp.~867--894.

\bibitem{minakowski2019singular}
{\sc P.~Minakowski, P.~B. Mucha, J.~Peszek, and E.~Zatorska}, {\em Singular
  cucker--smale dynamics}, in Active Particles, Volume 2, Springer, 2019,
  pp.~201--243.

\bibitem{motsch2011new}
{\sc S.~Motsch and E.~Tadmor}, {\em A new model for self-organized dynamics and
  its flocking behavior}, Journal of Statistical Physics, 144 (2011), p.~923.

\bibitem{natalini2012asymptotic}
{\sc R.~Natalini and M.~Ribot}, {\em Asymptotic high order mass-preserving
  schemes for a hyperbolic model of chemotaxis}, SIAM Journal on Numerical
  Analysis, 50 (2012), pp.~883--905.

\bibitem{natalini2012well}
{\sc R.~Natalini, M.~Ribot, and M.~Twarogowska}, {\em A well-balanced numerical
  scheme for a one dimensional quasilinear hyperbolic model of chemotaxis},
  arXiv preprint arXiv:1211.4010,  (2012).

\bibitem{noelle2006well}
{\sc S.~Noelle, N.~Pankratz, G.~Puppo, and J.~R. Natvig}, {\em Well-balanced
  finite volume schemes of arbitrary order of accuracy for shallow water
  flows}, Journal of Computational Physics, 213 (2006), pp.~474--499.

\bibitem{noelle2007high}
{\sc S.~Noelle, Y.~Xing, and C.-W. Shu}, {\em High-order well-balanced finite
  volume {WENO} schemes for shallow water equation with moving water}, Journal
  of Computational Physics, 226 (2007), pp.~29--58.

\bibitem{perea2009extension}
{\sc L.~Perea, G.~G{\'o}mez, and P.~Elosegui}, {\em Extension of the
  cucker-smale control law to space flight formations}, Journal of guidance,
  control, and dynamics, 32 (2009), pp.~527--537.

\bibitem{perthame2001kinetic}
{\sc B.~Perthame and C.~Simeoni}, {\em {A kinetic scheme for the Saint-Venant
  system with a source term}}, Calcolo, 38 (2001), pp.~201--231.

\bibitem{semplice2016adaptive}
{\sc M.~Semplice, A.~Coco, and G.~Russo}, {\em Adaptive mesh refinement for
  hyperbolic systems based on third-order compact {WENO} reconstruction},
  Journal of Scientific Computing, 66 (2016), pp.~692--724.

\bibitem{serini2003modeling}
{\sc G.~Serini, D.~Ambrosi, E.~Giraudo, A.~Gamba, L.~Preziosi, and
  F.~Bussolino}, {\em Modeling the early stages of vascular network assembly},
  The EMBO journal, 22 (2003), pp.~1771--1779.

\bibitem{short2008statistical}
{\sc M.~B. Short, M.~R. D'orsogna, V.~B. Pasour, G.~E. Tita, P.~J. Brantingham,
  A.~L. Bertozzi, and L.~B. Chayes}, {\em A statistical model of criminal
  behavior}, Mathematical Models and Methods in Applied Sciences, 18 (2008),
  pp.~1249--1267.

\bibitem{Shu97}
{\sc C.-W. Shu}, {\em Essentially non-oscillatory and weighted essentially
  non-oscillatory schemes for hyperbolic conservation laws}, tech. rep., 1997.

\bibitem{Shu88}
{\sc C.-W. Shu and S.~Osher}, {\em Efficient implementation of essentially
  non-oscillatory shock-capturing schemes}, Journal of Computational Physics,
  77 (1988), pp.~439--471.

\bibitem{shuosh}
\leavevmode\vrule height 2pt depth -1.6pt width 23pt, {\em Efficient
  implementation of essentially non-oscillatory shock-capturing schemes, {II}},
  Journal of Computational Physics, 83 (1989), pp.~32--78.

\bibitem{tadmor2016entropy}
{\sc E.~Tadmor}, {\em Entropy stable schemes}, in Handbook of Numerical
  Analysis, vol.~17, Elsevier, 2016, pp.~467--493.

\bibitem{thomann2019second}
{\sc A.~Thomann, M.~Zenk, and C.~Klingenberg}, {\em A second-order
  positivity-preserving well-balanced finite volume scheme for {E}uler
  equations with gravity for arbitrary hydrostatic equilibria}, International
  Journal for Numerical Methods in Fluids, 89 (2019), pp.~465--482.

\bibitem{xing2006high}
{\sc Y.~Xing and C.-W. Shu}, {\em High order well-balanced finite volume {WENO}
  schemes and discontinuous {G}alerkin methods for a class of hyperbolic
  systems with source terms}, Journal of Computational Physics, 214 (2006),
  pp.~567--598.

\bibitem{xing2014survey}
\leavevmode\vrule height 2pt depth -1.6pt width 23pt, {\em A survey of high
  order methods for the shallow water equations}, Journal of Mathematical
  Study, 47 (2014).

\bibitem{zhang2010maximum}
{\sc X.~Zhang and C.-W. Shu}, {\em On maximum-principle-satisfying high order
  schemes for scalar conservation laws}, Journal of Computational Physics, 229
  (2010), pp.~3091--3120.

\end{thebibliography}

\end{document}